\documentclass [11pt,oneside,a4paper,mathscr]{amsart}

\usepackage{amscd,amsmath,amssymb,euscript}
\usepackage[frame,cmtip,curve,arrow,matrix,line,graph]{xy}
\usepackage{tikz-cd}
\usepackage{hyperref}
\usepackage{stmaryrd}
\usepackage{bm}
\usepackage{here}
\usepackage{enumitem}
\title[Dolbeault Langlands for $\GL_2$ with reduced spectral curves]{A proof of Dolbeault geometric Langlands for $\GL_2$ with reduced spectral curves}
\author{Yukinobu Toda}

\newtheorem{thm}{Theorem}[section]
\newtheorem{cor}[thm]{Corollary}
\newtheorem{prop}[thm]{Proposition}

\newtheorem{conj}[thm]{Conjecture}
\newtheorem{lemma}[thm]{Lemma}

\theoremstyle{definition}
\newtheorem{defn}[thm]{Definition}

\newtheorem{thm*}[thm]{Theorem$^*$}
\newtheorem{remark}[thm]{Remark}

\newtheorem{example}[thm]{Example}

\newtheorem{assum}[thm]{Assumption}

\newtheorem{step}{Step}
\newtheorem{sstep}{Step}

\newcommand{\comment}[1]{}

\renewcommand{\leq}{\leqslant}
\renewcommand{\geq}{\geqslant}

\newcommand{\Y}{\mathcal{Y}}
\newcommand{\cS}{\mathcal{S}}
\newcommand{\zZ}{\mathcal{Z}}
\newcommand{\U}{\mathcal{U}}

\newcommand{\LL}{\operatorname{L}}

\newcommand{\QCoh}{\operatorname{QCoh}}

\newcommand{\X}{\mathcal{X}}

\newcommand{\cH}{\mathcal{H}}

\newcommand{\cL}{\mathcal{L}}
\newcommand{\cB}{\mathcal{B}}
\newcommand{\cC}{\mathcal{C}}

\newcommand{\cP}{\mathcal{P}}
\newcommand{\cE}{\mathcal{E}}
\newcommand{\cF}{\mathcal{F}}
\newcommand{\ev}{\mathrm{ev}}
\newcommand{\bgm}{B\mathbb{G}_m}

\newcommand{\diasquare}{\ar@{}[rd]|\square}
\newcommand{\cHecke}{\mathcal{H}ecke}
\newcommand{\cHom}{\mathcal{H}om}
\newcommand{\rank}{\operatorname{rank}}

\newcommand{\rB}{\mathrm{B}}

\newcommand{\Coh}{\operatorname{Coh}}
\newcommand{\Aut}{\operatorname{Aut}}

\newcommand{\Hilb}{\operatorname{Hilb}}
\newcommand{\Ker}{\operatorname{Ker}}
\newcommand{\id}{\operatorname{id}}
\newcommand{\Ext}{\operatorname{Ext}}
\newcommand{\Ind}{\operatorname{Ind}}
\newcommand{\IndCoh}{\operatorname{IndCoh}}
\newcommand{\IndL}{\operatorname{IndL}}
\newcommand{\Hom}{\operatorname{Hom}}

\newcommand{\Spec}{\operatorname{Spec}}

\newcommand{\GL}{\operatorname{GL}}

\renewcommand{\Im}{\operatorname{Im}}

\newcommand{\cl}{\mathrm{cl}}

\newcommand{\inclusion}{\ar@<-0.3ex>@{^{(}->}[r]}
\newcommand{\iinclusion}{\ar@<-0.3ex>@{^{(}->}[rr]}
\newcommand{\linclusion}{\ar@<0.3ex>@{_{(}->}[l]}
\newcommand{\dinclusion}{\ar@<-0.3ex>@{^{(}->}[d]}
\newcommand{\uinclusion}{\ar@<-0.3ex>@{^{(}->}[u]}

\newcommand{\ssslash}{/\!\!/}

\newcommand{\wt}{\mathrm{wt}}
\newcommand{\Hig}{\mathrm{Higgs}}
\newcommand{\Bun}{\mathrm{Bun}}

\makeatletter
\newcommand{\colim@}[2]{%
  \vtop{\m@th\ialign{##\cr
    \hfil$#1\operator@font colim$\hfil\cr
    \noalign{\nointerlineskip\kern1.5\ex@}#2\cr
    \noalign{\nointerlineskip\kern-\ex@}\cr}}%
}
\newcommand{\colim}{%
  \mathop{\mathpalette\colim@{}}\nmlimits@
}
\makeatother

\usetikzlibrary{positioning,fit,calc}
\tikzstyle{block}=[draw=black, width=1cm, minimum height=2cm, align=center] 
\tikzstyle{block2}=[draw=black, text width=2cm, minimum height=1cm, align=center] 
\tikzstyle{block3}=[draw=black, text width=2cm, minimum height=1cm, align=center] 

\makeatletter

\@addtoreset{equation}{section}	
\makeatother

\begin{document}

\begin{abstract}
In our previous paper with Tudor P\u{a}durariu, we introduced the notion of limit categories for moduli stacks of Higgs bundles and formulated the Dolbeault geometric Langlands correspondence. These limit categories are expected to provide an effective “classical limit” of the categories of 
D-modules on the moduli stack of bundles, and our formulation links categorical Donaldson–Thomas theory with the geometric Langlands correspondence.

In this paper, we prove the above Dolbeault geometric Langlands correspondence for $\GL_2$
 over the locus in the Hitchin base where the spectral curves are reduced. This is the first non-trivial case in which the relevant moduli stacks are not quasi-compact, and the use of limit categories is essential to the formulation and proof of the correspondence. 
 
 Our approach also outlines a strategy for proving the correspondence in greater generality and explains the current obstructions to such an extension.
\end{abstract}

\maketitle

 \setcounter{tocdepth}{2}
\tableofcontents

\section{Introduction}
\subsection{Main result}
Let $C$ be a smooth projective curve of genus $g$, and let $G$ be a reductive group with Langlands
dual ${}^{L}G$.
Let
\begin{align*}
h \colon \Hig_G \to \mathrm{B}
\end{align*}
be the derived moduli stack of $G$-Higgs bundles on $C$, with Hitchin fibration $h$.
For each connected component $\Hig_G(\chi)$ with $\chi \in \pi_1(G)$,
we have the quasi-compact open substack of semistable Higgs bundles
\begin{align*}
    \Hig_G(\chi)^{\mathrm{ss}} \subset \Hig_G(\chi).
\end{align*}
However, the stack $\Hig_G(\chi)$ of \emph{all} Higgs bundles is not quasi-compact unless $G$ is a torus.

In our previous paper with Tudor P\u{a}durariu~\cite{PTlim}, we introduced a dg-category (called a
\emph{limit category})
\begin{align}\label{intro:limit}
    \IndL_{\mathcal{N}}(\Hig_G(\chi))_w \subset \IndCoh_{\mathcal{N}}(\Hig_G(\chi))_w
\end{align}
for $\chi \in \pi_1(G)$ and $w\in Z_G^{\vee}$, where $Z_G$ is the center of $G$ and
$\mathcal{N}$ denotes the nilpotent singular support condition~\cite{AG}.
The category \eqref{intro:limit} is expected to provide an effective ``classical limit'' of the categories of
$\mathrm{D}$-modules on the moduli stack of $G$-bundles, and it overcomes several subtleties arising from
(quasi- or ind-)coherent sheaves on stacks that are not quasi-compact.

We then proposed a version of the \emph{Dolbeault geometric Langlands conjecture}
(with nilpotent singular support), which may be regarded as
a classical limit of the 
geometric Langlands correspondence, as follows:
\begin{conj}\emph{(\cite{PTlim})}\label{conj:intro}
    There is a $\rB$-linear equivalence
    \begin{align}\label{intro:IndL}
        \IndCoh_{\mathcal{N}}(\Hig_{{}^{L}G}(w)^{\mathrm{ss}})_{-\chi} \stackrel{\sim}{\to}
        \IndL_{\mathcal{N}}(\Hig_G(\chi))_w
    \end{align}
    which is compatible with Wilson/Hecke operators.
\end{conj}

So far, for $G=\GL_r$, Conjecture~\ref{conj:intro} is only known over the elliptic locus
$\rB^{\mathrm{ell}} \subset \rB$ by~\cite{Ardual}, where the spectral curves are integral (irreducible and reduced). 
Over the elliptic locus, we have
\begin{align}\label{intro:ssst}
    \Hig_G^{\mathrm{st}}\times_{\rB}\rB^{\mathrm{ell}}= \Hig_G^{\mathrm{ss}}\times_{\rB}\rB^{\mathrm{ell}}= \Hig_G\times_{\rB}\rB^{\mathrm{ell}}.
\end{align}
Here $(-)^{\mathrm{st}}$ stands for the stable part. In particular, $\Hig_G(\chi)$ is
quasi-compact over the elliptic locus, indeed it is ``almost'' a smooth symplectic variety, up to the trivial derived structure $k[-1]$ and $\mathbb{G}_m$-gerbe. 

Let 
$\mathrm{B}^{\mathrm{red}} \subset \mathrm{B}$
be the open subset corresponding to reduced (but not necessarily irreducible) spectral curves. 
The condition (\ref{intro:ssst}) essentially fails in this case; 
there are strictly semistable Higgs bundles in the left-hand side of (\ref{intro:IndL}), and the stack 
$\Hig_G(\chi)$ is not quasi-compact in the right-hand side of (\ref{intro:IndL})
because of infinitely many Harder–Narasimhan strata. 
In this paper, we prove the following result:
\begin{thm}\emph{(Theorem~\ref{thm:gl2})}\label{thm:intro}
For $G=\GL_2$, Conjecture~\ref{conj:intro} holds over the locus $\mathrm{B}^{\mathrm{red}} \subset \mathrm{B}$.
\end{thm}
The result of Theorem~\ref{thm:intro} is the first non-trivial case beyond the elliptic locus in which the condition~(\ref{intro:ssst}) essentially fails. In particular, the relevant moduli stacks are not quasi-compact, yet the framework still admits Hecke operators. The result of Theorem~\ref{thm:intro} and its proof show that the notion of limit category from~\cite{PTlim} is essential both for the formulation
and for the proof.

\begin{remark}\label{rmk:introdl}
(1) The equivalence~(\ref{intro:IndL}) is also expected to be compatible with parabolic induction functors, and hence with the semiorthogonal decompositions into quasi-BPS categories constructed in~\cite{PTlim}. One can also verify this compatibility in the setting of Theorem~\ref{thm:intro}, but we omit the proof since it is not needed for the proof of Theorem~\ref{thm:intro}. However, compatibility with parabolic induction is expected to be necessary beyond $\rB^{\mathrm{red}}$; see Subsection~\ref{subsec:toward}.

(2) In~\cite{PTlim}, we also proposed a version without nilpotent singular support:
\begin{align}\label{intro:renorm}
    \Phi \colon \IndCoh(\Hig_{{}^{L}G}(w)^{\mathrm{ss}})_{-\chi} \stackrel{\sim}{\to}
        \IndL(\Hig_G(\chi))_w,
\end{align}
which may be regarded as a classical limit of the renormalized geometric Langlands correspondence in~\cite[Remark~1.6.9]{GLC1}.
Our approach in this paper, together with compatibility with parabolic induction, should also prove~(\ref{intro:renorm}) over $\mathrm{B}^{\mathrm{red}}$ for $G=\GL_2$.
\end{remark}

\subsection{Summary of the ingredients of Theorem~\ref{thm:intro}}

An outline of the proof of Theorem~\ref{thm:intro} will be given in Subsection~\ref{subsec:outline}.
Here we summarize the main ingredients:

\vspace{2mm}

\noindent{\bf (i) Limit categories on non-quasi-compact Higgs stacks.}
Beyond the elliptic locus, the stack $\Hig_G(\chi)$ is not quasi-compact,
due to the presence of infinitely many Harder--Narasimhan strata.
The categories of (quasi- or ind-)coherent sheaves on such a stack are 
not well-behaved in general,
whereas the limit category introduced in~\cite{PTlim} is well-behaved (e.g.\ compactly generated) and admits Hecke operators,
making it the correct framework beyond the elliptic locus.

\vspace{2mm}

\noindent{\bf (ii) A left adjoint to the Hitchin section.}
Using the formalism of limit categories, we construct a left adjoint
$s_!$ to $s^*$ for the Hitchin section $s\colon \rB \to \Hig_G$.
The object $s_!\mathcal O_{\rB}$ plays the role of the ``classical limit'' of the vacuum Poincar\'e sheaf in the geometric Langlands correspondence~\cite{GLC1},
and it is crucial for the Whittaker normalization.

\vspace{2mm}

\noindent{\bf (iii) Fourier--Mukai transform via the (extended) Arinkin sheaf.}
We use Arinkin's Cohen--Macaulay extension of the Poincar\'e bundle~\cite{Ardual, MLi} (and its extension to the
generically regular locus) to construct the Fourier--Mukai functor
$\Phi^{\mathrm P}$ from the semistable locus to the generically regular locus.

\vspace{2mm}

\noindent{\bf (iv) Wilson/Hecke compatibility for $\Phi^{\mathrm P}$.}
We prove that $\Phi^{\mathrm P}$ intertwines Wilson and Hecke operators.
The argument uses an explicit description of the Arinkin sheaf via Haiman's isospectral Hilbert scheme~\cite{Ha}.

\vspace{2mm}

\noindent{\bf (v) Generation/conservativity and Whittaker normalization on $\rB^{\mathrm{red}}$.}
We exploit compact generation by iterated Wilson operators and, over the locus of reduced spectral curves, 
the conservativity of right adjoints of Hecke operators; together with a proof of the Whittaker
normalization for $\GL_2$, this yields Theorem~\ref{thm:intro}.

\subsection{History}
The Dolbeault geometric Langlands conjecture was proposed by Donagi--Pantev~\cite{DoPa}
as a classical limit of the geometric Langlands correspondence. Recall that the geometric Langlands correspondence is an equivalence 
\begin{align*}
    \IndCoh_{\mathcal{N}}(\mathrm{LocSys}_{^{L}G}) \simeq \text{D-mod}(\Bun_G)
\end{align*}
proposed in~\cite{BD0, AG} and proved in~\cite{GLC1, GLC2, GLC3, GLC4, GLC5}.
In its original form, the Dolbeault geometric Langlands conjecture is stated as an equivalence, 
see~\cite[Conjecture~1.3]{ZN} and also~\cite{KapWit} for a motivation from physics 
\begin{align}\label{intro:DP}
    \QCoh(\Hig_{{}^{L}G}) \simeq \QCoh(\Hig_G).
\end{align}

Note that we have two Hitchin fibrations over the common base $\rB$:
\begin{align}\notag
    \xymatrix{
    \Hig_{{}^{L}G} \ar[rd]_-{{}^{L}h} & & \Hig_G \ar[ld]^-{h} \\
    & \rB. &
    }
\end{align}
It was proved in~\cite{DoPa} that these fibrations are dual abelian fibrations over the locus
$\rB^{\mathrm{sm}} \subset \rB$ consisting of smooth spectral curves.
Consequently, the equivalence \eqref{intro:DP} over $\rB^{\mathrm{sm}}$ follows immediately from the classical
Fourier--Mukai transform~\cite{Mu1}.
Arinkin~\cite{Ardual} further extended this to the elliptic locus
$\rB^{\mathrm{ell}} \subset \rB$. 

However, beyond the elliptic locus $\rB^{\mathrm{ell}} \subset \rB$, essentially nothing is known about an
equivalence of the form \eqref{intro:DP}.
Indeed, outside $\rB^{\mathrm{ell}}$ the stack $\Hig_G$ is not quasi-compact, and the dg-categories of
quasi-coherent (or ind-coherent) sheaves on it are not well-behaved; for instance, they may fail to be compactly generated.
At present, it seems unrealistic to prove \eqref{intro:DP} beyond $\rB^{\mathrm{ell}}$ within this naive framework.
By contrast, it is proved in~\cite{PTlim} that the dg-category \eqref{intro:limit} is compactly generated,
even though $\Hig_G(\chi)$ is not quasi-compact.

The limit category \eqref{intro:limit} was introduced in~\cite{PTlim} precisely to circumvent these categorical
subtleties when working with (quasi- or ind-)coherent sheaves on non-quasi-compact stacks.
For $g\geq 2$, we have open immersions
\begin{align*}
    \rB^{\mathrm{sm}} \subsetneq \rB^{\mathrm{ell}} \subsetneq \rB^{\mathrm{red}} \subsetneq \rB.
\end{align*}
Theorem~\ref{thm:intro} gives the first instance in which the Dolbeault geometric Langlands correspondence
is proved beyond $\rB^{\mathrm{ell}}$, in a framework that is appropriate for non-quasi-compact stacks.

We also remark that Melo--Rapagnetta--Viviani~\cite{MRVF2}, building on ideas of Arinkin~\cite{Ardual}, proved a derived equivalence for the compactified Jacobians of reduced planar curves; see Theorem~\ref{thm:MRV}. In their setting, one considers compactified Jacobians for which semistability coincides with stability, and hence their result does not cover the statement of Theorem~\ref{thm:intro}. Moreover, the derived category of compactified Jacobians does not admit Hecke operators. Nevertheless, we use the result of~\cite{MRVF2} in the final step of the proof of Theorem~\ref{thm:intro}, namely in establishing the Whittaker normalization.

Our formulation in Conjecture~\ref{conj:intro} also reveals a strong connection between categorical
Donaldson--Thomas theory~\cite{T} and the geometric Langlands correspondence;
see~\cite[Section~1]{PTlim} for further discussion.
Moreover, the limit category \eqref{intro:limit} admits Hecke operators, and this turns out to be essential
in the proof of Theorem~\ref{thm:intro}. Also see~\cite[Section~1.4]{PTlim} for the relation to 
topological mirror symmetry by Hausel--Thaddeus~\cite{HauTha, HauICM}.

\subsection{Outline of the proof of Theorem~\ref{thm:intro}}\label{subsec:outline}
Here we outline the proof of Theorem~\ref{thm:intro}.
We first point out that the statement is trivial for $g\leq 1$;
for $g=0$ we have $\mathrm{B}^{\rm{red}}=\emptyset$ and for 
$g=1$ the reduced spectral curve $\mathcal{C}_b$
is a disjoint union of the elliptic curve $C$, so the classical 
Fourier--Mukai transform~\cite{Mu1} applies. In what follows, 
we assume that $g\geq 2$.
In Step 1 to Step 5 below, we assume $G=\GL_r$ with $r\geq 2$.

\begin{step}
Construction of the functor.
\end{step}
We first construct a functor
\begin{align}\label{intro:Phi}
    \Phi^{\mathrm P} \colon \Coh(\Hig_G(w)^{\mathrm{ss}})_{-\chi'} \to \Coh(\Hig_G(\chi)^{\mathrm{greg}})_{w'},
\end{align}
where $\chi'=\chi+(r^2-r)(g-1)$ and
\begin{align*}
    \Hig_G^{\mathrm{greg}} \subset \Hig_G
\end{align*}
is the open substack of generically regular Higgs bundles.
The functor \eqref{intro:Phi} is defined as a Fourier--Mukai transform with kernel given by
Arinkin's Cohen--Macaulay extension of the Poincar\'e line bundle on the locus of regular Higgs bundles;
see~\cite{Ardual, MLi}.

\begin{step}
Compatibility with Wilson/Hecke operators.
\end{step}
We next show that the functor \eqref{intro:Phi} is compatible with Wilson/Hecke operators over $\rB^{\mathrm{red}}$. 
Consider the following diagram 
\begin{align}\label{intro:WHcompati:!}
\xymatrix{
\Coh(S) \otimes \Coh(\Hig_G^{\mathrm{ss}}) \ar[r]^-{\mathrm W} \ar[d]_-{\id\otimes\Phi^{\mathrm P}} &
\Coh(\Hig_G^{\mathrm{ss}}) \ar[d]^-{\Phi^{\mathrm P}} \\
\Coh(S) \otimes \Coh(\Hig_G^{\mathrm{greg}}) \ar[r]^-{\mathrm H} &
\Coh(\Hig_G^{\mathrm{greg}}).
}
\end{align}
Here $S=\mathrm{Tot}(\Omega_C)$ is a (non-compact) Calabi--Yau surface.
The functor $\mathrm W$ is the \emph{Wilson operator}, defined via the universal Higgs bundle as a correspondence.
The functor $\mathrm H$ is the \emph{Hecke operator}, defined via the stack of Hecke correspondences.
Using an explicit construction of the Arinkin sheaf via
Haiman's isospectral Hilbert scheme~\cite{Ha}, we prove the following:

\begin{thm}\label{intro:prop:WH}\emph{(Theorem~\ref{prop:WH})}
The diagram (\ref{intro:WHcompati:!}) commutes over $\rB^{\mathrm{red}} \subset 
\rB.$
\end{thm}

\begin{step}
Classical limit of the vacuum Poincar\'e sheaf.
\end{step}
In this step we essentially use the limit category.
Let $s$ be the Hitchin section
\begin{align*}
    s \colon \rB \to \Hig_G
\end{align*}
defined by the universal spectral curve.
The map $s$ is not proper and its image is not closed outside the semistable locus; nevertheless, the theory of limit categories allows us to construct
a \emph{left adjoint} to $s^*$,
\begin{align*}
    s_{!} \colon \Coh(\rB) \to \LL(\Hig_G)
    = \IndL(\Hig_G)^{\mathrm{cp}},
\end{align*}
where $(-)^{\mathrm{cp}}$ denotes the full subcategory of compact objects.
The existence of $s_!$ is one of the key features of the limit category; such a left adjoint to $s^*$ does not exist
in the setting of (quasi- or ind-)coherent sheaves.

\begin{step}
Whittaker normalization.
\end{step}
We expect that, under the conjectural equivalence of Conjecture~\ref{conj:intro},
the structure sheaf $\mathcal O_{\Hig_{{}^{L}G}(w)^{\mathrm{ss}}}$ corresponds to
$s_{!}\mathcal O_{\rB}$ (up to a shift of weights). We refer to this expected identification as the
\emph{Whittaker normalization}.

We now restrict to the locus of reduced spectral curves. Fix an open subset
\begin{align*}
    \cB \subset \rB^{\mathrm{red, cl}},
\end{align*}
and write $\cH=\Hig_G\times_{\rB}\cB$, etc.
Then $\cH^{\mathrm{greg}}=\cH$ (since any Higgs bundle with reduced spectral curve is generically regular), so the functor \eqref{intro:Phi} induces a functor
\begin{align}\label{intro:Phi2}
    \Phi \colon \Coh(\cH(w)^{\mathrm{ss}})_{-\chi'} \to \Coh(\cH(\chi))_{w'}.
\end{align}
The following conjecture formulates the Whittaker normalization over reduced spectral curves:

\begin{conj}\emph{(Conjecture~\ref{conj:norm})}\label{conj:intro:w}
There is an isomorphism in $\LL_{\mathcal{N}}(\cH(\chi_0))_{w'}$, 
where $\chi_0=(r^2-r)(1-g)$.
\begin{align*}
    \Phi(\mathcal O_{\cH(w)^{\mathrm{ss}}}) \cong s_{!}\mathcal O_{\cB}.
\end{align*}
\end{conj}

\begin{step}
Equivalence under the Whittaker normalization.
\end{step}
Fix $c\in C$.
By applying the Wilson operator to $\mathcal O_{\cH(w)^{\mathrm{ss}}}$ $d$ times (with $r \mid (d+\chi)$)
and then taking the right adjoint, we construct a functor
\begin{align}\notag
    \mathrm W^{R}_{d,c} \colon \IndCoh_{\mathcal N}(\cH(w)^{\mathrm{ss}})_{-\chi'} \to \QCoh(S^{\times d}).
\end{align}
We show that the objects obtained by iterated Wilson transforms of $\mathcal O_{\cH(w)^{\mathrm{ss}}}$
compactly generate $\IndCoh_{\mathcal N}(\cH(w)^{\mathrm{ss}})$; this implies that $\mathrm W^{R}_{d,c}$ is conservative for $d\gg 0$.

Similarly, by applying the Hecke operator to $s_{!}\mathcal O_{\cB}$ $d$ times and taking the right adjoint,
we construct a functor
\begin{align}\notag
    \mathrm H^{R}_{d,c} \colon \IndL_{\mathcal N}(\cH(\chi))_{w'} \to \QCoh(S^{\times d}).
\end{align}
We also prove the conservativity of $\mathrm H^{R}_{d,c}$ under an additional assumption
(Assumption~\ref{assum:Bcirc}) on the singularities of the spectral curves corresponding to points of $\cB$,
which is automatically satisfied for $G=\GL_2$.

Then Theorem~\ref{intro:prop:WH}, together with Conjecture~\ref{conj:intro:w} and the compact generation statement above,
yields the following commutative diagram:
\begin{align}\notag
\xymatrix{
\IndCoh_{\mathcal N}(\cH(w)^{\mathrm{ss}})_{-\chi'} \ar[r]^-{\mathrm W^{R}_{d,c}} \ar@<-0.5ex>[d]_-{\Phi} &
\QCoh(S^{\times d}) \ar@{=}[d] \\
\IndL_{\mathcal N}(\cH(\chi))_{w'} \ar[r]^-{\mathrm H^{R}_{d,c}} \ar@<-0.5ex>[u]_-{\Phi^{R}} &
\QCoh(S^{\times d}).
}
\end{align}
Here $\Phi^{R}$ denotes the right adjoint of $\Phi$, which exists formally by the adjoint functor theorem.

Combining this diagram with the conservativity of the right adjoints of the Wilson and Hecke operators, we obtain:

\begin{thm}\label{thm:intro:Phi}\emph{(Theorem~\ref{thm:eq})}
Assume Conjecture~\ref{conj:intro:w}. Then the functor \eqref{intro:Phi2} induces a fully faithful functor
\begin{align}\label{intro:Phi3}
    \Phi \colon \IndCoh_{\mathcal N}(\cH(w)^{\mathrm{ss}})_{-\chi'} \hookrightarrow \IndL_{\mathcal N}(\cH(\chi))_{w'}.
\end{align}
Moreover, under Assumption~\ref{assum:Bcirc} (which is satisfied for $\GL_2$), the functor \eqref{intro:Phi3} is an equivalence.
\end{thm}

\begin{step}
Proof of the Whittaker normalization for $G=\GL_2$.
\end{step}
Since both sides of \eqref{intro:Phi3} are unchanged (up to equivalence) under the weight shift
$(\chi,w)\mapsto(\chi',w')$, in the case $G=\GL_2$ it remains to prove Conjecture~\ref{conj:intro:w}.
More precisely, we prove:

\begin{thm}\label{intro:thmW}\emph{(Proposition~\ref{prop:norm})}
Conjecture~\ref{conj:intro:w} holds for $G=\GL_2$.
\end{thm}

We prove this by constructing an explicit resolution of the Arinkin sheaf by vector bundles and evaluating their weights.
Along the way, we also show that the functor $\Phi$ takes the structure sheaf to the quasi-BPS category considered in~\cite{PTlim, PThiggs}. It reduces the Whittaker normalization to an analogous problem for a perturbed stability condition for which semistability coincides with stability, after passing to an
 \'{e}tale cover of the Hitchin base. Then we use the result of~\cite{MRVF2} to deduce Theorem~\ref{intro:thmW}.
At present, the argument of Theorem~\ref{intro:thmW} relies on special features of the case $G=\GL_2$.

\subsection{Toward a proof of Conjecture~\ref{conj:intro}}\label{subsec:toward}

Our proof in this paper suggests a strategy for a complete proof of
Conjecture~\ref{conj:intro} in the case $G=\GL_r$.

\vspace{5mm}

\noindent
(1) First, construct a functor
\begin{align}\label{funct:general}
    \Phi \colon \IndCoh_{\mathcal{N}}\bigl(\Hig_{^{L}G}(w)^{\mathrm{ss}}\bigr)_{-\chi'}
    \to \IndCoh\bigl(\Hig_{G}(\chi)\bigr)_{w'}
\end{align}
given by an extension of the Arinkin sheaf.

\noindent
(2) Show that $\Phi$ and its right adjoint $\Phi^{R}$ are compatible with
Wilson/Hecke operators, as well as with the parabolic induction functors
considered in~\cite{PTlim}.

\noindent
(3) Prove the Whittaker normalization
\[
\Phi\bigl(\mathcal{O}_{\Hig_{^{L}G}(w)^{\mathrm{ss}}}\bigr)
\cong s_{!}\mathcal{O}_{\mathrm{B}}.
\]

\noindent
(4) Show that the dg-category
$\IndCoh_{\mathcal{N}}\bigl(\Hig_{\GL_r}(w)^{\mathrm{ss}}\bigr)_{-\chi'}$
is compactly generated by iterated applications of Wilson operators to
$\mathcal{O}_{\Hig_G(w)^{\mathrm{ss}}}$, together with the essential images of
the parabolic induction functors.

\noindent
(5) Steps (1)--(4), together with an induction on $r$, imply that the functor
$\Phi$ in fact gives a fully faithful functor
\begin{align}\label{intro:strategy}
    \Phi \colon \IndCoh_{\mathcal{N}}\bigl(\Hig_{^{L}G}(w)^{\mathrm{ss}}\bigr)_{-\chi'}
    \to \IndL_{\mathcal{N}}\bigl(\Hig_{G}(\chi)\bigr)_{w'}.
\end{align}

\noindent
(6) Show that the dg-category
$\IndL_{\mathcal{N}}\bigl(\Hig_{\GL_r}(\chi)\bigr)_{w'}$
is compactly generated by iterated applications of Hecke operators to
$s_{!}\mathcal{O}_{\mathrm{B}}$, together with the essential images of the
parabolic induction functors. This would imply that the fully faithful functor
\eqref{intro:strategy} is essentially surjective, hence an equivalence.

\vspace{5mm}

Step~(1) may be the most difficult. In the case $G=\GL_2$ and $L$-twisted Higgs
bundles, a Cohen--Macaulay extension of the Arinkin sheaf is constructed
in~\cite{MLi}; however, this remains open in other cases. Note that in~\cite{GLC1}
the geometric Langlands functor is constructed by producing an action of
$\mathrm{QCoh}(\mathrm{LocSys}_{^{L}G})$ on $\text{D-mod}(\Bun_G)$ via the
geometric Satake correspondence. There has been recent progress on the classical limit of the geometric Satake correspondence~\cite{CW, LiangSatakeTypeA}; however, at this moment it seems extremely hard to use it to produce 
an action of $\QCoh(\Hig_{^{L}G}^{\mathrm{ss}})$ on 
$\IndL_{\mathcal{N}}(\Hig_G)$. 

Over the locus $\mathrm{B}^{\mathrm{red}}$, steps~(1), (2), (4) and (5) are already
established for any $G=\GL_r$, and in that setting one does not need to include parabolic
induction in~(4). The remaining obstructions are to prove~(3) and~(6), which we
established for $G=\GL_2$ in this paper.

\subsection{Acknowledgements}
The author is grateful to Tudor P\u{a}durariu for the joint work~\cite{PTlim}, many discussions related to the subject of this paper, and several comments on the first draft of this paper. In particular, the idea of using the Arinkin sheaf and Wilson/Hecke operators emerged from discussions with him.

A large part of this work was carried out while the author was visiting the Morningside Center of Mathematics, Chinese Academy of Sciences, as a Distinguished Visiting Professor from October 23 to November 24, 2025. The author would like to thank the Morningside Center of Mathematics for its hospitality.

The author is supported by the World Premier International Research Center Initiative (WPI Initiative), MEXT, Japan, the Inamori Research Institute for Science, and JSPS KAKENHI Grant Number JP24H00180.

    \subsection{Notation and Conventions}\label{subsec:notation0}
    \subsubsection*{Conventions on stacks}
We work over an algebraically closed field $k$ of characteristic $0$.
Unless stated otherwise, all (derived) stacks are locally Noetherian and quasi-separated over $k$.
For a stack $\X$ over $k$, we write $\X(k)$ for the set of $k$-valued points of $\X$.
We denote by $\X^{\mathrm{cl}}$ its classical truncation. 
A stack is called \textit{QCA} if it is quasi-compact with affine geometric stabilizer groups.

For a scheme $Y$ with an action of an algebraic group $G$, we write $Y/G$ for the associated quotient stack.
We also use the notion of good moduli spaces of stacks, see~\cite{MR3237451} (and also~\cite{ahlqvist2023good} for derived stacks).
For a point $y\in Y$, we denote by $\widehat{Y}_y:=\Spec \widehat{\mathcal{O}}_{Y, y}$, where $\widehat{\mathcal{O}}_{Y, y}$
is the formal completion at the maximal ideal. 

For a morphism of derived stacks $f\colon \X \to \Y$, let $\mathbb{L}_{f}$ and $\mathbb{T}_{f}$ denote the $f$-relative cotangent and tangent complexes. The morphism $f$ is called \textit{quasi-smooth} if $\mathbb{L}_{f}$ is perfect of cohomological amplitude $[-1, 1]$. 
When $\Y=\Spec k$, we write $\mathbb{L}_{\X}=\mathbb{L}_{f}$ and $\mathbb{T}_{\X}=\mathbb{T}_f$. 
A derived stack $\X$ over $k$ is 
\textit{quasi-smooth} if $\X \to \Spec k$ is quasi-smooth. 
For $m \ge 0$, its $m$-shifted cotangent stack is
\[
  \Omega_{\X}[m] := \Spec_{\X}\,\mathrm{Sym}\big(\mathbb{T}_{\X}[-m]\big) \longrightarrow \X .
\]
For a finite-dimensional $k$-vector space $V$ and $m \ge 0$, set
\[
  V[-m] := \Spec \mathrm{Sym}\big(V^{\vee}[m]\big),
\]
where $\mathrm{Sym}\big(V^{\vee}[m]\big)$ has zero differentials.

All the fiber products are derived fiber products, and we use the notation $\square$ for the 
Cartesian square. We also use the following \textit{reduced fiber product}:
For a morphism of stacks
$\X\to \mathcal{Z}\leftarrow \Y$, its \textit{reduced fiber product} is defined to be 
$(\X\times_{\mathcal{Z}} \Y)^{\mathrm{cl, red}}$, and we use the notation $\square^{\mathrm{red}}$ for the 
reduced Cartesian square.

\subsubsection{Notations for (ind and quasi)coherent sheaves}\label{subsec:notation}
For a dg-category $\cC$, denote by $\cC^{\mathrm{cp}} \subset \cC$ the subcategory of compact objects. The dg-category $\cC$ is called \textit{compactly generated} if $\cC=\Ind(\cC^{\mathrm{cp}})$, where $\Ind(-)$ is the ind-completion. 
A dg-functor $F \colon \cC_1 \to \cC_2$ is called \textit{continuous} 
if it commutes with taking direct sums. 

We use the notation of~\cite{MR3701352} for (ind and quasi)coherent sheaves. 
For a derived stack $\X$, we denote by 
$\Coh(\X)$ the dg-category of coherent sheaves,
by $\QCoh(\X)$ the dg-category of quasi-coherent sheaves,
and by $\IndCoh(\X)$ the dg-category of ind-coherent sheaves. 
The heart of a standard t-structure on $\Coh(\X)$ is denoted by 
$\Coh^{\heartsuit}(\X)$. 

For a morphism $f\colon \X \to \Y$, we use the standard $*$-pull-back and $*$-push-forward
\begin{align*}
    f^* \colon \QCoh(\Y) \to \QCoh(\X), \ f_{*} \colon \QCoh(\X) \to \QCoh(\Y).
\end{align*}
If $f$ is quasi-smooth (resp.\ proper), the above functors restrict to 
\begin{align*}
    f^* \colon \Coh(\Y) \to \Coh(\X), \ f_{*} \colon \Coh(\X) \to \Coh(\Y).
\end{align*}
If $f$ is quasi-smooth and proper, the functor $f_{*}$ admits a right adjoint $f^!$ and 
$f^*$ admits a left adjoint $f_{!}$
\begin{align*}
    f^!\colon \Coh(\mathcal{Y}) \to \Coh(\mathcal{X}), \ 
    f_{!} \colon \Coh(\mathcal{X}) \to \Coh(\mathcal{Y}).
\end{align*}
They are given by $f^!(-)=f^*(-)\otimes \omega_f[\dim f]$, 
$f_{!}(-)=f_{*}(-\otimes \omega_f[\dim f])$, where $\omega_f=\det \mathbb{L}_f$. 
For $A\in \Coh(\Y)$, we often use $A|_{\X}$ for the $*$-pull-back $f^*A$.
Sometimes for $A\in \Coh^{\heartsuit}(\Y)$, we may also write $A|_{\X}$ for the \textit{classical} pull-back $\cH^0(f^*A)\in \Coh^{\heartsuit}(\X)$.

For a closed substack $\zZ \subset \X$, we denote by $\QCoh_{\zZ}(\X)$ 
the subcategory of $\QCoh(\X)$ consisting of objects $A$ such that 
$A|_{\mathcal{X}\setminus \mathcal{Z}}=0$. 
We have the \textit{local cohomology functor}
\begin{align}\label{loc:coh}
    \Gamma_{\mathcal{Z}}(-) \colon \QCoh(\X) \to \QCoh_{\mathcal{Z}}(\X)
\end{align}
which gives a right adjoint of the inclusion $\QCoh_{\mathcal{Z}}(\X) \subset \QCoh(\X)$. 

For derived stacks $\mathcal{X}$ and $\mathcal{Y}$ such that either one of them is QCA, there is an equivalence, 
see~\cite[Corollary~4.3.4]{MR3037900}:
\begin{align}\notag
    \QCoh(\mathcal{X})\otimes \QCoh(\mathcal{Y}) \stackrel{\sim}{\to}\QCoh(\mathcal{X}\times \mathcal{Y}).
\end{align}
We use the notation $\Coh(\X) \otimes \Coh(\Y)$ for the subcategory of the left 
hand side, which corresponds to $\Coh(\X\times \Y)$ 
under the above equivalence. 
For $\cE\in \QCoh(\X\times \Y)$ and $A\in \QCoh(\Y)$, 
we often write $\cE\boxtimes A$ for $\cE\otimes p_{\Y}^* A$
where $p_{\Y}$ is the projection to $\Y$. 
If $\mathcal{X}$ is quasi-smooth, then 
there is also a duality equivalence: 
\begin{align}\notag
\mathbb{D}_{\X} \colon \Coh(\mathcal{X})^{\mathrm{op}} \stackrel{\sim}{\to}
\Coh(\mathcal{X}), \ (-)\mapsto \cHom(-, \mathcal{O}_{\X}).
\end{align}

\subsubsection{\texorpdfstring{Notations for $\mathbb{G}_m$-weights}{Notation on Gm-weights}}\label{subsec:Gmwt}
For a $\mathbb{G}_m$-gerbe $\X \to X$, there is an orthogonal weight 
decomposition 
\begin{align*}
    \Coh(\X)=\bigoplus_{w\in \mathbb{Z}}\Coh(\X)_w,
\end{align*}
where $\Coh(\X)_w$ is the subcategory of $\mathbb{G}_m$-weights $w$. 
For the Brauer class $\beta$ corresponding to the $\mathbb{G}_m$-gerbe $\X \to X$, 
we can identify $\Coh(\X)_w$ with the dg-category of $\beta^w$-twisted coherent sheaves 
$\Coh(X, \beta^w)$, see~\cite[Section~2.1.3]{MR2309155}.

For an object $A \in \Coh(\bgm)$,
let $A=\oplus_{w\in \mathbb{Z}} A_w$ 
be its weight decomposition. We write $A^{>0}=\oplus_{w>0} A_w$. 
We denote by $\wt(A)\subset \mathbb{Z}$ the set of $w\in \mathbb{Z}$ such that $A_w\neq 0$, 
and $\wt^{\mathrm{max}}(A)\in \mathbb{Z}$ is the maximal element in $\wt(A)$. 
We will use the following map 
\begin{align*}
c_1 \colon K(\bgm)_{\mathbb{Q}} \to H^2(\bgm, \mathbb{Q})=\mathbb{Q}, 
\end{align*}
which is given by $c_1(A)=\sum_w w \cdot \dim A_w$.  
\section{Preliminary}
In this section, we recall Higgs bundles, their moduli stacks, limit categories, 
and the formulation of Dolbeault geometric Langlands correspondence in~\cite{PTlim}. We refer to~\cite{Hitchin, Hitchin1987} for some basics of Higgs bundles, and~\cite{BeNaRa} for the spectral construction. 
We also recall Arinkin's construction of Cohen--Macaulay extension of Poincar\'e line bundle~\cite{Ardual}, 
and the associated Fourier--Mukai functor. 

\subsection{Moduli stacks of Higgs bundles}\label{subsec:higgs}
Let $C$ be a smooth projective curve over $k$ of genus $g\geq 2$. For $G=\GL_r$, a \textit{$G$-Higgs bundle}
consists of 
\begin{align}\label{def:higgs}
    (F, \theta), \ \theta \colon F\to F\otimes \Omega_C
\end{align}
where $F \to C$ is a rank $r$ vector bundle. 
We denote by 
\begin{align*}
    \Hig_{G}=\coprod_{\chi\in \mathbb{Z}} \Hig_{G}(\chi)
\end{align*}
the derived moduli stack of $G$-Higgs bundles, 
and each $\Hig_{G}(\chi)$ is the connected component 
corresponding to (\ref{def:higgs}) such that $\deg F=\chi$. 
It is a zero-shifted symplectic stack~\cite{PTVV}; indeed we have 
\begin{align*}
    \Hig_G=\Omega_{\Bun_G}
\end{align*}
where $\Bun_G$ is the moduli stack of $G$-bundles. 

A Higgs bundle (\ref{def:higgs}) is called \textit{regular} at $p\in C$ if 
$\theta|_{p} \colon F|_{p} \to F|_{p} \otimes \Omega_C|_{p} \cong F|_{p}$ is a regular 
matrix. It is \textit{regular} (resp. \textit{generically regular}) if $\theta|_{p}$ is regular for any 
$p\in C$ (resp. regular at the generic point $p\in C$). 
We denote by 
\begin{align*}
    \Hig_G(\chi)^{\mathrm{reg}} \subset \Hig_G(\chi)^{\mathrm{greg}} \subset \Hig_G(\chi)
\end{align*}
the open substack of regular Higgs bundles, generically regular Higgs bundles, respectively. 

We have the (derived) Hitchin map 
\begin{align*}
   h \colon \Hig_G \to \mathrm{B}=\bigoplus_{i=1}^r \Gamma(\Omega_C^i)=\mathrm{B}^{\mathrm{cl}}\times k[-1].
\end{align*}
Here $\rB^{\mathrm{cl}}=\oplus_{i=1}^r H^0(\Omega_C^i)$, and the 
$k[-1]$ factor comes from $H^1(\Omega_C)=k$. 
We have the following structure result of Hitchin map, 
which is mentioned in~\cite[Theorem~2.2.4]{BD0}: 
\begin{thm}\emph{(\cite{Hitchin, Faltings1993, Ginzburg2001})}\label{thm:higgs:basic}
We have the Cartesian square
\begin{align}\label{dia:HigB}
    \xymatrix{
\Hig_G^{\mathrm{cl}} \inclusion \ar[d] \diasquare & \Hig_G \ar[d] \\
\rB^{\mathrm{cl}} \inclusion & \rB.
    }
\end{align}
Moreover the left vertical arrow is flat and an lci morphism, 
in particular $\Hig_G^{\mathrm{cl}}$ is Cohen--Macaulay. 
Moreover 
$\dim \Hig_G^{\mathrm{cl}}=(2g-2)r^2+1$, 
$\dim \rB^{\mathrm{cl}}=(g-1)r^2+1$.
\end{thm}

Each point $b\in \rB^{\mathrm{cl}}$ corresponds to the 
spectral curve 
\begin{align*}
    \cC_b \subset S:=\mathrm{Tot}_C(\Omega_C).
\end{align*}
By the spectral construction, the Hitchin fiber 
$h^{-1}(b)$ corresponds to 
sheaves supported on $\cC_b$. 
Under the above correspondence, regular Higgs bundles correspond to 
line bundles on $\cC_b$, generically regular Higgs bundles correspond
to rank one torsion-free sheaves on $\cC_b$ which are line bundles at the generic points of $\cC_b$, respectively. 

We denote by $\mathrm{C} \to \rB$
the universal spectral curve. 
We have the following diagram 
\begin{align}\notag
\xymatrix{
\mathrm{C} \ar@{^{(}->}[r] \ar[rd]_-{\pi} & S \times \rB \ar[d] \ar[rd] & \\
& C\times \rB \ar[r] & \rB.
}
\end{align}
Here the left horizontal arrow is a closed immersion, and the other arrows are 
projections.
We have the universal Higgs bundle
\begin{align}\label{univ:F}
    (\mathrm{F}, \vartheta), \ \mathrm{F} \to C\times \Hig_G, \ 
    \vartheta \colon \mathrm{F} \to \mathrm{F} \boxtimes \Omega_C.
\end{align}
Here $\mathrm{F}$ is a rank $r$ vector bundle and $\vartheta$ is the universal 
Higgs field. By the spectral construction, it gives a universal sheaf
\begin{align}\label{univ:E0}
    \mathrm{E} \in \Coh(\mathrm{C}\times_{\rB} \Hig_G)
\end{align}
such that $\pi_{*}\mathrm{E}\cong \mathrm{F}$. By abuse of notation, 
we also write $\mathrm{E}$ for its push-forward to $S\times \Hig_G$. 

From the diagram (\ref{dia:HigB}), the stack $\Hig_G$ is almost classical; there is only 
 a factor of $k[-1]$. By the following lemma, it actually splits: 
\begin{lemma}\label{lem:reduced}
There is an equivalence 
\begin{align*}
    \Hig_G^{\cl}\times k[-1] \stackrel{\sim}{\to} \Hig_G.
\end{align*}
\end{lemma}
\begin{proof}
    The moduli stack $\Hig_G$ is equivalent to 
    \begin{align*}
        \Hig_G \simeq\Omega_{\Bun_G}=\Spec \mathrm{Sym}(\mathbb{T}_{\Bun_G}).
    \end{align*}
    Let $\pi^{\mathrm{rig}}$ be the $\mathbb{G}_m$-rigidification 
    \begin{align*}
    \pi^{\mathrm{rig}} \colon \Bun_G \to \Bun_G^{\mathrm{rig}},
    \end{align*}
    which is a $\mathbb{G}_m$-gerbe. We have an exact triangle 
    \begin{align}\label{seq:T}
        \mathbb{T}_{\pi^{\mathrm{rig}}} \to \mathbb{T}_{\Bun_G} \to (\pi^{\mathrm{rig}})^* \mathbb{T}_{\Bun_G^{\mathrm{rig}}}.
    \end{align}
    We show that the above sequence (\ref{seq:T}) splits. 
The first map of the above sequence is the natural map 
\begin{align*}
    \mathcal{O}_{\Bun_G}[1] \to p_{\Bun*}\cHom(\mathrm{F}, \mathrm{F})[1]
\end{align*}
where $\mathrm{F} \to C\times \Bun_G$ is the universal vector bundle
and $p_{\Bun} \colon C\times \Bun_G \to \Bun_G$ is the projection. 
The natural map 
\begin{align*}
    \mathcal{O}_{C\times \Bun_G} \to \cHom(\mathrm{F}, \mathrm{F})
\end{align*}
is a split injection, where the splitting is given by $(-) \to \mathrm{tr}(-)/r$.
Therefore we have
\begin{align*}
    \cHom(\mathrm{F}, \mathrm{F})\cong\mathcal{O}_{C\times \Bun_G} \oplus 
    \cHom(\mathrm{F}, \mathrm{F})_0
\end{align*}
where the second factor is the traceless part. 
It follows that 
\begin{align*}
  \mathbb{T}_{\Bun_G}&\cong p_{\Bun*}\mathcal{O}_{C\times \Bun_G}[1] \oplus 
    p_{\Bun*}\cHom(\mathrm{F}, \mathrm{F})_0[1] \\
    &\cong \mathcal{O}_{\Bun_G}[1]\oplus (H^1(\mathcal{O}_C) \otimes \mathcal{O}_{\Bun_G}) \oplus 
    p_{\Bun*}\cHom(\mathrm{F}, \mathrm{F})_0[1], 
\end{align*}
which implies that the sequence (\ref{seq:T}) splits. 

Then we have 
\begin{align*}
    \Omega_{\Bun_G}&\simeq\Spec \mathrm{Sym}((\pi^{\mathrm{rig}})^*\mathbb{T}_{\Bun_G^{\mathrm{rig}}}\oplus\mathcal{O}_{\Bun_G}[1])
    \simeq \Hig_G^{\mathrm{cl}}\times k[-1].
\end{align*}
\end{proof}

\subsection{Moduli stacks of semistable Higgs bundles}
A $\GL_r$-Higgs bundle $(F, \theta)$ is called \textit{(semi)stable} if 
for any sub-Higgs bundle $(F', \theta') \subset (F, \theta)$ 
we have 
\begin{align*}
    \frac{\deg F'}{\rank F'} <(\leq) \frac{\deg F}{\rank F}. 
\end{align*}
We denote by 
\begin{align*}
\Hig_G(\chi)^{\mathrm{st}} \subset 
    \Hig_G(\chi)^{\mathrm{ss}} \subset \Hig_G(\chi)
\end{align*}
the open substack of (semi)stable Higgs bundles. 
It is an Artin stack of finite type, in particular it is 
quasi-compact. 

\begin{remark}\label{rmk:semistable}
Let $\cC_b \subset S$ be the spectral curve of $(F, \theta)$, 
and $E\in \Coh^{\heartsuit}(\cC_b)$ the corresponding rank one torsion-free sheaf. Then $(F, \theta)$ is (semi)stable if and only if $E$ is Gieseker (semi)stable with respect to the polarization $\pi^*h$ where $\pi \colon \cC_b \to C$ is the projection and $h$ is an ample divisor on $C$, see~\cite{Hu} for the notion of Gieseker (semi)stability. 
Below, (semi)stable sheaf on $\Coh^{\heartsuit}(\cC_b)$ means Gieseker
(semi)stable sheaf with respect to $\pi^* h$. 
\end{remark}

We have the following diagram 
    \begin{align}\label{dia:gmoduli}
        \xymatrix{
        \Hig_G(\chi)^{\mathrm{st}} \inclusion \ar[d] \diasquare & \Hig_G(\chi)^{\mathrm{ss}} \ar[d] \ar[dr] & \\
        \mathrm{H}_G(\chi)^{\mathrm{st}} \inclusion &  \mathrm{H}_{G}(\chi)^{\mathrm{ss}} \ar[r] & \rB
        }
    \end{align}
    where the vertical arrows are good moduli space morphisms. 
    The derived scheme $\mathrm{H}_G(\chi)^{\mathrm{ss}}$
    is projective over $\rB$, the horizontal arrows are open 
    immersions, 
    and the left vertical arrow is a $\mathbb{G}_m$-gerbe. 
Moreover the classical truncation of $\mathrm{H}_G(\chi)^{\mathrm{st}}$
is a smooth symplectic variety. 
If $(r, \chi)$ is coprime, then 
we have $\mathrm{H}_G(\chi)^{\mathrm{st}}=\mathrm{H}_G(\chi)^{\mathrm{ss}}$, but 
in general the inclusion is strict. 

\subsection{Open locus of Hitchin base}
There are derived open subschemes 
\begin{align*}
    \mathrm{B}^{\mathrm{sm}} \subsetneq \mathrm{B}^{\mathrm{ell}} \subsetneq \mathrm{B}^{\mathrm{red}} \subsetneq \mathrm{B}
\end{align*}
where $\mathrm{B}^{\mathrm{red}}$ (resp. $\mathrm{B}^{\mathrm{ell}}, \mathrm{B}^{\mathrm{sm}}$) corresponds to reduced (resp. integral, smooth) spectral curves. 
The stack $\Hig_G(\chi)$ is quasi-compact over $\mathrm{B}^{\mathrm{ell}}$, but over 
$\mathrm{B}^{\mathrm{red}}$ it is not quasi-compact. 

\begin{example}\label{exam:g_2}
Let $G=\mathrm{GL}_2$ and $g=2$. 
For $b\in \mathrm{B}^{\mathrm{red}}$, 
suppose that $\cC_b$ is a reducible 
nodal curve
$\cC_b=C_1 \cup C_2$, where $C_1 \cap C_2$ consists of 
two points $x_1, x_2$. 
The stack 
$h_{\chi}^{-1}(b):=h^{-1}(b)\cap \Hig_{\GL_2}(\chi)$ contains an open substack 
\begin{align*}
    h_{\chi}^{-1}(b)^{\mathrm{reg}}=\coprod_{l_1+l_2=\chi+2}\cP^{l_1, l_2}
\end{align*}
where $\cP^{l_1, l_2}$ consists of line bundles $\mathcal{L}$ on $\cC_b$ such that $l_i=\deg \mathcal{L}|_{C_i}$. It is easy to see that 
\begin{align*}
\cP^{l_1, l_2}\cong \mathrm{Pic}^{l_1}(C_1)\times \mathrm{Pic}^{l_2}(C_2)\times \mathbb{G}_m \times \bgm.
\end{align*}
Here $\mathrm{Pic}^{l_i}(C_i)$ is the moduli space of degree $l_i$-line bundles on $C_i$, 
which is a $g$-dimensional abelian variety since $C_i \stackrel{\cong}{\to} C$. 

We have the following diagram 
\begin{align*}
\xymatrix{
& \widetilde{\cC}_{b, 1} \ar[rd]^-{\nu_1} & \\
\widetilde{\cC}_b \ar[rr]_-{\nu} 
\ar[ru] \ar[rd]  & & \cC_b. \\
& \widetilde{\cC}_{b, 2} \ar[ru]_-{\nu_2} & 
}
\end{align*}
Here $\nu$ is the normalization of $\cC_b$ and each $\nu_i$ is a partial normalization at $x_i$. 
We denote by 
$\cP_1^{l_1, l_2}, \ \cP_2^{l_1, l_2}, \ \cP_{12}^{l_1, l_2}$
the moduli stacks of line bundles on $\widetilde{\cC}_{b, 1}, \widetilde{\cC}_{b, 2}, \widetilde{\cC}_b$ respectively, 
whose restrictions to $C_i$ have degree $l_i$. 
We have isomorphisms 
\begin{align*}
    &\cP_i^{l_1, l_2} \cong \mathrm{Pic}^{l_1}(C_1)\times \mathrm{Pic}^{l_2}(C_2)\times \bgm, \\ &\cP_{12}^{l_1, l_2}\cong 
    \mathrm{Pic}^{l_1}(C_1)\times \mathrm{Pic}^{l_2}(C_2)\times \bgm\times \bgm.
\end{align*}
Then the closure of $\cP^{l_1, l_2}$ in $h^{-1}(b)$ is described as (cf.~\cite[Figure~1]{Viviani})
\begin{align*}
    \overline{\cP}^{l_1, l_2}=&\cP^{l_1, l_2}\cup \nu_{1*}\cP_1^{l_1, l_2-1}\cup \nu_{1*}\cP_1^{l_1-1, l_2}
    \cup \nu_{2*}\cP_2^{l_1, l_2-1}\cup \nu_{2*}\cP_2^{l_1-1, l_2} \\
    &\cup \nu_{*}\cP_{12}^{l_1-2, l_2} \cup \nu_{*}\cP_{12}^{l_1-1, l_2-1} \cup \nu_{*}\cP_{12}^{l_1, l_2-2}.
\end{align*}

The above situation is visualized as follows (here $\leadsto$ stands for the degeneration)
\begin{align*}
    \xymatrix{
& \cP^{1, -1} \ar@{~>}[ld] \ar@{~>}[d] \ar@{~>}[rd] \ar@{~>}[rrd] & &  &  \cP^{0, 0}  \ar@{~>}[lld] \ar@{~>}[ld] \ar@{~>}[d] \ar@{~>}[rd]&  \\
\cP_1^{1, -2}\ar@{~>}[rd] \ar@{~>}[rrd] & \cP_2^{1, -2} \ar@{~>}[d] \ar@{~>}[rd] & \cP_1^{0, -1}\ar@{~>}[d]\ar@{~>}[rd] & \cP_2^{0, -1} \ar@{~>}[ld] \ar@{~>}[d] & \cP_1^{-1, 0}\ar@{~>}[ld] \ar@{~>}[d]  & \cP_2^{-1, 0} \ar@{~>}[lld] \ar@{~>}[ld] & \\
 & \cP_{12}^{1, -3} & \cP_{12}^{0, -2}  & \cP_{12}^{-1, -1}  & \cP_{12}^{-2, 0} &  &
    }
\end{align*}

From the above picture, $h_{\chi}^{-1}(b)$ is obtained by gluing the closures of 
$\cP^{l_1, l_2}$, and it is connected. Because there are infinitely many pairs $(l_1, l_2)$ 
for a given $\chi$, the stack $h_{\chi}^{-1}(b)$ is not quasi-compact. 
\end{example}

\begin{remark}\label{rmk:dense}
    In general for $b\in \rB^{\mathrm{red}}$, the open 
    substack $h^{-1}(b)^{\mathrm{reg}} \subset h^{-1}(b)$ is dense. 
    This follows from the density of the 
Cartier-divisor locus in the Hilbert scheme of points on a reduced locally planar curve~\cite[Fact~2.4]{MRF3}. 
Indeed, any rank-one torsion-free sheaf is of the form $I_Z\otimes L$ for a line 
bundle $L$ on $\cC_b$ and a zero-dimensional subscheme $Z\subset \cC_b$, and then one deforms $Z$ inside $\cC_b$ to a subscheme in $\cC_b$ with invertible ideal sheaf. 
\end{remark}
\subsection{Limit categories}
Here we recall limit categories defined in~\cite{PTlim}, which itself is inspired and based on 
the works~\cite{SvdB, hls}. 
Let $\mathfrak{M}$ be a quasi-smooth derived stack which is of finite presentation over $k$, and 
its cotangent complex is self-dual
\begin{align}\label{L:cotangent}
    \mathbb{L}_{\mathfrak{M}} \simeq \mathbb{T}_{\mathfrak{M}}.
\end{align}
For $x \in \mathfrak{M}(k)$, 
we set 
\begin{align*}
T_x:=
\cH^{0}(\mathbb{T}_{\mathfrak{M}}|_{x}), \ 
\mathfrak{g}_x:=
\cH^{-1}(\mathbb{T}_{\mathfrak{M}}|_{x}).
\end{align*}
Note that $\mathfrak{g}_x$ is the Lie algebra of 
$G_x:=\mathrm{Aut}(x)$. 

Given a map 
\begin{align}\notag\nu \colon \bgm\to \mathfrak{M}, \ 
\mathrm{pt} \mapsto x
\end{align}
we have the corresponding cocharacter 
$\mathbb{G}_m \to G_x$. 
We let $\mathbb{G}_m$ act on $\mathfrak{g}_x, T_x$
through the above cocharacter, and regard them as elements of $K(\bgm)$. 
The map $\nu$ may not be quasi-smooth, but it extends to a quasi-smooth map
(unique up to 2-isomorphisms)
\begin{align}\label{nu:reg}
\nu^{\mathrm{reg}} \colon \mathfrak{g}_x^{\vee}[-1]/\mathbb{G}_m \to \mathfrak{M}
\end{align}
called \textit{regularization}, see~\cite[Construction~3.2.2]{HalpK32}, 
\cite[Construction-Definition~3.1]{PTlim}.
It satisfies that the map 
\begin{align*}
    \mathbb{L}_{\mathfrak{M}}|_{x} \to 
    \mathbb{L}_{\mathfrak{g}_x^{\vee}[-1]/\mathbb{G}_m}|_{0}=\mathfrak{g}_x[1]
\end{align*}
induces an isomorphism on $\cH^{-1}(-)$. 
We also denote by $\iota$ the natural 
morphism 
\begin{align*}
    \iota \colon \mathfrak{g}_x^{\vee}[-1]/\mathbb{G}_m \to 
    \bgm.
\end{align*}

Let $\delta \in \mathrm{Pic}(\mathfrak{M})_{\mathbb{R}}$. 
Recall the notations about $\mathbb{G}_m$-weights in Subsection~\ref{subsec:Gmwt}.
\begin{defn}\label{def:Lcat}(\cite[Definition~3.3]{PTlim})
When $\mathfrak{M}$ is QCA, the subcategory 
\begin{align}\label{def:L}
    \mathrm{L}(\mathfrak{M})_{\delta} \subset \Coh(\mathfrak{M})
\end{align}
is defined to consist of objects $\cE$ 
such that, for any map
$\nu \colon \bgm \to \mathfrak{M}$ 
with $\nu(0)=x$ and regularization (\ref{nu:reg}),
we have 
\begin{align}\label{wt:nu}
\wt(\iota_{\ast}\nu^{\mathrm{reg}\ast}(\cE))
\subset \left[\frac{1}{2} c_1 (T_x^{<0}),
\frac{1}{2} c_1 (T_x^{>0}) \right]+
\frac{1}{2}c_1(\mathfrak{g}_x)+c_1(\nu^*\delta).
\end{align}
\end{defn}

\begin{remark}\label{rmk:perfect}
The regularization (\ref{nu:reg}) is needed since the naive pull-back $\nu^* \cE$ is not necessarily
bounded and $\wt(\nu^{*} \cE)$ is not necessarily a finite set. 

On the other hand if 
$\cE$ is perfect then $\nu^* \cE$ is bounded and the above issue is avoided. 
In this case, the condition (\ref{wt:nu}) can be simplified as
\begin{align}\label{wt:cond2}
    \wt(\nu^* \cE) \subset \left[\frac{1}{2} c_1 (\nu^{\ast}\mathbb{L}_{\mathfrak{M}}^{<0}), 
    \frac{1}{2} c_1 (\nu^{\ast}\mathbb{L}_{\mathfrak{M}}^{>0})
    \right]+c_1(\nu^{\ast}\delta).
\end{align}
The proof is similar to~\cite[Lemma~3.7]{PTlim}, see Lemma~\ref{lem:perfect}. 
\end{remark}

For non-quasi-compact case, the limit category is defined as follows: 
\begin{defn}(\cite[Section~3.7]{PTlim})
For a general (not necessarily quasi-compact) quasi-smooth derived stack 
$\mathfrak{M}$ satisfying (\ref{L:cotangent}), we define 
\begin{align*}
    \IndL(\mathfrak{M})_{\delta}:=\lim_{\mathcal{U}\subset \mathfrak{M}} \Ind(\LL(\mathcal{U})_{\delta})
\end{align*}
where the limit is over all the QCA open substacks $\mathcal{U}\subset \mathfrak{M}$, and define 
\begin{align*}
    \LL(\mathfrak{M})_{\delta} \subset \IndL(\mathfrak{M})_{\delta}
\end{align*}
to be the subcategory of compact objects. 
For a smooth stack $\mathcal{X}$, we 
call $\LL(\Omega_{\X})_{\delta}$ the \textit{limit category}.
\end{defn}

Let $\Omega_{\mathfrak{M}}[-1]$ be the $(-1)$-shifted cotangent of $\mathfrak{M}$, and 
let 
\begin{align*}
    \mathcal{N} \subset \Omega_{\mathfrak{M}}[-1]
\end{align*}
be the nilpotent cone; a fiber of $\Omega_{\mathfrak{M}}[-1] \to \mathfrak{M}$
at $x \in \mathfrak{M}(k)$ is the Lie algebra $\mathfrak{g}_x$ of $\Aut(x)$ by the condition (\ref{L:cotangent}),
and the fiber of $\mathcal{N} \to \mathfrak{M}$ at $x$ consists of nilpotent elements of $\mathfrak{g}_x$. 
We also have the subcategories with nilpotent singular supports~\cite{AG}:
\begin{align*}
    \LL_{\mathcal{N}}(\mathfrak{M})_{\delta} \subset \LL(\mathfrak{M})_{\delta}, \ 
    \IndL_{\mathcal{N}}(\mathfrak{M})_{\delta} \subset \IndL(\mathfrak{M})_{\delta}.
\end{align*}

\begin{remark}\label{rmk:Ltilde}
Instead of compact objects, we can also consider
\begin{align*}
    \widetilde{\LL}_{\mathcal{N}}(\mathfrak{M})_{\delta}
    :=\lim_{\mathcal{U}\subset \mathfrak{M}}\LL_{\mathcal{N}}(\mathcal{U})_{\delta}, 
    \  \widetilde{\LL}(\mathfrak{M})_{\delta}
    :=\lim_{\mathcal{U}\subset \mathfrak{M}}\LL(\mathcal{U})_{\delta}.
\end{align*}
They contain $\LL_{\mathcal{N}}(\mathfrak{M})_{\delta}$, 
$\LL(\mathfrak{M})_{\delta}$ respectively, but in general 
they are strictly larger, see~\cite[Remark~3.18]{PTlim}.
\end{remark}

\subsection{Dolbeault geometric Langlands conjecture}\label{subsec:dlconj}
We recall the formulation of Dolbeault geometric Langlands conjecture 
proposed in~\cite{PTlim}, specialized to the case of $G=\GL_r$ and with 
nilpotent singular supports. 

For $G=\GL_r$, we have the decomposition 
\begin{align*}
    \IndCoh(\Hig_G)=\bigoplus_{w\in \mathbb{Z}} \IndCoh(\Hig_G)_{w}
\end{align*}
where each summand corresponds to objects with $\mathbb{G}_m$-weight $w$
for the diagonal torus 
$\mathbb{G}_m \subset \Aut(F, \theta)$. 
We have the subcategory 
\begin{align}\label{indl}
    \IndL_{\mathcal{N}}(\Hig_G(\chi))_w \subset \IndCoh(\Hig_G(\chi))_w
\end{align}
by setting $\delta$ to be 
\begin{align*}
    \delta=\det (\mathrm{F}_c)^{\otimes w/r} \in \mathrm{Pic}(\Hig_G)_{\mathbb{Q}}.
\end{align*}
Here $\mathrm{F}_c$ is the rank $r$-vector bundle
\begin{align*}
    \mathrm{F}_c:=\mathrm{F}|_{c\times \Hig_G} \to \Hig_G
\end{align*}
for a fixed $c\in C$ and $\mathrm{F}$ is the universal bundle (\ref{univ:F}). 
The subcategory (\ref{indl}) is independent of a choice of $c\in C$. 
\begin{thm}\emph{(\cite[Theorem~1.15]{PTlim})}\label{thm:cpt}
    The dg-category (\ref{indl}) is compactly generated by compact objects 
    \begin{align*}
        \LL_{\mathcal{N}}(\Hig_G(\chi))_w \subset \IndL_{\mathcal{N}}(\Hig_G(\chi))_w.
    \end{align*}
\end{thm}
\begin{remark}\label{rmk:cpt}
In~\cite[Theorem~1.15, Theorem~7.18]{PTlim}, the compact generation is stated 
for limit categories without nilpotent singular support condition. However using~\cite[Theorem~7.17]{PTlim}, the same argument of loc. cit. applies to show 
the compact generation with nilpotent singular supports.
\end{remark}

The following is a version of Dolbeault geometric Langlands conjecture in~\cite{PTlim}, 
specialized to $G=\GL_r$ and with nilpotent singular supports:
\begin{conj}\emph{(\cite[Conjecture~1.1]{PTlim})}\label{conj:dl}
For $G=\GL_r$ and $(\chi, w)\in \mathbb{Z}^2$, there is a $\rB$-linear equivalence 
\begin{align*}
    \IndCoh_{\mathcal{N}}(\Hig_G(w)^{\mathrm{ss}})_{-\chi} \simeq \IndL_{\mathcal{N}}(\Hig_G(\chi))_w.
\end{align*}
\end{conj}

\begin{remark}\label{rmk:equiv}
By Theorem~\ref{thm:cpt} and the compact generation of ind-coherent sheaves 
for QCA stacks~\cite{MR3037900} (also see~\cite[Corollary~9.2.7]{AG} in the case of nilpotent singular supports and global complete intersection type, which is applied to $\Hig_G(w)^{\mathrm{ss}}$), Conjecture~\ref{conj:dl} is equivalent to an equivalence
between compact objects
    \begin{align}\notag
    \Coh_{\mathcal{N}}(\Hig_G(w)^{\mathrm{ss}})_{-\chi} \simeq \LL_{\mathcal{N}}(\Hig_G(\chi))_w.
\end{align}
\end{remark}

\begin{remark}\label{rmk:period}
    There exist equivalences 
    \begin{align*}
        \IndL_{\mathcal{N}}(\Hig_G(\chi))_w \stackrel{\sim}{\to}
         \IndL_{\mathcal{N}}(\Hig_G(\chi))_{w+r}
    \end{align*}
    by taking $\otimes \det \mathrm{F}_c$. Also there is an equivalence 
    \begin{align*}
        \Hig_G(\chi) \stackrel{\sim}{\to} \Hig_G(\chi+r), \ (F, \theta) \mapsto (F, \theta)\otimes \mathcal{O}_C(c)
    \end{align*}
    which preserves the semistability. Therefore we can regard $(\chi, w)$
    as an element of $(\mathbb{Z}/r\mathbb{Z})^{2}$. 
\end{remark}

For a quasi-smooth derived scheme $\cB$ which is either 
an open subscheme in $\rB$ or $\rB^{\mathrm{cl}}$,  
we simply write 
\begin{align}\label{sH}
    \cH:=\Hig_G\times_{\rB} \cB.
\end{align}
We also write $\cH(\chi)$, $\cH(\chi)^{\mathrm{ss}}$ in a similar way. 
We say that 
\textit{DL holds over $\cB$} if there is a $\cB$-linear equivalence
\begin{align}\label{equiv:H}
    \IndCoh_{\mathcal{N}}(\cH(w)^{\mathrm{ss}})_{-\chi} \simeq \IndL_{\mathcal{N}}(\cH(\chi))_w.
\end{align}
Here see Subsection~\ref{subsec:variants} for the definition of the limit category when $\cB \subset \rB^{\mathrm{cl}}$ is an open subset.
\begin{remark}\label{rmk:Bcl}
For an open subscheme $\cB \subset \mathrm{B}$, suppose that
the DL holds over $\cB^{\mathrm{cl}} \subset \rB^{\mathrm{cl}}$. Then DL also 
holds over $\cB$.  
Indeed an equivalence (\ref{equiv:H}) for $\cB$ follows from 
that for $\cB^{\mathrm{cl}}$ by taking 
$\otimes \QCoh(k[-1])$ by Lemma~\ref{lem:reduced} and the diagram (\ref{dia:HigB}). 
\end{remark}

\subsection{Quasi-BPS categories}
By applying the definition of limit category to the 
quasi-compact moduli stack $\Hig_G(\chi)^{\mathrm{ss}}$, we obtain 
the \textit{quasi-BPS category} (with nilpotent singular support)
\begin{align*}
    \LL_{\mathcal{N}}(\Hig_G(\chi)^{\mathrm{ss}})_w \subset \Coh_{\mathcal{N}}(\Hig_G(\chi)^{\mathrm{ss}})_w.
\end{align*}
The quasi-BPS category is related to categorification of BPS invariants in 
Donaldson--Thomas theory, see~\cite{PThiggs, PThiggs2, PTlim}. 

We have the open immersion 
\begin{align*}
    j\colon \Hig_G(\chi)^{\mathrm{ss}}\subset \Hig_G(\chi)
\end{align*}
which induces the pull-back functor 
\begin{align}\label{funct:jast}
j^* \colon \LL_{\mathcal{N}}(\Hig_G(\chi))_w\to \LL_{\mathcal{N}}(\Hig_G(\chi)^{\mathrm{ss}})_{w}.
\end{align}
The following is one of the key properties of limit categories:
\begin{thm}\emph{(\cite[Theorem~7.19]{PTlim})}\label{thm:left}
The functor (\ref{funct:jast}) admits a fully-faithful left adjoint 
\begin{align*}
    j_{!} \colon \LL_{\mathcal{N}}(\Hig_G(\chi)^{\mathrm{ss}})_w\hookrightarrow\LL_{\mathcal{N}}(\Hig_G(\chi))_w.
\end{align*}
    
\end{thm}
\begin{remark}\label{rmk:jshrink}
From the proof of~\cite[Theorem~7.18, 7.19]{PTlim}, the image of $j_!$
is described as follows: for any map $\nu \colon \bgm \to \Hig_G(\chi)$
corresponding to a \textit{Harder--Narasimhan filtration}, the 
weight condition (\ref{wt:nu}) is strict on the left. 
Such a map $\nu$ is of the form 
\begin{align}\label{map:nupt}
    \nu\colon \mathrm{pt} \mapsto \bigoplus_{i=1}^a E_i, \ \frac{\deg E_1}{\rank E_1}>\cdots>\frac{\deg E_a}{\rank E_a}
\end{align}
where each $E_i$ is semistable, and 
with $\mathbb{G}_m$ weight $\mu_i$ on $E_i$ 
satisfying $\mu_1>\cdots>\mu_a$.

In particular for 
a perfect complex $\cE$ on $\mathfrak{M}=\Hig_G(\chi)$, 
it is in the image of $j_!$ if it satisfies the 
following stronger condition than (\ref{wt:cond2})
for any map $\nu \colon \bgm \to \mathfrak{M}$ as in (\ref{map:nupt}):
\begin{align}\notag
    \wt(\nu^* \cE) \subset \left(\frac{1}{2} c_1 (\nu^{\ast}\mathbb{L}_{\mathfrak{M}}^{<0}), 
    \frac{1}{2} c_1 (\nu^{\ast}\mathbb{L}_{\mathfrak{M}}^{>0})
    \right]+c_1(\nu^{\ast}\delta).
\end{align}
\end{remark}

\subsection{Some variants of limit categories}\label{subsec:variants}
We also use some variants of limit categories for derived stacks whose cotangent complexes are 
not self dual. First let $\mathfrak{M}$ be a quasi-smooth derived stack with $\mathbb{L}_{\mathfrak{M}} \simeq \mathbb{T}_{\mathfrak{M}}$. Suppose that it is written as 
\begin{align*}
    \mathfrak{M} \simeq \mathfrak{M}_{\circ}\times k[-1].
\end{align*}
In this case, for any $\nu \colon \bgm \to \mathfrak{M}_{\circ}$ with image $x\in \mathfrak{M}_{\circ}(k)$, 
we have the decomposition 
$\mathfrak{g}_x^{\vee}=(\mathfrak{g}_x)^{\vee}_0 \oplus k$
such that the regularization map of $\nu' \colon \bgm \stackrel{\nu}{\to} \mathfrak{M}_{\circ} \to \mathfrak{M}$ is of the form 
\begin{align*}
    (\nu')^{\mathrm{reg}}=\nu^{\mathrm{reg}} \times \id \colon 
    (\mathfrak{g}_x)^{\vee}_0[-1]/\mathbb{G}_m \times k[-1] \to \mathfrak{M}_{\circ}\times k[-1]. 
\end{align*}
When $\mathfrak{M}_{\circ}$ is QCA, the subcategory 
\begin{align*}
    \LL(\mathfrak{M}_{\circ})_{\delta} \subset \Coh(\mathfrak{M}_{\circ})
\end{align*}
is defined as in $\LL(\mathfrak{M})_{\delta}$ using the above definition of the regularization map $\nu^{\mathrm{reg}}$. 
In general, one defines 
\begin{align*}
    \IndL(\mathfrak{M}_{\circ})_{\delta}=\lim_{\mathfrak{U}_{\circ} \subset \mathfrak{M}_{\circ}}\IndL(\mathfrak{U}_{\circ})_{\delta}
\end{align*}
where the limit is over all QCA open substacks $\mathfrak{U}_{\circ}$. 

We also have the decomposition 
\begin{align*}
    \Omega_{\mathfrak{M}}[-1]=\Omega_{\mathfrak{M}_{\circ}}[-1] \times \mathbb{A}^1
\end{align*}
under which $\mathcal{N}$ is written as $\mathcal{N}=\mathcal{N}_{\circ}\times \{0\}$.
The subcategory 
\begin{align}\label{subcat:circ}
    \IndL_{\mathcal{N}}(\mathfrak{M}_{\circ})_{\delta} \subset \IndL(\mathfrak{M}_{\circ})_{\delta}
\end{align}
is defined to be the subcategory of objects with singular supports in $\mathcal{N}_{\circ}$. 
Since $k[-1]$ is a scheme, the factor $k[-1]$-factor does not affect the weight bounds from the maps from $\bgm$. Then the above subcategory satisfies that 
\begin{align*}
 \IndL_{\mathcal{N}}(\mathfrak{M}_{\circ})_{\delta} \otimes \mathrm{QCoh}(k[-1])\stackrel{\sim}{\to} 
    \IndL_{\mathcal{N}}(\mathfrak{M})_{\delta}.
\end{align*}

Suppose that there is a morphism 
$\mathfrak{M}_{\circ} \to \cB$
for a smooth scheme $\cB$. Let $T$ be another smooth scheme with 
an \'{e}tale morphism $T\to \cB$ and define 
\begin{align*}
    \mathfrak{M}_{\circ, T}:=\mathfrak{M}_{\circ}\times_{\cB}T.
\end{align*}
We use the same notation $\delta$ for its pull-back to $\mathfrak{M}_{\circ, T}$. 
For each $\nu \colon \bgm \to \mathfrak{M}_{\circ, T}$, let $\nu'$ be the composition 
$
    \nu' \colon \bgm \stackrel{\nu}{\to} \mathfrak{M}_{\circ, T} \to \mathfrak{M}_{\circ}
$
with image $x\in \mathfrak{M}_{\circ}(k)$.
Then the map 
\begin{align*}(\nu')^{\mathrm{reg}}
\colon (\mathfrak{g}_x^{\vee})_0[-1]/\mathbb{G}_m \to \mathfrak{M}_{\circ}
\end{align*}
factors through the quasi-smooth map 
\begin{align*}
    \nu^{\mathrm{reg}} \colon (\mathfrak{g}_x^{\vee})_0[-1]/\mathbb{G}_m \to \mathfrak{M}_{\circ, T}
\end{align*}
and the subcategory 
\begin{align*}
    \IndL(\mathfrak{M}_{\circ, T})_{\delta} \subset \IndCoh(\mathfrak{M}_{\circ, T})
\end{align*}
is defined in the same way as $\IndL(\mathfrak{M})_{\delta}$ 
using $\nu^{\mathrm{reg}}$.
The subcategory 
\begin{align}\label{ind:circ}
    \IndL_{\mathcal{N}}(\mathfrak{M}_{\circ, T})_{\delta} \subset \IndL(\mathfrak{M}_{\circ, T})
\end{align}
is defined to be the subcategory with singular support contained in 
\begin{align}\label{nsupport}
    \mathcal{N}_{\circ}\times_{\cB}T \subset \Omega_{\mathfrak{M}_{\circ, T}}[-1].
\end{align}
The subcategory 
\begin{align}\label{cpt:1}
    \LL(\mathfrak{M}_{\circ, T})_{\delta} \subset \IndL(\mathfrak{M}_{\circ, T})
\end{align}
is defined to be the subcategory of compact objects, and 
\begin{align}\label{cpt:2}\LL_{\mathcal{N}}(\mathfrak{M}_{\circ, T})_{\delta}
\subset \LL(\mathfrak{M}_{\circ, T})_{\delta}
\end{align}
is the subcategory with singular supports in (\ref{nsupport}). 

\begin{remark}\label{rmk:variant}
The above construction will be applied for $\mathfrak{M}=\Hig_{\GL_r}$
and $\mathfrak{M}_{\circ}=\Hig_{\GL_r}^{\mathrm{cl}}$ for $g\geq 2$.
Note that the argument of the proof of Theorem~\ref{thm:cpt} 
and 
Theorem~\ref{thm:left} are not affected 
by taking the above variants, i.e. removing the $k[-1]$-factor and \'{e}tale base change. In particular in this case, the categories in (\ref{ind:circ}) are 
compactly generated with compact objects (\ref{cpt:1}), (\ref{cpt:2}). Also 
for the above $\mathfrak{M}$
and the pull-back of the semistable locus $\mathfrak{M}_{\circ, T}^{\mathrm{ss}}$, 
there is a fully-faithful left adjoint 
\begin{align*}
    j_{T!} \colon \LL_{\mathcal{N}}(\mathfrak{M}_{\circ, T}^{\mathrm{ss}})_{\delta}
    \hookrightarrow \LL_{\mathcal{N}}(\mathfrak{M}_{\circ, T})_{\delta}
\end{align*}
for the open immersion $j_T \colon \mathfrak{M}_{\circ, T}^{\mathrm{ss}} \hookrightarrow \mathfrak{M}_{\circ, T}$, whose image is described 
as in Remark~\ref{rmk:jshrink}.
\end{remark}
\subsection{Arinkin sheaf}\label{subsec:Arsheaf}
Let $\cB \subset \rB^{\mathrm{cl}}$ be an open subset, and define $\cH$ as in (\ref{sH}). 
Consider 
\begin{align*}
    (\cH\times_{\cB}\cH)^{\sharp}:=
    (\cH\times_{\cB}\cH^{\mathrm{reg}}) \cup (\cH^{\mathrm{reg}}\times_{\cB}\cH).
\end{align*}
Then there is a line bundle 
\begin{align*}
    \cP^{\sharp} \to  (\cH\times_{\cB}\cH)^{\sharp}
\end{align*}
whose fiber at $(A_1, A_2)$ for $A_i \in \Coh^{\heartsuit}(\cC_b)$
for $b\in \cB$ is 
\begin{align}\label{def:Psharp}
     \cP^{\sharp}|_{(A_1, A_2)}=\det \chi(A_1\otimes A_2)\otimes \det \chi(A_1)^{-1}\otimes \det \chi(A_2)^{-1}\otimes \det \chi(\mathcal{O}_{\cC_b}).
\end{align}
The above formula makes sense since either $A_1$ or $A_2$ is a line bundle on $\cC_b$, 
so that $\chi(A_1\otimes A_2)$ makes sense. 

Formally let $\cC:=\mathrm{C}\times_{\rB} \cB$ and  
\begin{align*}
    \cE:=\mathrm{E}|_{\cC\times_{\cB}\cH} \in \Coh^{\heartsuit}(\cC\times_{\cB}\cH).
\end{align*}
Then on $\cH\times_{\cB}\cH^{\mathrm{reg}}$ we have 
\begin{align*}
    \cP^{\sharp}=\det p_{13*}(p_{12}^* \cE \otimes p_{23}^* \cE)
    \otimes (\det p_{13*}p_{12}^* \cE)^{-1} \otimes 
    (\det p_{13*}p_{23}^* \cE)^{-1} \otimes 
    \det p_{13*}\mathcal{O}.
\end{align*}
Here $p_{ij}$ are the projections from $\cH\times_{\cB} \cC \times_{\cB} \cH^{\mathrm{reg}}$
onto the corresponding factors. 

Arinkin~\cite{Ardual} constructed the maximal Cohen--Macaulay 
extension of $\cP^{\sharp}$ to $\cH\times_{\cB}\cH$
when $\cB=\rB^{\mathrm{ell, cl}}$. Outside the locus $\rB^{\mathrm{ell, cl}}$, i.e. when $\mathrm{B}^{\mathrm{ell, cl}} \subsetneq \cB \subset \rB^{\mathrm{cl}}$, it is observed in~\cite[Corollary~3.2.3]{MLi} (also see Remark~\ref{rmk:rk2}) that $\cP^{\sharp}$ can be extended 
to the generically regular Higgs locus, namely on 
\begin{align*}
    (\cH\times_{\cB}\cH)^{\flat}:=
    (\cH\times_{\cB}\cH^{\mathrm{greg}}) \cup (\cH^{\mathrm{greg}}\times_{\cB}\cH)
\end{align*}

\begin{thm}\emph{(\cite{Ardual, MLi})}\label{thm:Pflat}
The line bundle $\cP^{\sharp}$ extends to a maximal 
Cohen--Macaulay sheaf
\begin{align}\label{Arsheaf0}
    \cP^{\flat}\in \Coh^{\heartsuit}( (\cH\times_{\cB}\cH)^{\flat}).
\end{align}
Moreover it is flat over both factors of $\cH$. 
\end{thm}
We use the symbol $\cP$ to denote its restriction to 
$\cH^{\mathrm{ss}}\times_{\cB}\cH^{\mathrm{greg}}$;
\begin{align}\label{Arsheaf0.5}
    \cP:=\cP^{\flat}|_{\cH^{\mathrm{ss}}\times_{\cB}\cH^{\mathrm{greg}}}\in \Coh^{\heartsuit}(\cH^{\mathrm{ss}}\times_{\cB}\cH^{\mathrm{greg}}). 
\end{align}

\begin{remark}\label{rmk:rk2}
The main theorem of~\cite{MLi} is stated for rank-two \(L\)-twisted Higgs bundles with
\(\deg L>2g\), but in this paper we only use the construction of Section 3 of
\cite{MLi}.  That construction is the Hilbert-scheme realization of Arinkin's
Poincar\'e sheaf on the generically-regular locus.  It uses only the planar
spectral curve \(C_b\subset \operatorname{Tot}(L)\), the universal rank-one
torsion-free sheaf, and Haiman's isospectral Hilbert scheme.  Hence it applies
without change to \(\GL_r\) and to the untwisted case \(L=\Omega_C\).  The
rank-two hypothesis and the condition \(\deg L>2g\) in \cite{MLi} enter only
in the further extension to the non-generically-regular locus.
The Cohen--Macaulay extension for compactified Jacobians on reduced planar curve is also stated in~\cite[Fact~4.8]{MRVF2}.
\end{remark}
We will often use the following standard properties of Cohen--Macaulay
sheaves.  Although the references state them for schemes, the same
statements apply to classical Artin stacks locally of finite type: indeed,
the assertions are smooth-local and can be checked after
pulling back to a smooth atlas.

\begin{lemma}\emph{(\cite[Lemma~2.1]{Ardual})}\label{lem:CM-family}
Let $\mathcal{X}$ be a classical Artin stack locally of finite type and let
$Y$ be a scheme of pure dimension. Suppose that $Y$ is Cohen--Macaulay.
Let
\[
M\in \Coh^{\heartsuit}(\mathcal{X}\times Y)
\]
and suppose that, for every point $y\in Y$, the restriction
$
M|_{\mathcal{X}\times \{y\}}
$
is Cohen--Macaulay of a fixed codimension $d$. Then $M$ is
Cohen--Macaulay of codimension $d$.
\end{lemma}

\begin{lemma}\emph{(\cite[Lemma~2.2]{Ardual})}\label{lem:MCM-extension}
Let $\mathcal{X}$ be a classical Artin stack of pure dimension, and let
$
M\in \Coh^{\heartsuit}(\mathcal{X})
$
be a maximal Cohen--Macaulay sheaf on $\mathcal{X}$. Let
$\mathcal{Z}\subset \mathcal{X}$ be a closed substack of codimension at
least two, and let
$
j\colon \mathcal{X}\setminus \mathcal{Z}\hookrightarrow \mathcal{X}
$
be the open immersion. Then the natural morphism
$
M \to j^{\heartsuit}_{*}(M|_{\mathcal{X}\setminus \mathcal{Z}})
$
is an isomorphism, where
$
j^{\heartsuit}_{*}:=\mathcal{H}^0(j_*).
$
\end{lemma}

\subsection{Fourier--Mukai functor via Arinkin sheaf}
For $b\in \cB$ and $E\in \Coh^{\heartsuit}(\cC_b)$, we have 
$\deg \pi_{*}E=w$ if and only if 
\begin{align*}\deg_{\cC_b}(E)&:=\chi(E)-\chi(\mathcal{O}_{\cC_b}) \\
& =w+(r^2-r)(g-1).
\end{align*}
Here $\deg_{\cC_b}$ is the degree as a sheaf on $\cC_b$. 
By the description (\ref{def:Psharp}),
the Arinkin sheaf (\ref{Arsheaf0.5}) on 
$\cH(w)\times_{\cB}\cH(\chi)$ has bi-weight 
$(\chi', w')$, 
\begin{align*}
    \chi':=\chi+(r^2-r)(g-1), \ w':=w+(r^2-r)(g-1).
\end{align*}
Therefore it induces the $\cB$-linear functor 
\begin{align}\label{induce:Phi0}
\Phi \colon 
    \QCoh(\cH(w)^{\mathrm{ss}})_{-\chi'} \to \QCoh(\cH(\chi)^{\mathrm{greg}})_{w'}, 
\end{align}
defined by the Fourier--Mukai functor
\begin{align*}
    \Phi(-)=p_{2*}(p_1^{*}(-)\otimes \cP).
\end{align*}
The functor (\ref{induce:Phi0}) restricts to the functor 
\begin{align}\label{induce:Phi}
    \Phi \colon   \Coh(\cH(w)^{\mathrm{ss}})_{-\chi'} \to \Coh(\cH(\chi)^{\mathrm{greg}})_{w'}
\end{align}
since $\cH(w)^{\mathrm{ss}}$ admits a good moduli space which is projective over $\cB$
and $\cP$ is flat over $\cH$. 

\begin{remark}\label{rmk:reduced}
If $\cB\subset \rB^{\mathrm{red, cl}}$, then 
$\cH=\cH^{\mathrm{greg}}$ since any rank one torsion-free sheaf on a reduced curve is generically a line bundle. Therefore we have the maximal 
Cohen--Macaulay extension $\cP$
on $\cH\times_{\cB}\cH$. 
In particular in this case, the Arinkin sheaf induces the functor 
\begin{align}\notag
    \Phi \colon   \Coh(\cH(w)^{\mathrm{ss}})_{-\chi'} \to \Coh(\cH(\chi))_{w'}.
\end{align}
\end{remark}

\begin{lemma}\label{lem:commute}
The functor $\Phi$ in 
(\ref{induce:Phi0}) fits into the commutative diagram 
\begin{align}\label{dia:comfit}
\xymatrix{
\Coh(\cH(w)^{\mathrm{ss}})_{-\chi'} \ar[r]^-{\Phi} \ar[d]^-{\sim} & \Coh(\cH(\chi)^{\mathrm{greg}})_{w'} \ar[d]^-{\sim} \\
\Coh(\cH(w+r)^{\mathrm{ss}})_{-\chi'} \ar[r]^-{\Phi}  & \Coh(\cH(\chi)^{\mathrm{greg}})_{w'+r}
}
\end{align}
    Here the left vertical arrow is induced by 
    the isomorphism $\cH(w)^{\mathrm{ss}} \stackrel{\cong}{\to} \cH(w+r)^{\mathrm{ss}}$
    given by $\otimes \mathcal{O}_C(c)$ for $c\in C$ and the right vertical arrow is 
    given by $\otimes \det \mathcal{F}_c$, where $\mathcal{F}=\mathrm{F}|_{C\times \cH}$, 
    see Remark~\ref{rmk:period}. 
    \end{lemma}
    \begin{proof}
    The commutativity of (\ref{dia:comfit}) easily follows by comparing the kernel objects 
    of both compositions on $(\cH \times_{\cB}\cH)^{\sharp}$, using the formula 
    for $\cP^{\sharp}$, and the uniqueness of Cohen--Macaulay extension
    in Lemma~\ref{lem:MCM-extension}. 
\end{proof}

By applying the construction (\ref{Arsheaf0.5}) for $\cB=\rB^{\mathrm{cl}}$, we obtain 
\begin{align*}
    \cP\in \Coh^{\heartsuit}((\Hig_G^{\mathrm{ss}}\times_{\rB} \Hig_G^{\mathrm{greg}})^{\mathrm{cl}}). 
\end{align*}
By Lemma~\ref{lem:reduced}, we have 
\begin{align*}
    \Hig_G^{\mathrm{ss}}\times_{\rB} \Hig_G^{\mathrm{greg}}\simeq 
    (\Hig_G^{\mathrm{ss}}\times_{\rB} \Hig_G^{\mathrm{greg}})^{\mathrm{cl}} \times k[-1].
\end{align*}
We define 
\begin{align}\label{ArsheafP}
    \mathrm{P}:=\cP\boxtimes \mathcal{O}_{k[-1]}\in \Coh(\Hig_G^{\mathrm{ss}}\times_{\rB} \Hig_G^{\mathrm{greg}}). 
\end{align}
Note that $\mathrm{P}$ is not in the heart of the standard t-structure. 
We have the $\rB$-relative Fourier--Mukai transform 
\begin{align*}
    \Phi^{\mathrm{P}} \colon \QCoh(\Hig_G^{\mathrm{ss}}) \to \QCoh(\Hig_G^{\mathrm{greg}}), \ 
      (-)\mapsto  p_{2*}(p_1^*(-)\otimes \mathrm{P}).
\end{align*}
Here $p_i$ are the projections onto the corresponding factors. 
Similarly to the above, it restricts to the functor 
\begin{align}\label{Phi:P}
    \Phi^{\mathrm{P}} \colon \Coh(\Hig_G^{\mathrm{ss}}) \to \Coh(\Hig_G^{\mathrm{greg}}).
\end{align}

\subsection{Fourier--Mukai transform of compactified Jacobians of reduced curves}
Let $X$ be a reduced projective curve with at worst planar singularities 
(e.g. $X$ is the spectral curve $X=\cC_b$ for $b\in \mathrm{B}^{\mathrm{red}}$). 
Here we recall Fourier--Mukai transform of compactified Jacobians of $X$ 
via Arinkin sheaf, proved in~\cite{MRVF2}. Below we use the notation in~\cite[Section~1.1]{MRVF2}. 

Let $\{C_i\}_{i\in I}$ be the irreducible components of $X$. A \textit{polarization} is a tuple of 
rational numbers $\underline{q}=\{\underline{q}_{C_i}\}_{i\in I}$ for each $C_i \subset X$
such that $\lvert \underline{q} \rvert:=\sum_{i\in I}\underline{q}_{C_i} \in \mathbb{Z}$.
A torsion-free rank one sheaf $I\in \Coh^{\heartsuit}(X)$ of Euler characteristic $\chi(I)=\lvert \underline{q} \rvert$ 
is called \textit{$\underline{q}$-semistable} (resp.\textit{~stable})
if for any subcurve $Y\subsetneq X$ we have 
\begin{align*}
    \chi(I_Y) \geq  \sum_{C_i \subset Y}\underline{q}_{C_i}, \quad 
    \left(\mbox{resp. }  \chi(I_Y) >  \sum_{C_i \subset Y}\underline{q}_{C_i}\right).
\end{align*}
Here $I_Y$ is the maximal torsion-free quotient of $I|_{Y}$. 
A polarization $\underline{q}$ is called \textit{general} if there are no strictly $\underline{q}$-semistable sheaves. 
A \textit{fine compactified Jacobian} of $X$ is the fine moduli space $\overline{J}_X(\underline{q})$
of torsion-free rank one sheaves of Euler characteristic $\lvert \underline{q}\rvert$ which are $\underline{q}$-semistable (or equivalently $\underline{q}$-stable) with respect to a general polarization $\underline{q}$ on $X$. 

A fine compactified Jacobian is a good moduli space of the moduli stack of 
rank one semistable sheaves $\overline{\mathcal{J}}_X(\underline{q})$. 
The good moduli space morphism 
\begin{align*}
    \overline{\mathcal{J}}_X(\underline{q}) \to \overline{J}_X(\underline{q})
\end{align*}
is a $\mathbb{G}_m$-gerbe. Indeed it is a trivial $\mathbb{G}_m$-gerbe 
because of the existence of universal sheaf on $X\times \overline{J}_X(\underline{q})$, 
see~\cite[Section~2.1]{MRVF2}. 

The construction of the Arinkin sheaf in the previous subsection yields the 
Cohen--Macaulay sheaf 
\begin{align*}
    \cP_{\mathcal{J}} \in \Coh^{\heartsuit}( \overline{\mathcal{J}}_X(\underline{q}) \times 
     \overline{\mathcal{J}}_X(\underline{q}')) 
\end{align*}
of bi-weight $(\lvert \underline{q}'\rvert+p_a-1,  \lvert \underline{q}\rvert+p_a-1)$, 
where $p_a$ is the arithmetic genus of $X$.

\begin{thm}\label{thm:MRV}\emph{(\cite[Theorem~A]{MRVF2})}
The Fourier--Mukai functor 
   \begin{align}\label{Phi:J}
\Phi_{\mathcal{J}} \colon \Coh( \overline{\mathcal{J}}_X(\underline{q}))_{-\lvert \underline{q}'\rvert-p_a+1} \to \Coh( \overline{\mathcal{J}}_X(\underline{q}'))_{\lvert \underline{q}\rvert+p_a-1}
\end{align} 
with kernel $\cP_{\mathcal{J}}$ is an equivalence. Moreover the inverse of 
$\Phi_{\mathcal{J}}$ is given by the kernel object 
$\mathbb{D}(\cP_{\mathcal{J}})[p_a]$. 
\end{thm}

\begin{remark}\label{rmk:MRV2}
In~\cite[Theorem~A]{MRVF2}, the result is stated for $\lvert \underline{q}\rvert=\lvert \underline{q}'\rvert=1-p_a$; in this case the functor (\ref{Phi:J}) is canonically identified 
with 
\begin{align*}
    \Phi_J \colon \Coh(\overline{J}_X(\underline{q})) \stackrel{\sim}{\to} \Coh(\overline{J}_X(\underline{q}')). 
\end{align*}

In other cases, the functor (\ref{Phi:J}) is canonically identified with 
   \begin{align}\label{PhiJ_2}
     \Phi_J \colon \Coh(\overline{J}_X(\underline{q}), {\beta}^{-\lvert \underline{q}'\rvert-p_a+1}) \to \Coh(\overline{J}_X(\underline{q}'), {\beta'}^{\lvert \underline{q}\rvert+p_a-1}).    
   \end{align} 
   Here $\beta, \beta'$ are the Brauer classes corresponding to $\mathbb{G}_m$-gerbes 
   $\overline{\mathcal{J}}_X(\underline{q}), \overline{\mathcal{J}}_X(\underline{q}')$
   respectively. Since $\beta, \beta'$ are trivial classes, 
   after choosing trivializations of both gerbes, the functor is \textit{uncanonically} identified with 
   \begin{align*}
       \Coh(\overline{J}_X(\underline{q})) \to \Coh(\overline{J}_X(\underline{q}')). 
   \end{align*}
   
   The source of the uncanonicity is a choice of trivializations of $\beta, \beta'$; 
   as stated just below~\cite[Theorem~A]{MRVF2}, the above non-canonical functor 
   is still an equivalence, which implies that Theorem~\ref{thm:MRV} holds without
   the condition $\lvert \underline{q}\rvert=\lvert \underline{q}'\rvert=1-p_a$. 
   \end{remark}

   \begin{remark}\label{rmk:gieseker}
   The notion of $\underline{q}$-stability includes Gieseker semistable sheaves. 
   Indeed for a $\mathbb{Q}$-ample divisor $H$ on $X$, and $E\in \Coh^{\heartsuit}(X)$, 
   let $\underline{q}_{C_i}=c\cdot \deg(H|_{C_i})$, where $c\in \mathbb{Q}$ is determined by 
   $c\cdot \sum_{i\in I} \deg(H|_{C_i})=\chi(E)$. 
   Then $E$ is $H$-Gieseker (semi)stable if and only if it is $\underline{q}$-(semi)stable. 
   
       In particular if we denote by $\mathcal{M}_X(\chi')^{H\text{-ss}}$ the moduli stack of $H$-semistable 
       sheaves $E$ on $X$ with degree $\chi'$, the result of Theorem~\ref{thm:MRV}
       yields an equivalence for generic $H$ (i.e. when $H$-semistable equals $H$-stable)
       \begin{align}\label{equiv:M}
           \Coh(\mathcal{M}_X(w')^{H\text{-ss}})_{-\chi'} \stackrel{\sim}{\to}   \Coh(\mathcal{M}_X(\chi')^{H\text{-ss}})_{w'}.
       \end{align}
   \end{remark}

   \subsection{Relative Fourier--Mukai transform}
We will use a relative version of the equivalence (\ref{equiv:M}), which we formulate here. 

Let $T$ be a $k$-scheme of finite type, and let $X\to T$ be a flat family of reduced projective curves with at worst planar singularities 
(e.g.\ $X\to T$ is the universal spectral curve $\cC\to \cB$ for $\cB \subset \rB^{\mathrm{red,cl}}$).
For a $T$-relative $\mathbb{Q}$-ample divisor $H$ on $X$ and $\chi \in \mathbb{Z}$, let 
\begin{align}\label{map:T}
\mathcal{M}_{X/T}(\chi')^{H\text{-ss}} \to T
\end{align}
be the moduli stack of rank-one torsion-free $H$-semistable sheaves on the fibers of $X\to T$ of degree $\chi'$. 
Suppose that $H$ is generic, i.e.\ that $\mathcal{M}_{X/T}(\chi')^{H\text{-ss}}$
consists of $H$-stable sheaves. 
Then the good moduli space morphism
\begin{align*}
    \mathcal{M}_{X/T}(\chi')^{H\text{-ss}} \to M_{X/T}(\chi')^{H\text{-ss}}
\end{align*}
is a $\mathbb{G}_m$-gerbe with Brauer class $\beta$, and we have the decomposition into $\mathbb{G}_m$-weight subcategories
\begin{align*}
    \Coh(\mathcal{M}_{X/T}(\chi')^{H\text{-ss}})
    =
    \bigoplus_{w\in \mathbb{Z}}
    \Coh(\mathcal{M}_{X/T}(\chi')^{H\text{-ss}})_w,
\end{align*}
where each summand is equivalent to the category of $\beta^w$-twisted coherent sheaves on $M_{X/T}(\chi')^{H\text{-ss}}$.

The construction of the Arinkin sheaf in the previous subsection yields a Cohen--Macaulay sheaf
\begin{align*}
    \cP_{X/T} \in
    \Coh^{\heartsuit}\bigl(
   \mathcal{M}_{X/T}(w')^{H\text{-ss}}
    \times_T
    \mathcal{M}_{X/T}(\chi')^{H\text{-ss}}
    \bigr)
\end{align*}
of bi-weight $(\chi',w')$.

\begin{prop}\label{prop:MRV}
The Fourier--Mukai functor
\begin{align}\notag
\Phi_{X/T} \colon
\Coh( \mathcal{M}_{X/T}(w')^{H\text{-ss}})_{-\chi'}
\to
\Coh(\mathcal{M}_{X/T}(\chi')^{H\text{-ss}})_{w'}
\end{align}
with kernel $\cP_{X/T}$ is an equivalence. Moreover, the inverse of 
$\Phi_{X/T}$ is given by the kernel object
\begin{align}\label{kernel:inv}
\mathbb{D}(\cP_{X/T})\otimes p_1^*\omega_{M/T}[p_a],
\end{align}
where $p_a$ is the arithmetic genus of the fibers of $X\to T$, and $M=M_{X/T}(w')^{H\text{-ss}}$.
\end{prop}
\begin{proof}
    The case where $T=\Spec k$ is~\cite[Theorem~A]{MRVF2}, and we can reduce to this case. Indeed, the morphism \eqref{map:T} is proper and locally complete intersection by~\cite[Theorem~A]{MRF3}. Therefore the functor $\Phi_{X/T}$ admits a right adjoint
    \begin{align*}
        \Phi_{X/T}^R \colon
        \Coh(\mathcal{M}_{X/T}(\chi')^{H\text{-ss}})_{w'}
        \to
        \Coh( \mathcal{M}_{X/T}(w')^{H\text{-ss}})_{-\chi'}
    \end{align*}
    given by the Fourier--Mukai functor with kernel object (\ref{kernel:inv}). 
    The natural transformation
    \begin{align}\label{nat:PhiJ}
    \id \Rightarrow \Phi_{X/T}^R \circ \Phi_{X/T}
    \end{align}
    is compatible with base change, and its cone is zero on the fiber over each $t\in T$ by the result for $T=\Spec k$. Therefore \eqref{nat:PhiJ} is an isomorphism. 
    Similarly, $\Phi_{X/T} \circ \Phi_{X/T}^R \Rightarrow \id$ is an isomorphism. Hence $\Phi_{X/T}$ is an equivalence.
\end{proof}

\begin{remark}\label{rmk:dualizing}
    The relative dualizing sheaf of $M_{X/T}(w')^{H\text{-ss}}$ is trivial 
    on the fiber over any $t\in T$ by~\cite[Theorem~A]{MRF3}. 
    Therefore $\omega_{M/T}$ in Proposition~\ref{prop:MRV} is locally trivial on $T$.
\end{remark}

\section{Compatibility of Wilson/Hecke operators}
In this section, we prove the compatibility of Wilson/Hecke operators under 
the Fourier--Mukai functor given by the Arinkin sheaf. Throughout this section, 
$C$ is a smooth projective curve of genus $g\geq 2$ over 
$k$, and $G=\GL_r$.

\subsection{Wilson operator}
Recall from (\ref{univ:E0}) that we have the universal sheaf 
\begin{align*}
    \mathrm{E} \in \Coh(S\times \Hig_G^{\mathrm{ss}}).
\end{align*}
It is a perfect complex, and 
is a push-forward by the closed immersion 
\begin{align}\label{support:C}
\mathrm{C}\times_{\rB} \Hig_G^{\mathrm{ss}}
\hookrightarrow S\times \Hig_G^{\mathrm{ss}}.
\end{align}

The \textit{(left) Wilson operator} is given by the functor 
\begin{align*}
    \mathrm{W} \colon \Coh(S)\otimes \Coh(\Hig_G^{\mathrm{ss}}) \to \Coh(\Hig_G^{\mathrm{ss}})
\end{align*}
defined by 
\begin{align*}
    (-) \mapsto p_{\mathrm{H}*}((-)\otimes \mathrm{E})
\end{align*}
where $p_{\mathrm{H}} \colon S\times \Hig_G^{\mathrm{ss}} \to \Hig_G^{\mathrm{ss}}$ is the projection. Note that $p_{\mathrm{H}}$ is not proper, but 
the functor $\mathrm{W}$ preserves coherent sheaves 
since $\mathrm{E}$ is supported on (\ref{support:C})
which is proper over $\Hig_G^{\mathrm{ss}}$.

\subsection{Hecke operator}\label{subsec:hecke2}
Now we consider Hecke operators.  
Let $\mathrm{Hecke}_G$ be the derived moduli stack
classifying Higgs bundles $(F_i, \theta_i)$ for $i=1, 2$ with 
injections 
\begin{align}\label{seq:Fi}
    (F_1, \theta_1) \subset (F_2, \theta_2)
\end{align}
such that $F_2/F_1$ is isomorphic to a skyscraper sheaf of a closed point in $C$. 
By the spectral construction, it is equivalent to the derived
moduli stack which classifies 
exact sequences in $\Coh^{\heartsuit}(S)$
\begin{align}\label{seq:Ei}
    0 \to E_1 \to E_2 \to \mathcal{O}_x \to 0
\end{align}
where $E_i$ are compactly supported pure one-dimensional sheaves on $S$ and $x\in S$.

We have the evaluation morphisms 
\begin{align}\label{dia:greg}
    \xymatrix{
\mathrm{Hecke}_G \ar[d]_{(\ev_3^{\sharp}, \ev_2^{\sharp})} \ar[r]^{\ev_1^{\sharp}} & \Hig_G \\
\mathfrak{M}(1) \times \Hig_G. &
    }
\end{align}
Here 
$\mathfrak{M}(1)$ is the derived moduli stack which classifies skyscraper sheaves $\mathcal{O}_x$ for 
$x\in S$, which is explicitly described as 
\begin{align*}\mathfrak{M}(1)\simeq S\times k[-1]\times \bgm.\end{align*}
The map $\ev_i^{\sharp}$ for $i=1, 2$ sends (\ref{seq:Ei}) to $E_i$
and $\ev_3^{\sharp}$ sends it to $\mathcal{O}_x$. 

We modify the stack $\mathrm{Hecke}_G$ and define the following: 
\begin{align}\label{Hecke2}
\mathrm{Hecke}^{\prime}_G:=\mathrm{Hecke}_G\times_{\mathfrak{M}(1)}(S\times \bgm).
\end{align}
It admits the following evaluation maps 
\begin{align}\label{dia:gregprime}
    \xymatrix{
\mathrm{Hecke}^{\prime}_G \ar[d]_{(\ev_3, \ev_2)} \ar[r]^{\ev_1} & \Hig_G\ar[d] \\
S \times \Hig_G\ar[r] & \rB.
    }
\end{align}
Here $\ev_1, \ev_2$ are induced from $\ev_1^{\sharp}, \ev_2^{\sharp}$ and 
$\ev_3$ is the composition of the projection to $S\times \bgm$ 
with the projection onto $S$. 
We have the following lemma, whose proof will be given in Subsection~\ref{subsec:qsmooth}: 
\begin{lemma}\label{lem:qsmooth}
The maps $\ev_1, \ev_2$ are quasi-smooth 
and proper of relative virtual dimension $1$. 
The maps $(\ev_3, \ev_2), (\ev_3, \ev_1)$ are also quasi-smooth and proper of relative virtual dimension 
$-1$. 
\end{lemma}

The \textit{(left) Hecke operator}
is the functor 
\begin{align*}
    \mathrm{H} \colon \Coh(S) \otimes \Coh(\Hig_G) \to \Coh(\Hig_G)
\end{align*}
defined by 
\begin{align*}
    (-)\mapsto \ev_{1*}(\ev_3, \ev_2)^*(-). 
\end{align*}
Note that the functor $\mathrm{H}$ preserves coherent sheaves by 
Lemma~\ref{lem:qsmooth}. Moreover we have the following lemma, which is not used in this section but will be relevant 
in the next section. 
It is essentially proved in~\cite[Proposition~9.8]{PTlim}, but the statement in loc. cit. 
is different, so we will give its proof in Subsection~\ref{subsec:HeckeL}. 
\begin{lemma}\label{lem:HeckeL}
The functor $\mathrm{H}$ restricts to the functor of limit categories 
\begin{align*}
    \mathrm{H} \colon \Coh(S) \otimes \LL(\Hig_G(\chi))_w \to \LL(\Hig_G(\chi-1))_w.
\end{align*}
\end{lemma}

There is an open substack 
\begin{align*}
    \mathrm{Hecke}_G^{\mathrm{greg}} \subset \mathrm{Hecke}_G
\end{align*}
consisting of (\ref{seq:Fi}) such that either (equivalently, both) of the $\theta_i$ is/are generically regular. 
Then we have the following Cartesian squares 
\begin{align}\label{dia:Hgreg}
    \xymatrix{
S\times \Hig_G^{\mathrm{greg}} \dinclusion \diasquare
& \mathrm{Hecke}_G^{\prime \mathrm{greg}} \ar[l]_-{(\ev_3, \ev_2)}
\ar[r]^-{\ev_1} \dinclusion\diasquare & \Hig_G^{\mathrm{greg}} \dinclusion \\
  S\times \Hig_G & \ar[l]_-{(\ev_3, \ev_2)} \mathrm{Hecke}_G^{\prime} \ar[r]^-{\ev_1} & \Hig_G.
    }
\end{align}
Therefore we also have the left Hecke operator on the generically regular 
locus 
\begin{align*}
    \mathrm{H} \colon \Coh(S) \otimes \Coh(\Hig_G^{\mathrm{greg}}) \to \Coh(\Hig_G^{\mathrm{greg}})
\end{align*}
by $\mathrm{H}= \ev_{1*}(\ev_3, \ev_2)^*$ in the upper diagram of (\ref{dia:Hgreg}).

\subsection{Wilson/Hecke compatibility}
Recall the functor (\ref{Phi:P}) constructed by the Arinkin sheaf. 
We have the following diagram of $\rB$-linear functors
\begin{align}\label{com:WH}
\xymatrix{
\Coh(S) \otimes \Coh(\Hig^{\mathrm{ss}}_G) \ar[r]^-{\mathrm{W}} \ar[d]_-{\id\otimes\Phi^{\mathrm{P}}} & \Coh(\Hig_G^{\mathrm{ss}}) \ar[d]^-{\Phi^{\mathrm{P}}} \\
\Coh(S) \otimes \Coh(\Hig_G^{\mathrm{greg}}) \ar[r]^-{\mathrm{H}} & \Coh(\Hig_G^{\mathrm{greg}}). 
}
\end{align}
The purpose of this section is to prove the above diagram commutes: 
\begin{thm}\label{prop:WH}
The diagram (\ref{com:WH}) commutes over $\rB^{\mathrm{red}} \subset \rB$. 
\end{thm}

We reduce Theorem~\ref{prop:WH} to the statement for classical stacks. 
We fix an open subset $\cB \subset \rB^{\mathrm{cl}}$ and use the notation 
as in Subsection~\ref{subsec:dlconj}, especially the notation for the 
stack $\cH$ in (\ref{sH}). Note that $\cH$ is classical by the diagram (\ref{dia:HigB}). 
We have the Wilson operator over $\cB$
\begin{align}\label{funct:W}
    \mathrm{W} \colon \Coh(S) \otimes \Coh(\cH^{\mathrm{ss}})\to 
    \Coh(\cH^{\mathrm{ss}}) 
\end{align}
defined to be 
\begin{align*} (-) \mapsto p_{\cH*}(p_S^*(-) \otimes \cE), \ \cE:=
    \mathrm{E}|_{S \times \cH^{\mathrm{ss}}} \in \Coh^{\heartsuit}(S\times \cH^{\mathrm{ss}}).
\end{align*}
Here $p_{\cH}, p_S$ are the projections from  
$S\times \cH^{\mathrm{ss}}$. 

The Hecke operator over $\cB$ is given as follows.  
By taking the base change of the diagram (\ref{dia:gregprime}) with respect to 
$\cB \hookrightarrow \rB$, we have the commutative diagram 
\begin{align}\label{redHecke}
\xymatrix{
  \cHecke^{\mathrm{greg}} \ar[r]^-{\ev_1} \ar[d]_{(\ev_3,\ev_2)} &
  \cH^{\mathrm{greg}} \ar[d] \\
  S \times \cH^{\mathrm{greg}} \ar[r] &
  \cB. 
}
\end{align}
Here (by abuse of notation) we have used the same notation $\ev_i$ when pulled back via $\cB \subset \rB$. 
We have the Hecke operator over $\cB$ 
\begin{align}\label{hecke:greg}
    \mathrm{H} \colon \Coh(S) \otimes \Coh(\cH^{\mathrm{greg}})
    \to \Coh(\cH^{\mathrm{greg}})
\end{align}
defined to be 
\begin{align*}
    (-)\mapsto \ev_{1*}(\ev_3, \ev_2)^*(-).
\end{align*}

Recall the Fourier--Mukai functor $\Phi$ as in (\ref{induce:Phi}), defined 
by the Arinkin sheaf:
\begin{align*}
    \Phi \colon \Coh(\cH^{\mathrm{ss}}) \to \Coh(\cH^{\mathrm{greg}}).
\end{align*}
We have the following lemma: 
\begin{lemma}
The commutativity of (\ref{com:WH}) 
over $\cB \times k[-1] \subset \rB$
follows if the following diagram commutes: 
\begin{align}\label{com:WH3}
\xymatrix{
\Coh(S) \otimes \Coh(\cH^{\mathrm{ss}}) 
  \ar[r]^-{\mathrm{W}}
  \ar[d]_-{\mathrm{id}\otimes \Phi} 
& \Coh(\cH^{\mathrm{ss}}) 
  \ar[d]_-{\Phi} \\
\Coh(S) \otimes \Coh(\cH^{\mathrm{greg}}) 
  \ar[r]^-{\mathrm{H}} 
& \Coh(\cH^{\mathrm{greg}}). 
}
\end{align}
\end{lemma}
\begin{proof}
The lemma follows since tensor product of the diagram (\ref{com:WH3}) with $\Coh(k[-1])$ gives the diagram (\ref{com:WH}). 
\end{proof}

\subsection{Dual operators}\label{subsec:dual}
We will reduce the commutativity of the diagram (\ref{com:WH3}) to a commutativity of 
a similar diagram for \textit{dual} Wilson/Hecke operators. 

The \textit{dual Wilson operator} 
\begin{align*}
    \mathrm{W}^{\vee} \colon \Coh(\cH^{\mathrm{ss}})\to \Coh(\cH^{\mathrm{ss}})\otimes \Coh(S)
\end{align*}
is defined to be 
\begin{align*}
    (-) \mapsto p_{\cH}^*(-)\otimes \cE. 
\end{align*}
The \textit{dual Hecke operator} 
\begin{align*}
    \mathrm{H}^{\vee} \colon \Coh(\cH^{\mathrm{greg}})
    \to \Coh(\cH^{\mathrm{greg}})\otimes \Coh(S)
\end{align*}
is defined to be 
\begin{align*}
    (-)\mapsto (\ev_1, \ev_3)_{*}\ev_2^*(-).
\end{align*}
The above functor is associated with the following 
diagram (compare with (\ref{redHecke}))
\begin{align}\notag
    \xymatrix{
\cHecke^{\mathrm{greg}}\ar[r]^-{\ev_2} \ar[d]_{(\ev_1, \ev_3)} & \cH^{\mathrm{greg}} \ar[d] \\
 \cH^{\mathrm{greg}} \times S\ar[r] & \cB.
    }
\end{align}
As in Lemma~\ref{lem:qsmooth}, both of the maps $\ev_2$ and $(\ev_1, \ev_3)$ are quasi-smooth and proper. 

\begin{lemma}\label{lem:dual}
Suppose that the diagram 
    \begin{align}\label{com:WH4}
\xymatrix{
\Coh(\cH^{\mathrm{ss}}) 
  \ar[r]^-{\mathrm{W}^{\vee}} 
  \ar[d]_-{\Phi} 
& \Coh(\cH^{\mathrm{ss}}) \otimes \Coh(S)
  \ar[d]^-{\Phi\otimes \id} \\
\Coh(\cH^{\mathrm{greg}}) 
  \ar[r]^-{\mathrm{H}^{\vee}} 
& \Coh(\cH^{\mathrm{greg}})\otimes \Coh(S)
}
\end{align}
commutes. Then the diagram (\ref{com:WH3}) commutes. 
\end{lemma}
\begin{proof}
It is straightforward to check that the 
kernel objects of the clockwise (resp. anticlockwise) compositions of the diagrams
(\ref{com:WH3}) and (\ref{com:WH4}) are the same. 

Alternatively, for $A\in \Coh(S)$ let 
    \begin{align*}
        \mathrm{W}_A\colon \Coh(\cH^{\mathrm{ss}}) \to 
        \Coh(\cH^{\mathrm{ss}})
    \end{align*}
    be the corresponding Wilson operator. It is recovered from the dual 
    Wilson operator 
    $\mathrm{W}^{\vee}$ by 
    \begin{align*}
    \mathrm{W}_A(-)=p_{\cH*}(\mathrm{W}^{\vee}(-)\otimes p_S^{*}A).
    \end{align*}
    Similarly the corresponding Hecke functor 
    \begin{align*}
        \mathrm{H}_A\colon 
        \Coh(\cH^{\mathrm{greg}})\to \Coh(\cH^{\mathrm{greg}})
    \end{align*}
    is recovered from the dual Hecke operator $\mathrm{H}^{\vee}$ by 
    \begin{align*}
        \mathrm{H}_A(-)=p_{\cH*}(\mathrm{H}^{\vee}(-)\otimes p_S^* A).
    \end{align*}
    Therefore the lemma holds. 
\end{proof}

\begin{remark}\label{rmk:commute2}
The commutativity of the diagram (\ref{com:WH4}) over the elliptic locus 
is proved in~\cite[Theorem~6.3.6]{Gdual}. However the argument in loc. cit. does not apply beyond the elliptic locus, even for the reduced locus, 
since it essentially uses versal deformation of compactified Jacobians of integral curves, and the fact that $\Phi$ is an equivalence over the elliptic locus. Instead we prove the commutativity of (\ref{com:WH4}) by using the explicit construction of the Arinkin sheaf, and the reduction 
to the case of integral curves. 
\end{remark}

\subsection{Hecke operators to the Arinkin sheaf}\label{subsec:hecke}
By Lemma~\ref{lem:dual}, the result of Theorem~\ref{prop:WH} is reduced to showing the commutativity 
of (\ref{com:WH4}). We first show the commutativity of (\ref{com:WH4})
for points over $\Hig_G^{\mathrm{ss}}$. 

For a Higgs bundle $(F, \theta)$, let $E\in \Coh^{\heartsuit}(S)$
be the corresponding sheaf. Let
\begin{align}\label{tauE}
    \tau_E \colon \Spec k \to \cH
\end{align}
be the corresponding map. Then using the Arinkin sheaf (\ref{Arsheaf0}), we have 
\begin{align}\label{PE}
    \cP_E :=(\tau_E\times \id)^* \cP^{\flat}\in \Coh^{\heartsuit}(\cH_b^{\mathrm{greg}})
    \hookrightarrow \Coh^{\heartsuit}(\cH^{\mathrm{greg}}).
\end{align}
Here $b\in \cB$ corresponds to the spectral curve $\cC_b$ of $E$, 
$\tau_E \times \id$ is the map 
\begin{align*}
    \tau_E \times \id \colon \Spec k \times_{\cB}\cH^{\mathrm{greg}}=\cH_b^{\mathrm{greg}} \hookrightarrow \cH\times_{\cB}\cH^{\mathrm{greg}}
\end{align*}
and $\cH_b:=b\times_{\cB}\cH$.

The sheaf (\ref{PE}) is constructed in~\cite{Ardual, MLi} explicitly using Haiman's isospectral 
Hilbert schemes, which we recall now. 
Recall that the Hilbert scheme of points
$\Hilb^d(S)$
parametrizes zero-dimensional closed subschemes $Z\subset S$ with length $d$. 
The \textit{isospectral Hilbert scheme} $\widetilde{\Hilb}^d(S)$ is defined to be the
reduced fiber product, see Subsection~\ref{subsec:notation0}
\begin{align*}
    \xymatrix{
\widetilde{\Hilb}^d(S) \ar[r]^-{\psi} \ar[d]_-{\eta}\diasquare^{\mathrm{red}} & \Hilb^d(S) \ar[d] \\
S^{\times d} \ar[r] & \mathrm{Sym}^d(S).
    }
\end{align*}

By the result of Haiman~\cite{Ha}, the map $\psi$ is a flat finite map 
of degree $d!$, and there is an $\mathfrak{S}_d$-action on 
$\widetilde{\Hilb}^d(S)$.  Then the object
\begin{align}\label{QEd}
    \mathcal{Q}_E^d:=(\psi_{*}\eta^* (E^{\boxtimes d}))^{\mathrm{sign}} 
    \otimes \det \mathcal{A}^{-1}\in \Coh^{\heartsuit}(\Hilb^d(S))
\end{align}
is a Cohen--Macaulay sheaf supported on the closed subscheme 
\begin{align*}\mathrm{Hilb}^d(\cC_b) \subset 
\Hilb^d(S),
\end{align*}
where $\cC_b \subset S$
is the spectral curve 
of $E$. 
Here $\det \mathcal{A}$ is a line bundle on $\Hilb^d(S)$ whose fiber 
at $Z\subset S$ is $\det \chi(\mathcal{O}_Z)$, and $(-)^{\mathrm{sign}}$ 
 stands for the anti-invariant part with respect to
the action of $\mathfrak{S}_d$. 
Then for each quasi-compact open substack 
$U\subset \cH^{\mathrm{greg}}$, 
after applying the isomorphism given by 
\begin{align*}
   \phi_L \colon \cH^{\mathrm{greg}} \stackrel{\cong}{\to} \cH^{\mathrm{greg}}, \ 
   (-) \mapsto (-)\otimes L
\end{align*}
for a sufficiently ample line bundle $L$ on $C$, 
the sheaf $\cP_E|_{\phi_L(U)}$ is a 
descend of $\mathcal{Q}_E^d$ for $d\gg 0$ by the Abel Jacobi map 
(which is a smooth map for $d\gg 0$)
\begin{align}\label{AJ}
\mathrm{AJ} \colon 
    \Hilb^d(\cC/\cB) \to \cH^{\mathrm{greg}}. 
\end{align}
Here $\cC \to \cB$ is the universal spectral curve 
and $\mathrm{AJ}$ sends $Z\subset \cC_b$ to 
the dual of $I_Z\subset \mathcal{O}_{\cC_b}$.
We refer to~\cite[Section~4.1]{Ardual}, \cite[Section~3.2]{MLi} for more details. 

We now take $\cB \subset \rB^{\mathrm{red, cl}}$ and a Higgs bundle 
$(F, \theta)$ over $b\in \cB$ with corresponding sheaf $E\in \Coh^{\heartsuit}(S)$, so that its spectral curve $\cC_b \subset S$ 
is reduced. Recall that $\cH^{\mathrm{greg}}=\cH$ over such $\cB$. 
Using the above explicit construction of $\cP_E$, we show the following 
proposition: 
\begin{prop}\label{prop:PE}
The sheaf $\cP_E \in \Coh^{\heartsuit}(\cH)$ satisfies 
that 
\begin{align*}
    \mathrm{H}^{\vee}(\cP_E) \in \Coh^{\heartsuit}(\cH \times S)
\end{align*}
and it is a maximal Cohen--Macaulay sheaf on $\cH_b \times \cC_b$. 
\end{prop}
\begin{proof}
We define the derived stack over $\cH^{\dag}$
\begin{align}\notag
p^{\dag} \colon 
   \cH^{\dag} \to \cH
\end{align}
which classifies pairs 
\begin{align*}
    (E, s), \ E\in \Coh^{\heartsuit}(S), \ s \colon \mathcal{O}_S \to E
\end{align*}
where $E$ is compactly supported pure one-dimensional sheaf with support $\cC_b$ for $b\in \cB$ and $s$ is 
generically surjective. The map $p^{\dag}$ forgets the section $s$. 
    Its classical truncation is nothing but the moduli space of 
    Pandharipande-Thomas stable pairs~\cite{MR2545686} on $S$. Therefore by~\cite[Proposition~B.8]{PT3}, 
    we have 
    \begin{align*}
        (\cH^{\dag})^{\mathrm{cl}} \cong \bigcup_{d\geq 0} \Hilb^d(\cC/\cB).
    \end{align*}
    
We consider the following commutative diagram 
\begin{align}\label{dia:dag}
\xymatrix{
\cH^{\dag} \ar[d]^-{p^{\dag}}
&
    \ar[r]^-{(\ev_1^{\dag}, \ev_3^{\dag})} \diasquare
    \ar[d] \cHecke^{\dag} 
    \ar[l]_-{\ev_2^{\dag}} 
&\cH^{\dag}\times S \ar[d]^-{p^{\dag}\times \id_S} \\
\cH 
& \cHecke^{\mathrm{greg}} 
    \ar[r]^-{(\ev_1, \ev_3)} 
    \ar[l]_-{\ev_2} 
& \cH\times S
}
\end{align}
    Here $\cHecke^{\dag}$
    is defined by the right Cartesian square, and it classifies the diagrams
    \begin{align}\label{seq:section}
        \xymatrix{
& \mathcal{O}_S \ar@{=}[r] \ar[d]& \mathcal{O}_S \ar[d]& &\\
0 \ar[r] & E_1 \ar[r] & E_2 \ar[r] & \mathcal{O}_x \ar[r] & 0.      
        }
    \end{align}
    Here the bottom sequence is as in (\ref{seq:Ei}) and each vertical 
    arrow is generically surjective. 
    The right square in (\ref{dia:dag}) is Cartesian but the left square in (\ref{dia:dag}) is not Cartesian. 

    Let 
    \begin{align*}
        \overline{S}:=\mathbb{P}(\mathcal{O}_C \oplus \Omega_C)
        \to C
    \end{align*}
    be the projective compactification of $S$. There is a morphism 
    \begin{align}\notag
       \cH^{\dag} \to 
        \mathbf{Hilb}(\overline{S})=\bigcup_{d\geq 0}\mathbf{Hilb}^d(\overline{S})
    \end{align}
    where $\mathbf{Hilb}^d(\overline{S})$ is the derived 
    moduli stack of ideal sheaves $I_Z \subset \mathcal{O}_{\overline{S}}$ of zero-dimensional subscheme 
    $Z\subset \overline{S}$ with length $d$. 
The above morphism is given by 
\begin{align*}
    (\mathcal{O}_S \stackrel{s}{\to} E) \mapsto \mathbb{D}_{\overline{S}}(I)\otimes \det I
\end{align*}
where $I=(\mathcal{O}_{\overline{S}} \stackrel{s}{\to}
E)$ is the two term 
complex on $\overline{S}$, and $\mathbb{D}_{\overline{S}}(I)$ is its derived dual. 
For a closed subscheme $Z\subset S$ of dimension zero, we have 
\begin{align*}
    \Ext_{\overline{S}}^2(I_Z, I_Z)=\Hom(I_Z, I_Z\otimes \omega_{\overline{S}})=H^0(\omega_{\overline{S}})=0.
\end{align*}
    Therefore $\mathbf{Hilb}(\overline{S})$ is classical, i.e.  
    \begin{align*}
        \mathbf{Hilb}(\overline{S})=\Hilb(\overline{S}).
    \end{align*}
    Moreover there is a factorization 
    \begin{align*}
        \cH^{\dag} \hookrightarrow 
         \cB\times\Hilb(S)\hookrightarrow \cB\times\Hilb(\overline{S})
    \end{align*}
    where the first arrow is a closed immersion and the second arrow is an open immersion. 

    Similarly let $\mathbf{Hilb}^{d, d-1}(\overline{S})$ be the 
    derived moduli stack which classifies exact sequences 
    \begin{align}\label{IZ}
        0\to I_{Z_2} \to I_{Z_1} \to \mathcal{O}_x \to 0
    \end{align}
    where $Z_1 \subset Z_2\subset \overline{S}$ are zero-dimensional subschemes with 
    length $d-1$, $d$, respectively. It is quasi-smooth~\cite[Proposition~3.8]{PoSa}, and its classical truncation 
    has the expected dimension $2d$ by~\cite[Proposition~3.5.2]{Ha}.
    Therefore  
    $\mathbf{Hilb}^{d, d-1}(\overline{S})$ is also classical 
    \begin{align*}
       \mathbf{Hilb}^{d, d-1}(\overline{S})=\mathrm{Hilb}^{d, d-1}(\overline{S}) 
    \end{align*}
    where the right-hand side is the classical nested Hilbert 
    scheme which parametrizes zero-dimensional subschemes $Z_1 \subset Z_2 \subset \overline{S}$. Let
    \begin{align*}
        \mathrm{Hilb}^{d, d-1}(S) \subset \Hilb^{d, d-1}(\overline{S})
    \end{align*}
    be the open subscheme corresponding to $Z_1 \subset Z_2 \subset S$, and 
    set 
    \begin{align*}
        \Hilb^{\mathrm{nest}}(S):=\bigcup_{d\geq 1}\Hilb^{d, d-1}(S).
    \end{align*}
    
    We have the evaluation morphisms 
    \begin{align*}
        \Hilb^d(S) \stackrel{\ev_2^{\mathrm{H}}}{\leftarrow} \Hilb^{d, d-1}(S) \stackrel{(\ev_1^{\mathrm{H}}, \ev_3^{\mathrm{H}})}{\to}
        \Hilb^{d-1}(S)\times S.
    \end{align*}
   By taking the dual of the sequence (\ref{seq:section}), we obtain 
   \begin{align}\label{seq:DP}
       0\to \mathbb{D}_{\overline{S}}(P_{2}) \otimes L \to \mathbb{D}_{\overline{S}}(P_1)\otimes L \to \mathcal{O}_x \to 0
   \end{align}
   where $P_i$ and $L$ are given by 
   \begin{align*}P_i=(\mathcal{O}_{\overline{S}} \to E_i)\in \Coh(\overline{S}), \ 
   L=\det P_1=\det P_2. 
   \end{align*}
   The sequence (\ref{seq:DP}) is of the form (\ref{IZ}), and 
   the above construction induces the closed immersion 
   \begin{align*}
      \cHecke^{\dag}\hookrightarrow 
       \cB\times\Hilb^{\mathrm{nest}}(S)
   \end{align*}
   such that there is a commutative diagram 
   \begin{align}\label{dia:dag2}
\xymatrix{\cB\times\Hilb(S)  \diasquare
& \cB\times\Hilb^{\mathrm{nest}}(S)
    \ar[r]^-{(\ev_1^{\mathrm{H}}, \ev_3^{\mathrm{H}})} 
    \ar[l]_-{\ev_2^{\mathrm{H}}} 
& \cB\times\Hilb(S)\times S \\
\cH^{\dag} \uinclusion_-{i}& \cHecke^{\dag} 
\uinclusion    \ar[r]^-{(\ev_1^{\dag}, \ev_3^{\dag})} 
    \ar[l]_-{\ev_2^{\dag}} 
&\cH^{\dag}\times S \uinclusion_-{i\times \id_S}
}
\end{align}
Here the vertical arrows are closed immersions and the left square is Cartesian (since $Z_2 \subset \cC_b$ implies $Z_1 \subset \cC_b$), but the right square is not.

From the following diagram 
\begin{align*}
    \cH \stackrel{p^{\dag}}{\leftarrow}\cH^{\dag} \stackrel{i}{\hookrightarrow}\cB\times\Hilb(S)
\end{align*}
we define the following object 
\begin{align*}
    \widetilde{\mathcal{Q}}_E:=i_{*}(p^{\dag})^* \cP_E \in \Coh(\{b\}\times\Hilb(S)).
\end{align*}
Here $b\in \cB$ corresponds to the spectral curve of $E$. 
We denote by 
\begin{align*}\widetilde{\mathcal{Q}}_E^d\in \Coh(\Hilb^d(S))
\end{align*}
the component of $\widetilde{\mathcal{Q}}_E$. By the construction of 
$\cP_E$ as a descent of $\mathcal{Q}_E^d$ for $d\gg 0$, we have $\mathcal{Q}_E^d=\widetilde{\mathcal{Q}}_E^d$ for $d\gg 0$. 

Let $\mathrm{H}^{\vee}_{\Hilb}$ be the functor 
\begin{align}\label{HEvee}
    \mathrm{H}^{\vee}_{\Hilb}\colon \Coh(\Hilb(S)) \to \Coh(\Hilb(S)) \otimes \Coh(S)
\end{align}
defined by 
\begin{align*}
    (-) \mapsto (\ev_1^{\mathrm{H}}, \ev_3^{\mathrm{H}})_{*}(\ev_2^{\mathrm{H}})^*(-). 
\end{align*}
Then by the diagrams (\ref{dia:dag}), (\ref{dia:dag2}) and base change, we have 
\begin{align}\label{isom:HQP}
    (i\times \id_S)_{*}(p^{\dag}\times \id_S)^* \mathrm{H}^{\vee}(\cP_E)\cong \mathrm{H}^{\vee}_{\Hilb}(\widetilde{\mathcal{Q}}_E).
\end{align}
Therefore by the smoothness of the Abel Jacobi map (\ref{AJ}) for $d\gg 0$, 
the proposition follows from Lemma~\ref{lem:heckeQ} below. 
\end{proof}

\begin{lemma}\label{lem:heckeQ}
For $d\gg 0$, we have
\begin{align}\label{hecke:Q}
    \mathrm{H}_{\Hilb}^{\vee}(\mathcal{Q}_E^d)
    \in \Coh^{\heartsuit}(\Hilb^{d-1}(S)\times S),
\end{align}
and it is a maximal Cohen--Macaulay sheaf on
$\Hilb^{d-1}(\cC_b)\times \cC_b$.
\end{lemma}

\begin{proof}
Consider the following diagram:
\begin{align}\label{dia:Sd}
    \xymatrix{
S^{\times d} \ar[d]\diasquare^{\mathrm{red}}&
\widetilde{\Hilb}^d(S) \ar[l]^-{\eta}  \ar[d]_-{\psi} \diasquare&
W^d \ar[d]_-{\psi_2} \ar[l]^-{\eta_2} \ar@/_18pt/[ll]_-{\eta_3} \\
\mathrm{Sym}^d(S) \ar@{=}[d]&
\Hilb^d(S) \ar[l]&
\Hilb^{d,d-1}(S) \ar[l]^-{\ev_2^{\mathrm{H}}}
\ar[d]^{(\ev_1^{\mathrm{H}}, \ev_3^{\mathrm{H}})} \\
\mathrm{Sym}^d(S)&
\mathrm{Sym}^{d-1}(S) \times S \ar[l]&
\Hilb^{d-1}(S)\times S \ar[l].
    }
\end{align}
Here the bottom arrows are the addition map for symmetric powers of $S$
and the Hilbert--Chow map for the Hilbert scheme of points.
By the base change, we have
\begin{align*}
    \mathrm{H}^{\vee}_{\Hilb}(\mathcal{Q}_E^d)
    &=
    \mathrm{H}^{\vee}_{\Hilb}
    \bigl((\psi_* \eta^* E^{\boxtimes d})^{\mathrm{sign}}
    \otimes \det \mathcal{A}^{-1}\bigr)\\
    &=
    (\ev_1^{\mathrm{H}}, \ev_3^{\mathrm{H}})_*
    \ev_2^{\mathrm{H}*}
    \bigl((\psi_* \eta^* E^{\boxtimes d})^{\mathrm{sign}}
    \otimes \det \mathcal{A}^{-1}\bigr)\\
    &\cong
    (\ev_1^{\mathrm{H}}, \ev_3^{\mathrm{H}})_*
    \bigl((\psi_{2*}\eta_3^* E^{\boxtimes d})^{\mathrm{sign}}
    \otimes \ev_2^{\mathrm{H}*}\det \mathcal{A}^{-1}\bigr).
\end{align*}
Moreover,
\begin{align*}
    \ev_2^{\mathrm{H}*}\det \mathcal{A}
    \cong
    (\ev_1^{\mathrm{H}}, \ev_3^{\mathrm{H}})^*
    (\det \mathcal{A}\boxtimes \mathcal{O}_S),
\end{align*}
since the line bundle on $S$ whose fiber at $x\in S$ is
$\det \chi(\mathcal{O}_x)$ is trivial. Therefore,
\begin{align}\label{HQ}
      \mathrm{H}^{\vee}_{\Hilb}(\mathcal{Q}_E^d)
      \cong
      (\ev_1^{\mathrm{H}}, \ev_3^{\mathrm{H}})_*
      \bigl((\psi_{2*}\eta_3^* E^{\boxtimes d})^{\mathrm{sign}}\bigr)
      \otimes \det \mathcal{A}^{-1}.
\end{align}

It is enough to show that
\begin{align}\label{ets:Q}
    (\ev_1^{\mathrm{H}}, \ev_3^{\mathrm{H}})_*
    \bigl((\psi_{2*}\eta_3^* E^{\boxtimes d})^{\mathrm{sign}}\bigr)
    \in \Coh^{\heartsuit}(\Hilb^{d-1}(S)\times S)
\end{align}
and that it is a maximal Cohen--Macaulay sheaf on
$\Hilb^{d-1}(\cC_b)\times \cC_b$.
Note that the diagram (\ref{dia:Sd}) is over $\mathrm{Sym}^d(S)$, and the 
above question is local over $\mathrm{Sym}^d(S)$. 
Therefore it is enough to check this formally locally on $\mathrm{Sym}^d(S)$. 
We reduce it to the
Hecke-eigen property for compactified Jacobians of integral curves,
proved in~\cite[Theorem~6.3.6]{Gdual}.

Let $(x_1,\ldots,x_d)\in \cC_b$. By Lemma~\ref{lem:intmodel}, 
we can find data
\begin{align*}
    (D,S'=\mathbb{P}^2, y_1,\ldots,y_d), \quad D\subset S', \quad y_i\in D,
\end{align*}
where $D$ is an integral plane curve, $D\in \lvert \mathcal{O}_{\mathbb{P}^2}(m) \rvert$
for $m\gg 0$, and the singularity of $D$ at $y_i$ is formally
isomorphic to the singularity of $\cC_b$ at $x_i$, namely
\begin{align*}
    \widehat{\mathcal{O}}_{\cC_b,x_i}
    \cong
    \widehat{\mathcal{O}}_{D,y_i}.
\end{align*}
Moreover, we may choose a rank-one torsion-free sheaf $F$ on $D$ such
that
\begin{align}\label{isom:EF}
    E|_{\widehat{\cC}_{b,x_i}}
    \cong
    F|_{\widehat{D}_{y_i}},
    \qquad 1\leq i\leq d.
\end{align}

Let $\overline{J}_D$ be the compactified Jacobian of $D$, regarded as
the classical moduli space of rank-one torsion-free sheaves on $D$.
Let $Hecke_D$ be the classical moduli space of exact sequences on $D$
\begin{align*}
    0\to L_1 \to L_2 \to \mathcal{O}_y \to 0,
\end{align*}
where $L_i$ are rank-one torsion-free sheaves and $y\in D$.
We have evaluation maps
\begin{align*}
    \overline{J}_D
    \stackrel{\ev_{D2}}{\leftarrow}
    Hecke_D
    \stackrel{(\ev_{D1},\ev_{D3})}{\longrightarrow}
    \overline{J}_D \times D.
\end{align*}
Let $\mathrm{H}_D^{\vee}$ be the functor
\begin{align*}
    \mathrm{H}_{D}^{\vee}
    =
    (\ev_{D1}, \ev_{D3})_*\ev_{D2}^*
    \colon
    \Coh(\overline{J}_D)
    \to
    \Coh(\overline{J}_D)\otimes \Coh(D).
\end{align*}
Let $\cP_F\in \Coh^{\heartsuit}(\overline{J}_D)$ be the corresponding
Arinkin sheaf. In~\cite[Theorem~6.3.6]{Gdual}, the Wilson--Hecke
compatibility for Arinkin's autoequivalence~\cite{Ardual} is proved.
As a consequence, we have the Hecke-eigen property
\begin{align}\label{eigen:D}
    \mathrm{H}_D^{\vee}(\cP_F)
    \cong
    \cP_F\boxtimes F.
\end{align}

We can interpret the above Hecke-eigen property in terms of the Hecke
correspondence for derived moduli stacks. 
Let $b' \in \cB' \subset \lvert \mathcal{O}_{\mathbb{P}^2}(m) \rvert$, where
$\cB'$ is the open subset corresponding to integral plane curves, and suppose
that $b'$ corresponds to $D$. 
Replacing $S$ by $S'$ and $\cB$ by $\cB'$, we obtain the diagram of Hecke
correspondences for derived moduli stacks
\begin{align*}
\xymatrix{
 \cHecke' \ar[r]^-{\ev_2'} \ar[d]_-{(\ev_1', \ev_3')} & \cH' \ar[d] \\
 S'\times \cH' \ar[r] & \cB'.
}
\end{align*}
Here $\cH'$ is the derived moduli stack of pure one-dimensional sheaves on
$S'$ whose support curve lies in $\cB'$, and $\cHecke'$ is the derived moduli
stack of exact sequences
\begin{align*}
    0\to E_1 \to E_2 \to \mathcal{O}_x \to 0,
    \quad E_i \in \Coh^{\heartsuit}(S'), \quad x\in S',
\end{align*}
where each $E_i$ is pure one-dimensional. Then, as in
Lemma~\ref{lem:qsmooth}, the maps $\ev_1'$ and $\ev_2'$ are quasi-smooth and
proper of relative virtual dimension one. Moreover:
\begin{itemize}[leftmargin=1.5em]
    \item The stack $\cH'$ is smooth, and in particular classical. Indeed,
    $\Ext_{\mathbb{P}^2}^2(E,E)=0$ for a one-dimensional sheaf $E$ on
    $\mathbb{P}^2$ with integral support. Moreover, $\cH' \to \cB'$ is flat,
    and its fiber over $b'$ is $\overline{J}(D)\times \bgm$; this follows
    from~\cite[Theorem~5.5(ii)]{MRF3}.
    
    \item The stack $\cHecke'$ is classical, and $\cHecke' \to \cB'$ is flat
    with fiber over $b'$ equal to $Hecke_D \times \bgm$; this follows
    from~\cite[Lemma~6.3.11]{Gdual}.
\end{itemize}
Therefore the Hecke operator for $S'$,
\begin{align*}
    \mathrm{H}^{'\vee}
    =(\ev_1', \ev_3')_{*}(\ev_2')^*
    \colon \Coh(\cH') \to \Coh(\cH') \otimes \Coh(S'),
\end{align*}
also satisfies the Hecke-eigen property \eqref{eigen:D}:
\begin{align*}
    \mathrm{H}^{'\vee}(\cP_F) \cong \cP_F \boxtimes F.
\end{align*}
In particular, $\mathrm{H}^{'\vee}(\cP_F)$ is a maximal
Cohen--Macaulay sheaf on $\cH'_{b'}\times D$.

Let $\mathrm{H}_{\Hilb'}^{\vee}$ be the functor
\begin{align*}
    \Coh(\Hilb(S'))
    \to
    \Coh(\Hilb(S'))\otimes \Coh(S')
\end{align*}
defined as in~\eqref{HEvee} for $S'$. Consider the diagram
\begin{align*}
    \cH'
    \stackrel{p^{'\dag}}{\leftarrow}
    \cH^{'\dag}
    \stackrel{i'}{\hookrightarrow}
    \Hilb(S').
\end{align*}
Here $\cH^{'\dag}$ is defined as $\cH^{\dag}$, by replacing $S$ with $S'$. 
Set
\begin{align*}
    \widetilde{\mathcal{Q}}_F
    :=
    i'_{*}(p'^{\dag})^*\cP_F
    \in \Coh(\Hilb(S')).
\end{align*}
By the same argument as in~\eqref{isom:HQP}, we have an isomorphism
\begin{align}\label{HQ'}
    (i'\times \id_{S'})_*
    (p^{'\dag}\times \id_{S'})^*
    \mathrm{H}^{'\vee}(\cP_F)
    \cong
    \mathrm{H}_{\Hilb'}^{\vee}(\widetilde{\mathcal{Q}}_F).
\end{align}
By~\eqref{eigen:D}, for $d\gg 0$ the left-hand side on $\Hilb^d(S')$ is a maximal Cohen--Macaulay sheaf
on $\Hilb^{d-1}(D)\times D$, since $p'^{\dag}$
is smooth on $\cH'^{\dag} \cap \Hilb^d(S')$ for $d\gg 0$ and $i'$ is a closed immersion. 
Also let 
\begin{align*}
    \mathcal{Q}_F^d \in \Coh^{\heartsuit}(\Hilb^d(S'))
\end{align*}
be defined as in (\ref{QEd}) for the pair $(S', F)$. 
Then as before, we have $\widetilde{\mathcal{Q}}_F^d=\mathcal{Q}_F^d$ for $d\gg 0$. 
It follows that, for $d\gg 0$ we have 
\begin{align*}
     \mathrm{H}_{\Hilb'}^{\vee}(\mathcal{Q}_F^d) \in \Coh^{\heartsuit}(\Hilb^{d-1}(S') \times S')
\end{align*}
and it is a maximal Cohen--Macaulay sheaf on $\Hilb^{d-1}(D) \times D$. 
Equivalently, the same property holds for the object (\ref{ets:Q}) 
constructed from the pair $(S', F)$. 

Finally, the diagram~\eqref{dia:Sd} and the analogous diagram for $S'$
are identified after taking formal completions over the corresponding
points of the symmetric powers:
\begin{align}\label{isom:formal}
    \widehat{\mathrm{Sym}^d(S)}_{(x_1+\cdots+x_d)}
    \cong
    \widehat{\mathrm{Sym}^d(S')}_{(y_1+\cdots+y_d)}.
\end{align}
Under this identification, and using~\eqref{isom:EF}, the
objects (\ref{ets:Q}) 
corresponding to $(\cC_b,E,S)$ and $(D,F,S')$, pulled back under the base-change maps 
\begin{align}\label{map:formal}
     \widehat{\mathrm{Sym}^d(S)}_{(x_1+\cdots +x_d)}\to 
     \mathrm{Sym}^d(S), \ 
    \widehat{\mathrm{Sym}^d(S')}_{(y_1+\cdots +y_d)}
    \to \mathrm{Sym}^d(S')
\end{align}
are isomorphic. 
Indeed, Hilbert schemes of
points and nested Hilbert schemes are formally local over the symmetric
power: their formal completions over a zero-dimensional cycle depend only
on the formal neighborhoods of the support points. The same is true for
the isospectral Hilbert schemes and for the diagram~\eqref{dia:Sd}, since
they are obtained from these Hilbert schemes by fiber products. Moreover,
completion commutes with finite push-forward and with the sign idempotent
in characteristic zero. Therefore the objects~\eqref{ets:Q}
for $(\cC_b,E,S)$ pulled back by the first map in (\ref{map:formal}) 
is identified with the corresponding object for $(D,F,S')$
pulled back under the second map in (\ref{map:formal}). 
Since the desired maximal Cohen--Macaulay properties hold
for the latter by maximal Cohen--Macaulay property of (\ref{HQ'}),
they also hold for
$\mathrm{H}_{\Hilb}^{\vee}(\mathcal{Q}_E^d)$ 
at every point lying over the cycle $x_1+\cdots+x_d$.
As the point was arbitrary, the lemma follows.
\end{proof}

\subsection{Proof of Theorem~\ref{prop:WH}}\label{subsec:proofWH}
\begin{proof}
We take an open subset $\cB \subset \rB^{\mathrm{red, cl}}$. Since $\cH^{\mathrm{greg}}=\cH$ in this case, we omit $\mathrm{greg}$ from the notation. 
By Lemma~\ref{lem:dual}, it is enough to prove that (\ref{com:WH4}) commutes. 

Consider the composition 
\begin{align}\notag
\Coh(\cH^{\mathrm{ss}}) \stackrel{\mathrm{W}^{\vee}}{\to} \Coh(\cH^{\mathrm{ss}})\otimes \Coh(S) 
\stackrel{\Phi\otimes \id}{\to}\Coh(\cH)\otimes \Coh(S).
\end{align}
Its kernel object $\cP_{\mathrm{W}}$ relative over $\cB$ is given as follows. 
For the diagram 
\begin{align*}
\xymatrix{
   \cH^{\mathrm{ss}}\times_{\cB}\cH \times S \ar[r]^-{p_{13}} \ar[d]_-{p_{12}} & 
   \cH^{\mathrm{ss}}\times S \\
 \cH^{\mathrm{ss}}\times_{\cB}\cH & 
      }
\end{align*}
We have 
\begin{align}\label{formula:PW}
    \cP_{\mathrm{W}}=p_{12}^*\cP\otimes p_{13}^* \cE \in 
    \Coh(\cH^{\mathrm{ss}}\times_{\cB}\cH \times S).
\end{align}
The above object lies in $\Coh(-)$ since $p_{12}$ is flat and $\cE$ is a perfect 
complex. 

Suppose that $E\in \Coh^{\heartsuit}(\cC_b)$ is semistable, 
so that the map $\tau_E$ in (\ref{tauE}) factors through 
$\cH^{\mathrm{ss}} \subset \cH$.
Then the pull-back of $\cP_{\mathrm{W}}$ by the map 
\begin{align}\label{tauEid}
\tau_E \times \id \colon
    \Spec k \times_{\cB} \cH \times S=\cH_b\times S
    \to \cH^{\mathrm{ss}}\times_{\cB} \cH \times S
\end{align}
is isomorphic to 
\begin{align}\label{OE}
\cP_E\boxtimes E, \ E\in \Coh^{\heartsuit}(S).
\end{align}
Since it is a Cohen--Macaulay sheaf, 
for any Cohen--Macaulay $k$-scheme $T$ with a morphism $T\to \cH^{\mathrm{ss}}$, 
the pull-back of $\cP_{\mathrm{W}}$ to $T\times_{\cB}\cH\times S$ is a 
$T$-flat family of Cohen--Macaulay sheaves by Lemma~\ref{lem:CM-family}.
Since $\cH^{\mathrm{ss}}$ is Cohen--Macaulay, 
we have  
\begin{align*}
    \cP_{\mathrm{W}} \in \Coh^{\heartsuit}(\cH^{\mathrm{ss}}\times_{\cB}\cH \times S)
\end{align*}
and it is a Cohen--Macaulay sheaf flat over $\cH^{\mathrm{ss}}$. 
Since $E$ is supported on the corresponding spectral curve $\cC_b$, we indeed 
have 
\begin{align*}
      \cP_{\mathrm{W}} \in \Coh^{\heartsuit}(\cH^{\mathrm{ss}}\times_{\cB}\cH \times_{\cB} \cC)
\end{align*}
which is a maximal Cohen--Macaulay sheaf, under the obvious fully-faithful functor 
\begin{align*}
    \Coh^{\heartsuit}(\cH^{\mathrm{ss}}\times_{\cB}\cH \times_{\cB} \cC) 
    \hookrightarrow \Coh^{\heartsuit}(\cH^{\mathrm{ss}}\times_{\cB}\cH \times S).
\end{align*}

Similarly consider the composition 
\begin{align}\notag
\Coh(\cH^{\mathrm{ss}}) \stackrel{\Phi}{\to} \Coh(\cH) 
\stackrel{\mathrm{H}^{\vee}}{\to}\Coh(\cH)\otimes \Coh(S).
\end{align}
Its kernel object $\cP_{\mathrm{H}}$ relative over $\cB$ is given as follows. 
Consider the following diagram 
\begin{align*}
  \xymatrix{
\mathcal{M} \ar[r] \ar[d]_-{p_2} \diasquare \ar@/^20pt/[rr]^-{q_3} \ar@/_40pt/[dd]_-{p_3}&
   \cHecke^{\mathrm{greg}} \ar[r]_-{q} \ar[d]_-{p} &
   \cH \times S \\
\cH^{\mathrm{ss}} \times_{\cB} \cH
   \ar[r] \ar[d]  &
   \cH \\
\cH^{\mathrm{ss}} &
}
\end{align*}
We have the induced morphism 
\begin{align}\label{mappq}
(p_3, q_3) \colon \mathcal{M} \to \cH^{\mathrm{ss}} \times_{\cB} \cH\times S    
\end{align}
and 
\begin{align*}
    \cP_{\mathrm{H}}=(p_3, q_3)_{*}p_2^*\cP \in \Coh(\cH^{\mathrm{ss}} \times_{\cB} \cH\times S  ).
\end{align*}
The above object lies in $\Coh(-)$ since the map (\ref{mappq}) is proper (which follows from the properness of $q$) and $p_2$ is 
quasi-smooth. 
By Proposition~\ref{prop:PE}, the pull-back of $\cP_{\mathrm{H}}$ under 
the map (\ref{tauEid}) satisfies that 
\begin{align*}
    (\tau_E \times \id)^* \cP_{\mathrm{H}} 
    \cong \mathrm{H}^{\vee}(\cP_E)\in \Coh^{\heartsuit}(\cH \times S)
\end{align*}
and is a maximal Cohen--Macaulay sheaf on $\cH_b \times \cC_b$. 
Therefore by the same argument for $\cP_{\mathrm{W}}$, we have 
\begin{align*}
      \cP_{\mathrm{H}} \in \Coh^{\heartsuit}(\cH^{\mathrm{ss}}\times_{\cB}\cH \times_{\cB} \cC)
\end{align*}
and is a maximal Cohen--Macaulay sheaf. 

It is enough to show that both kernel objects are isomorphic, i.e. 
\begin{align}\label{isom:PWH}
    \cP_{\mathrm{W}} \cong \cP_{\mathrm{H}}.
\end{align}
Note that $\cH^{\mathrm{ss}}\times_{\cB}\cH \times_{\cB} \cC$
has pure-dimension since $\cH$ has pure-dimension and 
$\cH \to \cB$, $\cC \to \cB$ are flat. 
Therefore by the uniqueness of Cohen--Macaulay extension in Lemma~\ref{lem:MCM-extension}, 
see~\cite[Lemma~2.2]{Ardual}, 
it is enough to show that $\cP_{\mathrm{W}}=\cP_{\mathrm{H}}$
outside a codimension two closed substack.
Note that the stack 
\begin{align*}
    \cH^{\mathrm{ss}}\times_{\cB}\cH \times_{\cB} \cC
\end{align*}
classifies tuples
\begin{align}\label{F12x}
    (F_1, F_2, x), \ F_i \in \Coh^{\heartsuit}(\cC_b), \ b\in \cB, \ x\in \cC_b
\end{align}
such that $F_1$ is a semistable sheaf and $F_2$ is generically a line bundle on $\cC_b$. 
We define the open substack 
\begin{align*}
    \mathcal{U} \subset \cH^{\mathrm{ss}}\times_{\cB}\cH \times_{\cB} \cC
\end{align*}
to be consisting of (\ref{F12x}) such that $F_1$ is a line bundle on $\cC_b$ 
and $F_2$ is a line bundle at $x$. Since the above open immersion is an isomorphism over $\cB^{\mathrm{sm}}$, and the complement has 
codimension at least one on each fiber over $b\in \cB$ by 
Remark~\ref{rmk:dense}, 
the complement of $\mathcal{U}$ is of codimension greater than or equal to two. Since both sides of (\ref{isom:PWH}) are maximal Cohen--Macaulay sheaves on $\cH^{\mathrm{ss}} \times_{\cB} \cH \times_{\cB} \cC$, 
by the uniqueness of the Cohen--Macaulay extension outside a codimension-two locus in Lemma~\ref{lem:MCM-extension}, it is enough to show that 
\begin{align}\label{PWHU}
\cP_{\mathrm{W}}|_{\mathcal{U}}\cong \cP_{\mathrm{H}}|_{\mathcal{U}}.
\end{align}

Over the open substack $\mathcal{U}$, the morphism (\ref{mappq}) gives 
an isomorphism 
\begin{align*}
    (p_3, q_3) \colon (p_3, q_3)^{-1}(\mathcal{U}) \stackrel{\cong}{\to} \mathcal{U}.
\end{align*}
Indeed the fiber of (\ref{mappq}) at (\ref{F12x})
is given by 
non-trivial exact sequences 
\begin{align}\label{ext:F2}
    0\to F_2 \to F_2' \to \mathcal{O}_x \to 0. 
\end{align}
Since $\cC_b$ is Gorenstein and $F_2$ is a line bundle on $\cC_b$ at $x$, 
we have 
\begin{align*}
    \Ext_{\cC_b}^1(\mathcal{O}_x, F_2)=\Hom(F_2, \mathcal{O}_x\otimes \omega_{\cC_b})^{\vee}=k.
\end{align*}
It follows that we have 
\begin{align*}
(p_3, q_3)^{-1}(F_1, F_2, x)=\mathbb{G}_m/\mathbb{G}_m=\mathrm{pt}. 
\end{align*}
Therefore over $\mathcal{U}$, the object $\cP_{\mathrm{H}}|_{\mathcal{U}}$ is a line 
bundle whose fiber is 
\begin{align*}
    \cP_{\mathrm{H}}|_{(F_1, F_2, x)} =
    \det \chi(F_1\otimes F_2')\otimes \det \chi(F_1)^{\vee}\otimes 
    \det \chi(F_2')^{\vee} \otimes \det \chi(\mathcal{O}_{\cC_b}),
\end{align*}
where $F_2'$ is the unique non-trivial extension (up to scalar)
in (\ref{ext:F2}).
Using the exact sequence (\ref{ext:F2}), the right-hand side is canonically identified with 
\begin{align*}
    \det \chi(F_1\otimes F_2)\otimes \det \chi(F_1)^{\vee}\otimes 
    \det \chi(F_2)^{\vee} \otimes \det \chi(\mathcal{O}_{\cC_b})\otimes 
    \det \chi(F_1|_{x}).
\end{align*}
It is nothing but $\cP_{\mathrm{W}}|_{(F_1, F_2, x)}$ by the formula (\ref{formula:PW})
for $\cP_{\mathrm{W}}$, therefore (\ref{PWHU}) holds. 
\end{proof}

\subsection{A variant}
We will also use the following variant of the diagram (\ref{com:WH4}). 
Again we take $\cB \subset \rB^{\mathrm{red, cl}}$.

Let $\widetilde{\mathrm{W}}^R$ be the functor 
\begin{align*}
    \widetilde{\mathrm{W}}^R \colon \Coh(\cH^{\mathrm{ss}}) \to \Coh(\cH^{\mathrm{ss}})\otimes \Coh(S)
\end{align*}
defined by 
\begin{align*}
    \widetilde{\mathrm{W}}^R(-)=p_{\cH}^{!}(-)\otimes \mathbb{D}_{S\times \cH}(\cE)
    =p_{\cH}^*(-)\otimes\cE^{\vee}[1].
\end{align*}
Here $\cE^{\vee}=\cH^1(\mathbb{D}_{S\times \cH}(\cE))\in \Coh^{\heartsuit}(S\times \cH^{\mathrm{ss}})$, which is a $\cH^{\mathrm{ss}}$-flat family of pure one-dimensional sheaves on $S$. 
Note that $p_{\cH}$ is not proper, but the above formula makes 
sense since $\mathbb{D}_{S\times \cH}(\cE)$ is supported on 
$\cC\times_{\cB}\cH^{\mathrm{ss}}$ which is proper over 
$\cH^{\mathrm{ss}}$. The functor $\widetilde{\mathrm{W}}^R$ gives a right adjoint 
of the functor $\mathrm{W}$ in (\ref{funct:W}). 

Similarly let $\widetilde{\mathrm{H}}^R$ be the functor 
\begin{align*}
\widetilde{\mathrm{H}}^R \colon \Coh(\cH) \to \Coh(\cH)\otimes \Coh(S)
\end{align*}
defined by 
\begin{align}\label{HRtilde}
\widetilde{\mathrm{H}}^R(-)=(\ev_2, \ev_3)_{*}\ev_1^!(-).
\end{align}
Note that both $\ev_1$ and $(\ev_2, \ev_3)$ are quasi-smooth and proper by Lemma~\ref{lem:qsmooth}, therefore the above functors make sense. 
The functor 
\begin{align*}
    \ev_1^! \colon \Coh(\cH) \to \Coh(\cHecke^{\mathrm{greg}})
\end{align*}
is a right adjoint of $\ev_{1*}$ and given by 
\begin{align}\label{formula:ev1}
    \ev_1^!(-)=\ev_1^* \otimes \omega_{\ev_1}[1]
\end{align}
since $\ev_1$ has virtual dimension one. 
We note the following: 
\begin{lemma}\label{lem:omegaev}
    The line bundle $\omega_{\ev_1}$ is of the form $(\ev_3, \ev_2)^* \mathcal{L}$ for a line bundle $\mathcal{L}$ on $S\times \cH$. 
\end{lemma}
\begin{proof}
    This follows from 
    \begin{align*}
        \omega_{\ev_1}|_{(0\to E_1 \to E_2 \to \mathcal{O}_x \to 0)}
        &=\det \chi(\mathcal{O}_x[-1], E_2)^{\vee} \\
        &=\det \chi(\mathcal{O}_x, E_2) \\
        &=\det \chi(\mathcal{O}_x, E_1)\otimes \det \chi(\mathcal{O}_x, \mathcal{O}_x).
    \end{align*}
    The last expression is a line bundle on $S\times \cH$. 
\end{proof}
The functor (\ref{HRtilde}) is a right adjoint functor of $\mathrm{H}$ given by (\ref{hecke:greg}). 
The proof of the following proposition is almost the same as the proof of the commutativity of (\ref{com:WH4}). In Subsection~\ref{subsec:modify}, we explain where to modify the proof. 
\begin{prop}\label{prop:comad}
The following diagram is commutative: 
\begin{align}\label{com:WH3ad}
\xymatrix{
\Coh(\cH^{\mathrm{ss}}) 
  \ar[r]^-{\widetilde{\mathrm{W}}^R}
  \ar[d]_-{\Phi} 
& \Coh(\cH^{\mathrm{ss}}) \otimes \Coh(S)
  \ar[d]^-{\Phi\otimes \id} \\
\Coh(\cH) 
  \ar[r]^-{\widetilde{\mathrm{H}}^R} 
& \Coh(\cH)\otimes \Coh(S). 
}
\end{align}
\end{prop}

\section{Dolbeault Langlands for reduced spectral curves}
In this section, we prove Conjecture~\ref{conj:dl} over an open subset of $\rB^{\mathrm{red}}$, 
under two additional hypotheses; Whittaker normalization and some restriction of singularities of spectral curves.  
\subsection{Higgs bundles with reduced spectral curves}
Recall that $\rB^{\mathrm{red}} \subset \rB$ is the open subset corresponding 
to reduced spectral curves. 
We fix an open subset 
\begin{align}\label{Bcirc}
    \cB \subset \rB^{\mathrm{red, cl}}.
\end{align}
As in (\ref{sH}), we write 
\begin{align}\label{H:simple}
    \cH=\Hig_{G}\times_{\rB} \cB \subset \Hig_G^{\cl},
\end{align}
and $\cC \to \cB$ denotes the universal curve over $\cB$
 with projection 
$\pi \colon \cC \to C\times \cB$. 
Note that $\cH^{\mathrm{greg}}=\cH$ by the condition (\ref{Bcirc}), and it is classical.

Let $\mathcal{N} \subset \Omega_{\cH}[-1]$
be the nilpotent cone, and 
\begin{align*}
    \LL_{\mathcal{N}}(\cH(\chi))_w \subset \LL(\cH(\chi))_w
\end{align*}
the subcategory of nilpotent singular supports. 
In this case, we have the following: 
\begin{lemma}
    \label{lem:nilp}
    We have 
    \begin{align*}
         \Coh_{\mathcal{N}}(\cH(\chi))=\mathrm{Perf}(\cH(\chi)).
    \end{align*}
    In particular, we have 
      \begin{align*}
         \LL_{\mathcal{N}}(\cH(\chi))_w=\LL(\cH(\chi))_w \cap \mathrm{Perf}(\cH(\chi)).
    \end{align*}
\end{lemma}
\begin{proof}
    For $b\in \cB$, let $\cC_b$ be the reduced spectral curve with 
    irreducible components $C_1, \ldots, C_k$.
    Then any torsion-free object $E \in \Coh^{\heartsuit}(\cC_b)$ 
    with rank one admits a filtration whose associated graded 
    is $\oplus_{i=1}^k F_i$ for a rank one torsion-free sheaf $F_i$ on $C_i$. 
    Since $\Hom_{\cC_b}(F_i, F_j)=0$ for $i\neq j$, any
automorphism of $E$ acts faithfully on the associated graded
$\bigoplus_i F_i$.
    Therefore 
    we have 
    \begin{align*}
        \Aut(E) \subset \Aut(\oplus_{i=1}^k F_i)\cong\mathbb{G}_m^k. 
    \end{align*}
    The last isomorphism follows from $\mathrm{Aut}(F_i)=\mathbb{G}_m$, as 
    $C_i$ is irreducible and $F_i$ is torsion free of rank one. 
    Since the Lie algebra of the torus does not have non-trivial 
    nilpotent elements, the nilpotent cone 
    $\mathcal{N} \subset \Omega_{\cH}[-1]$
    is the same as the zero section of $\Omega_{\cH}[-1] \to \cH$. Therefore the lemma holds from~\cite[Theorem~4.2.6]{AG} and noting that any compact
    object of $\QCoh(\cH(\chi))$ is a perfect complex.
\end{proof}

\subsection{Whittaker normalization}\label{subsec:whit}
We set 
\begin{align*}
    \chi_0 := \deg \pi_{*}\mathcal{O}_{\cC_b}=(r^2-r)(1-g).
\end{align*}
We have the following \textit{Hitchin section}
\begin{align}\label{sec:s}
    s \colon \rB \to \Hig_G(\chi_0). 
\end{align}
It is induced by the structure sheaf of the universal spectral curve, i.e. 
it sends $b\in \rB$ to the structure sheaf
$\mathcal{O}_{\cC_b}$. 
Note that its image lies in $\Hig_G(\chi_0)^{\mathrm{reg}}$. 
For the Arinkin sheaf $\mathrm{P}$ in (\ref{ArsheafP}), we have the following: 
\begin{lemma}\label{lem:pulls}
For the morphism 
\begin{align*}
    \id \times s \colon \Hig_G(w) \to \Hig_G(w)\times_{\rB} \Hig_G(\chi_0)^{\mathrm{reg}},
\end{align*}
we have $(\id \times s)^* \mathrm{P}\cong
\mathcal{O}_{\Hig_G(w)}$.
\end{lemma}
\begin{proof}
    The lemma immediately follows by the construction of Arinkin 
    sheaf $\cP^{\sharp}$ on $(\cH\times_{\cB} \cH)^{\sharp}$ for $\cB=\rB^{\mathrm{cl}}$, see Subsection~\ref{subsec:Arsheaf}. 
    Indeed for the Hitchin section $s \colon \cB \to \cH^{\mathrm{reg}}$, 
    and a rank one torsion-free sheaf $E\in \Coh(\cC_b)$, we have 
    \begin{align*}
        (\id \times s)^* \cP^{\sharp}|_{E}=\det \chi(E\otimes \mathcal{O}_{\cC_b})\otimes \det \chi(E)^{-1}\otimes \det \chi(\mathcal{O}_{\cC_b})^{-1}\otimes \det \chi(\mathcal{O}_{\cC_b})
    \end{align*}
    which is canonically identified with $k$. Therefore 
     $(\id \times s)^* \cP^{\sharp}\cong \mathcal{O}_{\cH}$.
     The lemma then follows by applying $\boxtimes \mathcal{O}_{k[-1]}$.
\end{proof}

The Hitchin section (\ref{sec:s}) is quasi-smooth, and it also factors through 
\begin{align}\notag
    s \colon \rB \stackrel{\overline{s}}{\to}
    \Hig_G(\chi_0)^{\mathrm{ss}} \stackrel{j}{\hookrightarrow} \Hig_G(\chi_0),
\end{align}
since $\mathcal{O}_{\cC_b} \in \Coh^{\heartsuit}(\cC_b)$ is semistable (indeed stable). 

\begin{lemma}\label{lem:s-left}
The functor 
\begin{align*}
    \overline{s}^* \colon \LL(\Hig_G(\chi_0)^{\mathrm{ss}})_{w'} \to \Coh(\rB)
\end{align*}
admits a left adjoint 
\begin{align}\label{sleft}
    \overline{s}_{!} \colon \Coh(\rB) \to \LL(\Hig_G(\chi_0)^{\mathrm{ss}})_{w'}.
\end{align}
\end{lemma}
\begin{proof}
Recall the diagram of good moduli spaces (\ref{dia:gmoduli}). 
    The composition 
    \begin{align*}
        \rB\stackrel{\overline{s}}{\to} \Hig_G(\chi_0)^{\mathrm{ss}} \to 
         \mathrm{H}_G(\chi_0)^{\mathrm{ss}}
    \end{align*}
    is a closed immersion, which further factors through the quasi-smooth
    closed immersion into the open subset $\mathrm{H}_G(\chi_0)^{\mathrm{st}} \subset \mathrm{H}_G(\chi_0)^{\mathrm{ss}}$
    \begin{align*}
        \overline{s}' \colon \rB \hookrightarrow  \mathrm{H}_G(\chi_0)^{\mathrm{st}}.
    \end{align*}
     Indeed it is a closed immersion between smooth varieties multiplied with $k[-1]$.
     
    The left vertical arrow in (\ref{dia:gmoduli}) is a $\mathbb{G}_m$-gerbe, 
    classified by a Brauer class $\beta$ on  $ \mathrm{H}_G(\chi_0)^{\mathrm{st}}$.
    Since $(\overline{s}')^*\beta$ is trivial on $\rB$, 
    we have the functor 
    \begin{align*}
       (\overline{s}')^* \colon 
        \Coh( \mathrm{H}_G(\chi_0)^{\mathrm{st}}, \beta^{w'}) \to \Coh(\rB).
    \end{align*}
    Here the left-hand side is the category of $\beta^{w'}$-twisted coherent sheaves on $\mathrm{H}_G(\chi_0)^{\mathrm{st}}$.
    Since $\overline{s}'$ is a closed immersion, it admits a left adjoint 
    \begin{align}\label{funct:left}
        (\overline{s}')_!\colon 
        \Coh(\rB) \to 
        \Coh( \mathrm{H}_G(\chi_0)^{\mathrm{st}}, \beta^{w'})
        \simeq\Coh(\Hig_G(\chi_0)^{\mathrm{st}})_{w'}.
    \end{align}
It is given by the formula 
    \begin{align}\label{formula:s!}(\overline{s}')_!=\overline{s}'_{*}[r^2(1-g)-1],
    \end{align}
    where the shift is determined by the relative dimension of $\overline{s}'$. 
    
    Since $\overline{s}'$ is a closed immersion whose image is also closed in $\mathrm{H}_G(\chi_0)^{\mathrm{ss}}$, the image of the 
    functor (\ref{funct:left}) naturally defines an object in the limit category of 
    $\Hig_G(\chi_0)^{\mathrm{ss}}$. 
    Therefore the functor (\ref{funct:left})
    determines a functor (\ref{sleft}), which is a left adjoint of $\overline{s}^*$ by the construction. 
\end{proof}

We have the pull-back functor 
\begin{align}\label{s:ast}
    s^* \colon \LL(\Hig_G(\chi_0))_{w'} \to 
    \Coh(\rB). 
\end{align}
By Lemma~\ref{lem:s-left}, 
the functor (\ref{s:ast}) admits left adjoint 
\begin{align*}
    s_{!} \colon \Coh(\rB) \stackrel{\overline{s}_!}{\to} \LL(\Hig_G(\chi_0)^{\mathrm{ss}})_{w'} \stackrel{j_{!}}{\hookrightarrow}
    \LL(\Hig_G(\chi_0))_{w'}
\end{align*}
where $j_{!}$ is the left adjoint of $j^*$, see Theorem~\ref{thm:left}.

Since the above constructions are linear over $\rB$, 
the same construction and the proof of Theorem~\ref{thm:left} in~\cite[Theorem~7.19]{PTlim} also apply after
the base change $\cB \hookrightarrow \rB$ given by (\ref{H:simple}). 
It gives the functor 
\begin{align}\label{sshrink}
    s_{!} \colon \Coh(\cB) \stackrel{\overline{s}_!}{\to} \LL(\cH(\chi_0)^{\mathrm{ss}})_{w'} \stackrel{j_{!}}{\hookrightarrow}
    \LL(\cH(\chi_0))_{w'}. 
\end{align}
The image of the above functor lies in the subcategory of 
nilpotent singular supports
\begin{align*}
    s_! \colon \Coh(\cB) \to \LL_{\mathcal{N}}(\cH(\chi_0))_{w'}
\end{align*}
since the image of $\overline{s}$ lies in the smooth locus of $\cH(\chi_0)$ and $j_{!}$ preserves nilpotent singular support condition, see~\cite[Section~8.10]{PTlim}. It gives a left adjoint of the 
functor 
\begin{align*}
    s^* \colon \LL_{\mathcal{N}}(\cH(\chi_0))_{w'} \to \Coh(\cB).
\end{align*}

We conjecture the following property of the functor $\Phi$ in (\ref{induce:Phi}), which corresponds to Whittaker normalization. 
\begin{conj}\label{conj:norm}
There is an isomorphism 
\begin{align*}
    \Phi(\mathcal{O}_{\cH(w)^{\mathrm{ss}}}) \cong s_{!}\mathcal{O}_{\cB} \in \LL_{\mathcal{N}}(\cH(\chi_0))_{w'}.
\end{align*}
\end{conj}

\begin{remark}\label{rmk:sshrink}
The object $s_{!}\mathcal{O}_{\cB}$ is more exotic than 
$\overline{s}_{!}\mathcal{O}_{\cB}$. Indeed in the situation of Example~\ref{exam:g_2}, 
the support of $\overline{s}_! \mathcal{O}_{\cB}$ on $h^{-1}(b)$ is just one point
corresponding to $\mathcal{O}_{\cC_b}$ in $\cP^{0, 0}$. However after applying $j_!$, 
the object $s_{!}\mathcal{O}_{\cB}=j_{!} \overline{s}_{!}\mathcal{O}_{\cB}$ is supported 
beyond $\cP^{0, 0}$; indeed for $l\geq 2$ and $1\leq m\leq l-1$, let $p_{l, m}$ be the point 
\begin{align*}
    p_{l, m}&={(\mathcal{O}_{C_1}((l-m)x_1+mx_2), \mathcal{O}_{C_2}(-mx_1+(m-l)x_2)}) \\ &\in \mathrm{Pic}^{l}(C_1)\times \mathrm{Pic}^{-l}(C_2).
\end{align*}
    Then one can calculate that 
    \begin{align*}
        \mathrm{Supp}(s_! \mathcal{O}_{\cB})\cap \cP^{l, -l}=
        \{p_{l, m} : 1\leq m\leq l-1\} \times \mathbb{G}_m \times \bgm,
    \end{align*}
    in particular $s_{!}\mathcal{O}_{\cB}$ spreads beyond the semistable locus. 
\end{remark}

\subsection{Wilson/Hecke compatibility under Whittaker normalization}
In the rest of this section, under Conjecture~\ref{conj:norm}, we show that the functor 
(\ref{induce:Phi}) restricts to the functor 
\begin{align*}
    \Coh_{\mathcal{N}}(\cH(w)^{\mathrm{ss}})_{-\chi'} \to 
    \LL_{\mathcal{N}}(\cH(\chi))_{w'}
\end{align*}
which is fully faithful, and indeed an equivalence under an additional hypothesis in Assumption~\ref{assum:Bcirc}. We will prove it using the compatibility of Wilson/Hecke operators in the previous section. 

For $d\geq 0$, we consider the functor induced by the iterated Wilson operators (\ref{funct:W})
applied to $\mathcal{O}_{\cH(w)^{\mathrm{ss}}}\in \Coh(\cH(w)^{\mathrm{ss}})_0$
\begin{align*}
    \mathrm{W}_d  \colon \Coh(S^{\times d}) \to \Coh(\cH(w)^{\mathrm{ss}})_{d}.
\end{align*}
We also consider the functor induced by iterated Hecke operators (\ref{hecke:greg}) applied to 
$s_{!}\mathcal{O}_{\cB}\in \LL_{\mathcal{N}}(\cH(\chi_0))_{w'}$, 
\begin{align}\label{funct:Hd}
    \mathrm{H}_d \colon \Coh(S^{\times d}) \to 
    \Coh(\cH(\chi_0-d))_{w'}.
\end{align}
We have the following corollary of the result in the 
previous section: 
\begin{cor}\label{cor:comWH}
Under Conjecture~\ref{conj:norm}, the following diagram commutes: 
\begin{align}\notag
    \xymatrix{
\Coh(S^{\times d}) \ar[r]^-{\mathrm{W}_d} \ar@{=}[d] & 
\Coh(\cH(w)^{\mathrm{ss}})_{d} \ar[d]^-{\Phi} \\
\Coh(S^{\times d}) \ar[r]^-{\mathrm{H}_d} & 
 \Coh(\cH(\chi_0-d))_{w'}.
    }
\end{align}
\end{cor}
\begin{proof}
    The corollary follows by applying Theorem~\ref{prop:WH} $d$-times. 
\end{proof}

The image of the functor $\mathrm{H}_d$ lies in the limit category by the following lemma: 
\begin{lemma}\label{lem:functHd}
    The image of the functor (\ref{funct:Hd}) lies in 
    \begin{align*}
        \LL_{\mathcal{N}}(\cH(\chi_0-d))_{w'} \subset \Coh(\cH(\chi_0-d))_{w'}.
    \end{align*}
\end{lemma}
\begin{proof}
As in Lemma~\ref{lem:HeckeL}, 
the Hecke operator in (\ref{hecke:greg})
    \begin{align}\label{hecke:H}
    \mathrm{H} \colon 
        \Coh(S) \otimes \Coh(\cH) \to \Coh(\cH)
    \end{align}
    restricts to the functor of limit categories
    \begin{align}\label{hecke:H2}
        \mathrm{H}\colon   \Coh(S) \otimes \LL(\cH(\chi))_{w'} \to \LL(\cH(\chi-1))_{w'}.
    \end{align}
Indeed as the Hecke operator is linear over $\cB$, the argument of Lemma~\ref{lem:HeckeL} goes through by replacing $\Hig_G$ with preimages of 
open subsets of $\rB$. Moreover the trivial derived structure $k[-1]$
does not affect the $\mathbb{G}_m$-weights in the definition of limit categories, therefore the functor $\mathrm{H}$ restricts to the functor (\ref{hecke:H2}). 

    Since $\ev_1$ in the diagram (\ref{redHecke}) is quasi-smooth and proper by 
    Lemma~\ref{lem:qsmooth}, the push-forward $\ev_{1*}$ preserves perfect complexes, see~\cite[Theorem~0.3]{Toenpush}. Therefore 
    the functor (\ref{hecke:H}) also restricts to the functor 
    \begin{align*}
         \mathrm{H} \colon 
        \Coh(S) \otimes \mathrm{Perf}(\cH) \to \mathrm{Perf}(\cH). 
    \end{align*}
    It follows that, by Lemma~\ref{lem:nilp}, the functor (\ref{hecke:H2}) further restricts to the functor 
    \begin{align}\label{funct:heckeN}
        \mathrm{H}\colon   \Coh(S) \otimes \LL_{\mathcal{N}}(\cH(\chi))_{w'} \to \LL_{\mathcal{N}}(\cH(\chi-1))_{w'}.
    \end{align}
    Iterating this $d$-times, we obtain the lemma. 
\end{proof}

Together with the compact generation under Wilson operators proved in the next subsection, the Arinkin sheaf determines, assuming Conjecture~\ref{conj:norm}, a functor to the limit category that also satisfies Wilson/Hecke compatibility:
\begin{cor}\label{cor:rest}
Under Conjecture~\ref{conj:norm}, the
functor $\Phi$ restricts to the functor 
\begin{align}\notag
    \Phi \colon \Coh_{\mathcal{N}}(\cH(w)^{\mathrm{ss}})_{-\chi'} \to 
    \LL_{\mathcal{N}}(\cH(\chi))_{w'}, 
\end{align}
and we have the commutative diagram 
\begin{align}\label{dia:WH1.5}
    \xymatrix{
\Coh(S^{\times d}) \ar[r]^-{\mathrm{W}_d} \ar@{=}[d] & 
\Coh_{\mathcal{N}}(\cH(w)^{\mathrm{ss}})_{d} \ar[d]^-{\Phi} \\
\Coh(S^{\times d}) \ar[r]^-{\mathrm{H}_d} & 
 \LL_{\mathcal{N}}(\cH(\chi_0-d))_{w'}.
    }
\end{align}
\end{cor}
\begin{proof}
In Corollary~\ref{cor:compact}, we show that $\QCoh(\cH(w)^{\mathrm{ss}})$
is compactly generated by the images of Wilson operators applied to $\mathcal{O}_{\cH(w)^{\mathrm{ss}}}$, hence 
$\mathrm{Perf}(\cH(w)^{\mathrm{ss}})$ is split generated by the images of 
these objects. Then 
  the corollary follows by combining Corollary~\ref{cor:comWH}
  and Lemma~\ref{lem:functHd}. 
\end{proof}

We next consider the compatibility with the right adjoints of Wilson/Hecke operators:
\begin{prop}\label{prop:adWH}
Under Conjecture~\ref{conj:norm}, there is a commutative 
diagram 
\begin{align}\notag
    \xymatrix{
\Coh_{\mathcal{N}}(\cH(w)^{\mathrm{ss}})_{d}\ar[r]^-{\mathrm{W}_d^R} \ar[d]_-{\Phi} & 
  \QCoh(S^{\times d})  \ar@{=}[d] \\
 \LL_{\mathcal{N}}(\cH(\chi_0-d))_{w'}\ar[r]^-{\mathrm{H}_d^R} & 
\QCoh(S^{\times d}).
    }
\end{align}
Here $\mathrm{W}_d^R$, $\mathrm{H}_d^R$ are right adjoint functors 
of $\mathrm{W}_d$, $\mathrm{H}_d$ in the diagram (\ref{dia:WH1.5}).
\end{prop}
\begin{proof}
By Proposition~\ref{prop:comad} and noting $\cH^{\mathrm{greg}}=\cH$,
the following diagram commutes 
   \begin{align}\label{dia:PhiS}
\xymatrix{
\Coh(\cH^{\mathrm{ss}}) 
  \ar[r]^-{\widetilde{\mathrm{W}}^{R}} 
  \ar[d]_-{\Phi} 
& \Coh(\cH^{\mathrm{ss}}) \otimes \Coh(S)
  \ar[d]^-{\Phi\otimes \id} \\
\Coh(\cH) 
  \ar[r]^-{\widetilde{\mathrm{H}}^{R}} 
& \Coh(\cH)\otimes \Coh(S). 
}
\end{align}
Similarly to (\ref{funct:heckeN}), the bottom functor restricts to the functor 
\begin{align*}
    \widetilde{\mathrm{H}}^R\colon \LL_{\mathcal{N}}(\cH(\chi))_{w'} \to \LL_{\mathcal{N}}(\cH(\chi+1))_{w'}\otimes \Coh(S).
\end{align*}
Then by Corollary~\ref{cor:rest}, the diagram (\ref{dia:PhiS}) restricts 
   to the commutative diagram 
    \begin{align}\notag
\xymatrix{
\Coh_{\mathcal{N}}(\cH(w)^{\mathrm{ss}})_{-\chi'} 
  \ar[r]^-{\widetilde{\mathrm{W}}^{R}} 
  \ar[d]_-{\Phi} 
& \Coh_{\mathcal{N}}(\cH(w)^{\mathrm{ss}})_{-\chi'-1} \otimes \Coh(S)
  \ar[d]^-{\Phi\otimes \id} \\
\LL_{\mathcal{N}}(\cH(\chi))_{w'} 
  \ar[r]^-{\widetilde{\mathrm{H}}^{R}} 
& \LL_{\mathcal{N}}(\cH(\chi+1))_{w'}\otimes \Coh(S). 
}
\end{align}
By applying it $d$-times, we have the commutative diagram 
   \begin{align}\label{apply:d1}
\xymatrix{
\Coh_{\mathcal{N}}(\cH(w)^{\mathrm{ss}})_{-\chi'} 
  \ar[r]^-{\widetilde{\mathrm{W}}_d^{R}} 
  \ar[d]_-{\Phi} 
& \Coh_{\mathcal{N}}(\cH(w)^{\mathrm{ss}})_{-\chi'-d} \otimes \Coh(S^{\times d})
  \ar[d]^-{\Phi\otimes \id} \\
\LL_{\mathcal{N}}(\cH(\chi))_{w'} 
  \ar[r]^-{\widetilde{\mathrm{H}}_d^{R}} 
& \LL_{\mathcal{N}}(\cH(\chi+d))_{w'}\otimes \Coh(S^{\times d}). 
}
\end{align}

We also have the commutative diagram by Lemma~\ref{lem:pulls}
 \begin{align}\label{apply:d2}
\xymatrix{
\Coh_{\mathcal{N}}(\cH(w)^{\mathrm{ss}})_{0} 
  \ar[r]^-{h_{*}} 
  \ar[d]_-{\Phi} 
& \Coh(\cB)
  \ar@{=}[d] \\
\LL_{\mathcal{N}}(\cH(\chi_0))_{w'} 
  \ar[r]^-{s^*}
& \Coh(\cB). 
}
\end{align}
The bottom horizontal arrow is isomorphic to 
\begin{align}\label{def:s*}
s^*(-)\cong 
    \Hom(s_{!}\mathcal{O}_{\cB}, -)  \colon 
    \LL_{\mathcal{N}}(\cH(\chi_0))_{w'} \to \Coh(\cB)
\end{align}
since $s_{!}$ is left adjoint to $s^*$, see Lemma~\ref{lem:s-left}.
It is straightforward to check that the 
compositions 
\begin{align*}
    \mathrm{W}_d^R \colon 
    \Coh_{\mathcal{N}}(\cH(w)^{\mathrm{ss}})_d
    &\stackrel{\widetilde{\mathrm{W}}_d^R}{\to}
     \Coh_{\mathcal{N}}(\cH(w)^{\mathrm{ss}})_0\otimes \Coh(S^{\times d}) \\
     &\stackrel{\Gamma(h_{*}(-))\otimes \id}{\to}
     \QCoh(S^{\times d}),
     \end{align*}
     is a right adjoint functor of $\mathrm{W}_d$. Similarly the 
     composition 
     \begin{align}\label{def:HcR}
 \mathrm{H}_d^R \colon 
    \LL_{\mathcal{N}}(\cH(\chi_0-d))_{w'}
    &\stackrel{\widetilde{\mathrm{H}}_d^R}{\to}
     \LL_{\mathcal{N}}(\cH(\chi_0))_{w'}\otimes \Coh(S^{\times d})\\
     &\notag \stackrel{\Hom(s_{!}\mathcal{O}_{\cB}, -)\otimes \id}{\to}
     \QCoh(S^{\times d})
\end{align}
is a right adjoint functor of $\mathrm{H}_d$. 
Therefore the proposition follows from the commutative diagrams (\ref{apply:d1}), (\ref{apply:d2}). 
\end{proof}

Fix $c\in C$. Then for $\chi \in \mathbb{Z}$ and 
$d>0$ such that $r|(d+\chi)$, there is an equivalence, 
\begin{align*}
    \Coh_{\mathcal{N}}(\cH(w)^{\mathrm{ss}})_{-\chi'}
    \stackrel{\sim}{\to}  \Coh_{\mathcal{N}}(\cH(w)^{\mathrm{ss}})_{d}, 
\end{align*}
given as follows, see Remark~\ref{rmk:period} 
\begin{align*}
    (-)\mapsto (-)\otimes \det \cF_c^{\otimes(\chi'+d)/r}.
\end{align*}
By composing with $\mathrm{W}_d$, we obtain the functor 
\begin{align*}
\mathrm{W}_{d, c} \colon 
\Coh(S^{\times d}) \stackrel{\mathrm{W}_d}{\to}
     \Coh_{\mathcal{N}}(\cH(w)^{\mathrm{ss}})_{d}
     \simeq 
     \Coh_{\mathcal{N}}(\cH(w)^{\mathrm{ss}})_{-\chi'}
\end{align*}
with right adjoint 
\begin{align*}
    \mathrm{W}_{d, c}^R \colon 
   \Coh_{\mathcal{N}}(\cH(w)^{\mathrm{ss}})_{-\chi'}\simeq 
   \Coh_{\mathcal{N}}(\cH(w)^{\mathrm{ss}})_{d}
   \stackrel{\mathrm{W}_d^R}{\to} \QCoh(S^{\times d}).
\end{align*}

There is also an equivalence 
\begin{align*}
    \LL_{\mathcal{N}}(\cH(\chi))_{w'}
    \stackrel{\sim}{\to}  \LL_{\mathcal{N}}(\cH(\chi_0-d))_{w'}
\end{align*}
induced by the isomorphism, see Remark~\ref{rmk:period}
\begin{align*}\cH(\chi) \stackrel{\cong}{\to} \cH(\chi_0-d), \ (-) \mapsto 
(-)\otimes \mathcal{O}_C((\chi_0-d-\chi)/r \cdot c). 
\end{align*}
By composing with $\mathrm{H}_d$, we obtain the functor 
\begin{align*}
\mathrm{H}_{d, c} \colon 
\Coh(S^{\times d}) \stackrel{\mathrm{H}_d}{\to}
     \LL_{\mathcal{N}}(\cH(\chi_0-d))_{w'}
     \simeq 
     \LL_{\mathcal{N}}(\cH(\chi))_{w'}
\end{align*}
with right adjoint 
\begin{align}\label{HRdc}
   \mathrm{H}_{d, c}^R \colon 
   \LL_{\mathcal{N}}(\cH(\chi))_{w'}
   \simeq \LL_{\mathcal{N}}(\cH(\chi_0-d))_{w'}
   \stackrel{\mathrm{H}_d^R}{\to}
     \QCoh(S^{\times d}).
\end{align}

As a corollary of the above arguments, we have the following: 
\begin{cor}\label{cor:comC}
Under Conjecture~\ref{conj:norm}, for $d>0$ with 
$r|(d+\chi)$ and $c\in C$, we have the commutative diagrams 
\begin{align*}
   \xymatrix{
\Coh(S^{\times d}) \ar[r]^-{\mathrm{W}_{d, c}} \ar@{=}[d] & 
\Coh_{\mathcal{N}}(\cH(w)^{\mathrm{ss}})_{-\chi'} \ar[d]^-{\Phi} \\
\Coh(S^{\times d}) \ar[r]^-{\mathrm{H}_{d, c}} & 
 \LL_{\mathcal{N}}(\cH(\chi))_{w'}
    } \
   \xymatrix{
\Coh_{\mathcal{N}}(\cH(w)^{\mathrm{ss}})_{-\chi'}\ar[r]^-{\mathrm{W}_{d, c}^R} \ar[d]_-{\Phi} & 
 \ar@{=}[d] \QCoh(S^{\times d})\\
 \LL_{\mathcal{N}}(\cH(\chi))_{w'}\ar[r]^-{\mathrm{H}_{d, c}^R} & 
\QCoh(S^{\times d}).
    }   
    \end{align*}
\end{cor}

Let 
\begin{align*}
    \Phi \colon \IndCoh_{\mathcal{N}}(\cH(w)^{\mathrm{ss}})_{-\chi'}
    \to \IndL_{\mathcal{N}}(\cH(\chi))_{w'}
\end{align*}
be the functor defined by taking the ind-completion of the functor in 
Corollary~\ref{cor:rest}, and using the compact generations of both sides, see Theorem~\ref{thm:cpt}, Remark~\ref{rmk:equiv}. 
Since it preserves the compact objects, 
by the adjoint functor theorem it admits a continuous 
right adjoint 
\begin{align*}
    \Phi^R \colon \IndL_{\mathcal{N}}(\cH(\chi))_{w'}
    \to \IndCoh_{\mathcal{N}}(\cH(w)^{\mathrm{ss}})_{-\chi'}.
\end{align*}

\begin{cor}\label{cor:PhiR}
 Under Conjecture~\ref{conj:norm}, for $d>0$ with 
$r|(d+\chi)$ and $c\in C$, we have the commutative diagrams, 
where each functor is continuous
   \begin{align}\notag
   \xymatrix{
\IndCoh_{\mathcal{N}}(\cH(w)^{\mathrm{ss}})_{-\chi'}\ar[r]^-{\mathrm{W}_{d, c}^R} \ar@<-0.5ex>[d]_-{\Phi} & 
 \ar@{=}[d] \QCoh(S^{\times d})\\
 \IndL_{\mathcal{N}}(\cH(\chi))_{w'}\ar[r]^-{\mathrm{H}_{d, c}^R} 
   \ar@<-0.5ex>[u]_-{\Phi^R}& 
\QCoh(S^{\times d}).
    }   
    \end{align} 
\end{cor}
\begin{proof}
    The commutativity of $\Phi$ with $\mathrm{W}_{d, c}^R, \mathrm{H}_{d, c}^R$ follows by taking the ind-completions 
    of the right diagram in Corollary~\ref{cor:comC}.
    The commutativity of $\Phi^R$ with $\mathrm{W}_{d, c}^R, \mathrm{H}_{d, c}^R$ follows by taking the right adjoints 
    of the ind-completions of the left diagram in Corollary~\ref{cor:comC}. 
\end{proof}

\subsection{Compact generation by Wilson operators}
Here we show the compact generation of $\QCoh(\cH(w)^{\mathrm{ss}})$ under Wilson operators. 

For $A\in \Coh(S)$, let 
\begin{align*}
    \mathrm{W}_A \colon \Coh(\cH(w)^{\mathrm{ss}})_* \to \Coh(\cH(w)^{\mathrm{ss}})_{*+1}
\end{align*}
be the corresponding Wilson operator. 
It is given by 
\begin{align*}
    \mathrm{W}_A(-)=(-)\otimes p_{\cH*}(p_S^*A \otimes \cE)
\end{align*}
where $p_S, p_{\cH}$ are the projections from $S\times \cH(w)^{\mathrm{ss}}$.
By the construction, it preserves 
perfect complexes
\begin{align*}
    \mathrm{W}_A \colon \mathrm{Perf}(\cH(w)^{\mathrm{ss}})_* \to \mathrm{Perf}(\cH(w)^{\mathrm{ss}})_{*+1}.
\end{align*}
We fix $c\in C$ and consider the line bundle 
\begin{align*}
   L_{c}:=\det(\cF_c)\in \mathrm{Pic}(\cH(w)^{\mathrm{ss}}). 
\end{align*}
It has $\mathbb{G}_m$-weight $r$, so induces the equivalence 
\begin{align*}
    \otimes L_c \colon \mathrm{Perf}(\cH(w)^{\mathrm{ss}})_* \stackrel{\sim}{\to} \mathrm{Perf}(\cH(w)^{\mathrm{ss}})_{*+r}.
\end{align*}
We have the following proposition:
\begin{prop}\label{prop:cgen}
The objects of the form 
\begin{align}\label{Wperf}
    \otimes L_c^{-m}\circ \mathrm{W}_{A_d} \circ \cdots \circ \mathrm{W}_{A_1}(\mathcal{O}_{\cH(w)^{\mathrm{ss}}}) \in \mathrm{Perf}(\cH(w)^{\mathrm{ss}})
\end{align}
for $m\geq 0$ and $A_1, \ldots, A_d \in \Coh(S)$
compactly generate $\mathrm{QCoh}(\cH(w)^{\mathrm{ss}})$.
\end{prop}
\begin{proof}
Let
\begin{align*}
    \mathbf{W} \subset \mathrm{Perf}(\cH(w)^{\mathrm{ss}})
\end{align*}
be the subcategory split generated by objects of the form (\ref{Wperf}). 
Since $\cH(w)^{\mathrm{ss}}$ is QCA, an object in $\mathbf{W}$ is a compact object in $\QCoh(\cH(w)^{\mathrm{ss}})$ by~\cite[Corollary~1.4.3]{MR3037900}.
Therefore it is enough to show that, if $M\in \mathrm{QCoh}(\cH(w)^{\mathrm{ss}})$
satisfies 
\begin{align}\notag
\Hom(\mathbf{W}, M)=0, 
\end{align}
then $M=0$. 

    Let $T\in \Coh(C)$ be a locally free sheaf such that 
    \begin{align*}
        (\rank(T), \deg(T))=(r, -w+r(g-1)).
    \end{align*}
    Then for any Higgs bundle $(F, \theta)$ with $\rank(F)=r$ and $\deg(F)=w$,
    we have 
    $\chi(T\otimes F)=0$. Let $\pi \colon S \to C$ be the projection. Then the line bundle
    \begin{align*}
    \mathcal{L}_T:=
        \det p_{\cH*}(p_C^*T \otimes \cF) =\det p_{\cH*}(p_S^*\pi^*T \otimes \cE) \in \mathrm{Pic}(\cH(w)^{\mathrm{ss}})
    \end{align*}
    has $\mathbb{G}_m$-weight zero, and it descends to an ample line bundle on the 
    good moduli space,  see~\cite[Section~8.4]{Hu}
    \begin{align}\label{pmoduli}
       q_{\cH} \colon  \cH(w)^{\mathrm{ss}} \to \mathrm{H}(w)^{\mathrm{ss}}.
    \end{align}
    Namely we have 
    \begin{align*}
        \mathcal{L}_T=q_{\cH}^* \mathrm{L}_T, \ \mathrm{L}_T \in \mathrm{Pic}(\mathrm{H}(w)^{\mathrm{ss}})
    \end{align*}
for an ample line bundle $\mathrm{L}_T$. By the ampleness of $\mathrm{L}_T$, there is $\ell \geq 0$ 
such that its powers 
\begin{align}\label{cpt:L}\{\mathrm{L}_T^{\otimes k}\}_{0\leq k\leq \ell}
\subset \mathrm{Perf}(\mathrm{H}(w)^{\mathrm{ss}})
\end{align}
compactly generate $\mathrm{QCoh}( \mathrm{H}(w)^{\mathrm{ss}})$, see~\cite[Lemma~4.2.2, 4.2.4]{Bonvan}. 

For $m\gg 0$, the object 
\begin{align*}
    p_{\cH*}(p_C^*(T\otimes \mathcal{O}_C(mc))\otimes \cF) \in \mathrm{Perf}(\cH(w)^{\mathrm{ss}})
\end{align*}
is a vector bundle of rank $mr^2$, and 
\begin{align*}
    \det p_{\cH*}(p_C^*(T\otimes \mathcal{O}_C(mc)) \otimes \cF)
    \hookrightarrow  (p_{\cH*}(p_C^*(T\otimes \mathcal{O}_C(mc))\otimes \cF))^{\otimes mr^2}
\end{align*}
is a direct summand. Moreover we have 
\begin{align*}
     \det p_{\cH*}(p_C^*(T\otimes \mathcal{O}_C(mc)) \otimes \cF)=\mathcal{L}_T\otimes L_c^{\otimes mr}.
\end{align*}
It follows that we have the direct summand 
\begin{align}\label{dsum}
    (-)\otimes \mathcal{L}_T \subset^{\oplus} \overbrace{\mathrm{W}_{\pi^*(T\otimes \mathcal{O}_C(mc))}\circ \cdots \circ  \mathrm{W}_{\pi^*(T\otimes \mathcal{O}_C(mc))}}^{mr^2}(-) \otimes 
    L_c^{-mr}.
\end{align}

For $W\in \mathbf{W}$, we have $W\otimes \mathcal{L}_T^{\otimes k} \in \mathbf{W}$
for any $k\in \mathbb{Z}$ by (\ref{dsum}). Therefore $\Hom(\mathbf{W}, M)=0$ implies that 
\begin{align*}
    \Hom(W\otimes q_{\cH}^*\mathrm{L}_T^{\otimes k}, M)=
    \Gamma(q_{\cH*}(W^{\vee}\otimes M)\otimes \mathrm{L}_T^{\otimes -k})=0
\end{align*}
for any $k$. By the compact generation by the objects (\ref{cpt:L}), we conclude that 
\begin{align}\label{vanish:W}
    q_{\cH*}(W^{\vee}\otimes M)=0
\end{align}
for any $W\in \mathbf{W}$. 

Now each point $y\in \mathrm{H}(w)^{\mathrm{ss}}$
corresponds to a polystable Higgs bundle 
\begin{align*}
E_y=
\bigoplus_{i=1}^a V_i \otimes (F_i, \theta_i)
\end{align*}
where $\deg F_i/\rank(F_i)=w/r$ and $E_i=(F_i, \theta_i)$ is stable. 
We have the moment map 
\begin{align*}
    \mu \colon \Ext^1(E_y, E_y)=\bigoplus_{i, j}\Hom(V_i, V_j)\otimes \Ext^1(E_i, E_j)
    \to \Ext^2(E_y, E_y)=\mathfrak{g}(v)
\end{align*}
given by $x\mapsto x\cup x$, 
where $\mathfrak{g}(v)$ is the Lie algebra of 
$G(v)=\prod_{i=1}^a \GL(V_i)$.
Then formally locally at $y$, 
the good moduli space map (\ref{pmoduli}) is 
isomorphic to the following quotient stack, see~\cite[Remark~4.2]{PTK3}
\begin{align*}
    \mu^{-1}(0)/G(v) \to \mu^{-1}(0)\ssslash G(v).
\end{align*}
Let $f_{c}:=\mathcal{O}_{\pi^{-1}(c)}\in \Coh^{\heartsuit}(S)$. 
  The object 
  \begin{align*}
      p_{\cH*}(\cF_{c})
      =\mathrm{W}_{f_{c}}(\mathcal{O}_{\cH(w)^{\mathrm{ss}}})\in \mathbf{W}
      \subset \mathrm{Perf}(\cH(w)^{\mathrm{ss}})
  \end{align*}
  is, on the above formal local model, isomorphic to 
  \begin{align}\label{cpt:V}
      \bigoplus_{i=1}^aV_i^{\oplus \rank(F_i)}\otimes \mathcal{O}_{\mu^{-1}(0)/G(v)} \subset \mathrm{Perf}(\mu^{-1}(0)/G(v)). 
  \end{align}
  Its determinant 
  \begin{align}\label{determinant}
      \bigotimes_{i=1}^a (\det V_i)^{\otimes \mathrm{rank}(F_i)} \otimes 
      \mathcal{O}_{\mu^{-1}(0)/G(v)}
  \end{align}
  corresponds to $L_c$ on the above local model. 
  Since $\mu^{-1}(0)$ is affine, $\mathrm{Perf}(\mu^{-1}(0)/G(v))$ is split generated 
  by $\mathrm{Perf}(BG(v))$ by the pull-back of 
  $\mu^{-1}(0)/G(v) \to BG(v)$. Then the tensor products of (\ref{cpt:V}) 
  up to integer twists of (\ref{determinant}) split generate 
$\mathrm{Perf}(\mu^{-1}(0)/G(v))$, 
hence give compact generators of $\QCoh(\mu^{-1}(0)/G(v))$. 
  Therefore the vanishing (\ref{vanish:W}) implies that 
$M=0$ formally locally at $y$. Since it holds for any $y$, we have $M=0$. 
\end{proof}

\begin{remark}
    \label{rmk:dual}
    Alternatively, for $A\in \Coh(S)$, consider
    \begin{align*}
\mathrm{W}_A^R(-)=(-)\otimes p_{\cH*}(p_{S}^{*}(A)\otimes \cE)^{\vee}.
    \end{align*}
    Then
    Proposition~\ref{prop:cgen} implies that the objects of the form 
    \begin{align*}
        \mathrm{W}_{f_{c}}^{R}\circ \cdots \circ \mathrm{W}_{f_{c}}^{R}
        \circ \mathrm{W}_{A_d}\circ \cdots \circ \mathrm{W}_{A_1}
        (\mathcal{O}_{\cH(w)^{\mathrm{ss}}})
    \end{align*}
    compactly generates $\mathrm{QCoh}(\cH(w)^{\mathrm{ss}})$.
\end{remark}

From the proof of the above proposition, we can take $d$ to 
be bounded above. Therefore we have the following corollary: 
\begin{cor}\label{cor:compact}
    For $d\gg 0$, the image of the functor 
    \begin{align*}
        \Coh(S^{\times d}) \to \mathrm{Perf}(\cH(w)^{\mathrm{ss}})_d
    \end{align*}
    given by 
    \begin{align*}
        A_1\boxtimes \cdots \boxtimes A_d \mapsto 
        \mathrm{W}_{A_d}\circ \cdots \circ \mathrm{W}_{A_1}(\mathcal{O}_{\cH(w)^{\mathrm{ss}}})
    \end{align*}
    compactly generates $\mathrm{QCoh}(\cH(w)^{\mathrm{ss}})_d$.
\end{cor}

\begin{remark}\label{rmk:cpt2}
    The proof of Proposition~\ref{prop:cgen} and Corollary~\ref{cor:compact} work without the assumption 
    $\cB\subset \rB^{\mathrm{red, cl}}$. Indeed the same result holds for $\cB=\rB$. 
\end{remark}

As a corollary of the compact generation by Wilson operators, 
we have the conservativity of the right adjoint functors 
of Wilson operators in Corollary~\ref{cor:PhiR} for $d\gg 0$. 
\begin{cor}\label{cor:cons1}
For $A\in \IndCoh_{\mathcal{N}}(\cH(w)^{\mathrm{ss}})_{-\chi'}$, 
suppose that $\mathrm{W}_{d, c}^R(A)=0$ for $d \gg 0$
with $r|(d+\chi)$ and any $c\in C$. 
    Then $A=0$. 
\end{cor}
\begin{proof}
    This corollary is a direct consequence of Corollary~\ref{cor:compact} and Lemma~\ref{lem:cptC}, 
    noting the following from Lemma~\ref{lem:nilp}
    under the assumption $\cB \subset \rB^{\mathrm{red, cl}}$
    \begin{align*}
        \IndCoh_{\mathcal{N}}(\cH(w)^{\mathrm{ss}})_{-\chi'}=
        \QCoh(\cH(w)^{\mathrm{ss}})_{-\chi'}.
    \end{align*}
\end{proof}
We have used the following obvious lemma: 
\begin{lemma}\label{lem:cptC}
Let $\cC_1, \cC_2$ be compactly generated dg-categories 
and $F \colon \cC_1 \to \cC_2$ be a continuous functor which 
preserves compact objects $\cC_i^{\mathrm{cp}}$. Let $F^R \colon \cC_2 \to \cC_1$ be the continuous right adjoint of $F$.  
Then $\cC_2$ is compactly generated by $F(\cC_1^{\mathrm{cp}})$ if and only if $F^R$ is conservative. 
\end{lemma}
\begin{proof}
    The lemma follows that, for $A\in \cC_2$, we have 
    $\Hom(F(\cC_1^{\mathrm{cp}}), A)=0$ if and only if 
    $\Hom(\cC_1^{\mathrm{cp}}, F^R(A))=0$, if and only if 
    $F^R(A)=0$ by the compact generation of $\cC_1$. 
\end{proof}

\subsection{Compact generation by iterated Hecke operators}
We prove similar statements as in the previous subsection for the Hecke operators. However we need to impose the following condition for the open subset (\ref{Bcirc}):  
\begin{assum}\label{assum:Bcirc}
For $b\in \cB$, 
let $x\in \cC_b$ be a singular point
and $I_x \subset \widehat{\mathcal{O}}_{\cC_b, x}$ the conductor ideal. 
Then we assume that $\widehat{\mathcal{O}}_{\cC_b, x}/I_x$ is a chain ring for any $b$ and $x$, i.e.  
it is of the form $k[t]/t^n$ for some $n \geq 1$. 
\end{assum}
\begin{example}\label{exam1}
(1) If $G=\GL_2$, then $\cB=\rB^{\mathrm{red, cl}}$ satisfies the above condition. 
Indeed in this case, since $\cC_b \to C$ is a double cover, the ring $\widehat{\mathcal{O}}_{\cC_b, x}$
is of the form $k[[x, y]]/(y^2-x^m)$ for some $m\geq 2$. 
In particular, 
$\widehat{\mathcal{O}}_{\cC_b, x}/I_x$ is isomorphic to $k[t]/(t^{m/2})$ when $m$ is even, 
and $k[t]/(t^{(m-1)/2})$ when $m$ is odd, as a ring. 

(2) If $\cC_b$ has at worst nodal singularities, then 
it satisfies the condition in Assumption~\ref{assum:Bcirc}. 
\end{example}

\begin{remark}\label{rmk:ideal}
    The condition in Assumption~\ref{assum:Bcirc} is equivalent to 
    that $\widehat{\mathcal{O}}_{\cC_b, x}/I_x$ has only a finite number of 
    ideals. This property will be used in the proof of Lemma~\ref{lem:pq}.
\end{remark}

Contrary to the case of Wilson operators, we first prove that the right adjoints of Hecke operators are conservative under Assumption~\ref{assum:Bcirc}: 
\begin{prop}\label{prop:cons2}
For $A \in \IndL_{\mathcal{N}}(\cH(\chi))_{w'}$, 
suppose that $\mathrm{H}_{d, c}^R(A)=0$ for $d\gg 0$
with $r|(d+\chi)$ and any $c\in C$. Then,  
under Assumption~\ref{assum:Bcirc}, $A=0$. 
\end{prop}
\begin{proof}
Let $\cC \to \cB$ be the universal spectral curve over 
$\cB$. 
By the construction, the functor $\mathrm{H}_{d, c}^R$ factors through 
\begin{align}\label{factor}
    \QCoh((\cC)^{\times_{\cB}d}) \stackrel{i_{*}}{\to}\QCoh(\cB\times S^{\times d})
    \to \QCoh(S^{\times d})
\end{align}
where the first functor is the push-forward along the closed immersion 
\begin{align*}
    i \colon (\cC)^{\times_{\cB}d} \hookrightarrow \cB\times S^{\times d}
\end{align*}
and the second functor is the push-forward along the second projection. 
Both functors in (\ref{factor}) are conservative; the first functor is conservative as $i$ is a closed immersion and the second one is conservative as $\cB$ is quasi-affine. 
Therefore its composition is also conservative. 
Then it is enough to show the claim by replacing $\mathrm{H}_{d, c}^R$ with 
\begin{align*}
    \widehat{\mathrm{H}}_{d, c}^R \colon \IndL_{\mathcal{N}}(\cH(\chi))_{w'} \to 
     \QCoh((\cC)^{\times_{\cB}d}).
\end{align*}

Let $\cHecke^{(d)}$ be the correspondence of the $d$-times iterated Hecke operator; 
\begin{align*}
    \cHecke^{(d)}=\overbrace{\cHecke\times_{\ev_1, \cH, \ev_2} \cHecke \times_{\ev_1, \cH, \ev_2} \cdots \times_{\ev_1, \cH, \ev_2}\cHecke}^d 
\end{align*}
where $\cHecke$ is as in the diagram (\ref{redHecke}) (with notation `greg' omitted). 
We have the following diagram 
\begin{align*}
    \xymatrix{
\mathcal{U} \ar[d]_-{q} \inclusion \diasquare  \ar@/^18pt/[rr]^-{p}&
\cHecke^{(d)} \ar[d]^-{((\ev_3)^d, \ev_2)} \ar[r]^-{\ev_1} & 
\cH(\chi_0-d) \underset{\otimes \mathcal{O}(ac)}{\cong} \cH(\chi) \\
(\cC)^{\times_{\cB}d} \inclusion_-{\id\times s}
& (\cC)^{\times_{\cB}d}\times_{\cB}\cH(\chi_0). &
    }
\end{align*}
Here the right isomorphism is as in Remark~\ref{rmk:period}, where $a$ is given by 
\begin{align}\label{denote:a}a=\frac{\chi-\chi_0+d}{r}\in \mathbb{Z}, 
\ r|(d+\chi), \ 
d\gg 0.
\end{align}
The stack $\mathcal{U}$ over $b\in \cB$ consists of filtrations in $\Coh^{\heartsuit}(\cC_b)$
\begin{align}\label{filt:F}
    F_0 \subset F_1 \subset \cdots \subset F_d=\mathcal{O}_{\cC_b}
\end{align}
such that $F_i/F_{i-1} \cong \mathcal{O}_{x_i}$ for $x_i \in \cC_b$, and the 
map $p$ sends (\ref{filt:F}) to $F_0(ac):=F_0\otimes \pi^* \mathcal{O}_C(ac)$. 
By the construction of $\mathrm{H}_{d, c}^R$ in (\ref{def:HcR}), (\ref{HRdc}) and noting (\ref{def:s*}), 
we have 
\begin{align*}
    \widehat{\mathrm{H}}_{d, c}^R(-)=(\id\times s)^*((\ev_3)^d, \ev_2)_{*}(\ev_1)^{!}(-).
\end{align*}
By the base change, it is enough to show that for $A \in \QCoh(\cH(\chi))$ satisfying 
$q_{*}p^* A=0$ for $d\gg 0$, we have $A=0$. 
Here we have used the fact that $(\ev_1)^!$ and $(\ev_1)^*$ differ by 
taking the tensor line bundle (up to shift) pulled back via $((\ev_3)^d, \ev_2)$, see Lemma~\ref{lem:omegaev}. 

We define 
\begin{align}\label{Hk}
    \cH(\chi)_{\leq k} \subset \cH(\chi)
\end{align}
to be the substack whose fiber at $b\in \cB$ consists of 
$E \in \Coh^{\heartsuit}(\cC_b)$ which admits an exact sequence 
\begin{align}\label{exact:ELQ}
    0\to E \to \mathcal{L} \to Q \to 0
\end{align}
where $\mathcal{L}$ is a line bundle on $\cC_b$ and $Q$ is a zero-dimensional 
sheaf on $\cC_b$ with $\chi(Q) \leq k$; for example see Remark~\ref{rmk:nodal} when $\cC_b$ is nodal. 
Note that $\cH(\chi)_{\leq 0}$ is the regular locus of $\cH(\chi)$. 
One can easily check that (\ref{Hk}) is an open substack, 
and $\cH(\chi)_{\leq k}=\cH(\chi)$ for $k\gg 0$. 
We have the open immersions 
\begin{align*}
    \cH(\chi)_{\leq 0} \subset  \cH(\chi)_{\leq 1} \subset \cdots 
    \subset \cH(\chi). 
\end{align*}
We define the following closed substack 
with reduced structure 
\begin{align*}
    \cH(\chi)_{\geq k}:=\cH(\chi)\setminus \cH(\chi)_{\leq k-1}
    \subset \cH(\chi).
\end{align*}
Then we have the sequence of closed substacks 
\begin{align*}
  \cH(\chi) \hookleftarrow      \cH(\chi)_{\geq 1} \hookleftarrow  \cH(\chi)_{\geq 2} \hookleftarrow \cdots
\end{align*}
We set 
\begin{align*}
    \cH(\chi)_k := \cH(\chi)_{\geq k}\setminus  \cH(\chi)_{\geq k+1} \subset 
    \cH(\chi)_{\geq k}
\end{align*}
which is an open substack of $\cH(\chi)_{\geq k}$. 

Let $p_{\geq k}$ be the induced map from $p$
\begin{align}\label{map:pk}
    p_{\geq k} \colon \mathcal{U}_{\geq k}:=(p^{-1}(\cH(\chi)_{\geq k}))^{\mathrm{red}} \to \cH(\chi)_{\geq k}. 
\end{align}
Then the restriction of $q$ to 
$\mathcal{U}_{\geq k}$
factors through 
\begin{align*}
    q_{\geq k} \colon  \U_{\geq k} \to  (\cC)^{\times_{\cB}d}_{\geq k}
    \subset  (\cC)^{\times_{\cB}d}
\end{align*}
where 
$(\cC)^{\times_{\cB}d}_{\geq k}$ is the 
reduced closed subscheme consisting of points 
$(x_1, \ldots, x_d)$ on $\cC_b$ such that the number of $1\leq i\leq d$ with 
$x_i \in \cC_b^{\mathrm{sg}}$ is bigger than or equal to $k$, and 
$\cC_b^{\mathrm{sg}} \subset \cC_b$ is the set of singular 
points. 
We set 
\begin{align*}
     (\cC)^{\times_{\cB}d}_k:= (\cC)^{\times_{\cB}d}_{\geq k}
     \setminus  (\cC)^{\times_{\cB}d}_{\geq k+1},
\end{align*}
and define the following open substack of $\U_{\geq k}$
\begin{align*}
    \mathcal{U}_k :=q_{\geq k}^{-1}((\cC)^{\times_{\cB}d}_k) \subset 
    \mathcal{U}_{\geq k}. 
\end{align*}
Then the map (\ref{map:pk}) restricts to the map 
\begin{align*}
    p_k \colon \mathcal{U}_k \to \cH(\chi)_k. 
\end{align*}
Indeed for a filtration (\ref{filt:F}), suppose that there is $I \subset \{1, \ldots, d\}$ with $\sharp I=k$ and $x_i \in \mathcal{C}_{b}^{\mathrm{sg}}$
if and only if $i\in I$. 
 Then we have the exact sequence on $\Coh^{\heartsuit}(\cC_b)$
\begin{align}\label{exact:F}
    0\to F_0(ac)\to L(ac) \to Q \to 0
\end{align}
where $L=\mathcal{O}_{\cC_b}(-\sum_{i\notin I}x_i)$ is a line bundle 
on $\cC_b$ and $Q$ is supported on $x_i$ for $i\in I$. This implies that 
$F_0(ac)$ lies in $\cH(\chi)_k$.
The above construction is summarized in the following diagram 
\begin{align}\label{dia:Uk}
\xymatrix{
  (\cC)^{\times_{\cB}d}_k \dinclusion \diasquare & \ar[l]_-{q_k} \U_k \ar[r]^-{p_k} \dinclusion & 
  \cH(\chi)_k \dinclusion \\
   (\cC)^{\times_{\cB}d}_{\geq k} \dinclusion  & \ar[l]_-{q_{\geq k}} \U_{\geq k} \dinclusion \ar[r]^-{p_{\geq k}} \diasquare^{\mathrm{red}}& 
  \cH(\chi)_{\geq k} \dinclusion \\
   (\cC)^{\times_{\cB}d}  & \ar[l]_-{q} \U \ar[r]^-{p} & 
  \cH(\chi).
}
\end{align}

We also define 
\begin{align*}
      (\cC)^{\times_{\cB}d}_{\leq k}:=  (\cC)^{\times_{\cB}d}\setminus   (\cC)^{\times_{\cB}d}_{\geq k+1}
\end{align*}
which is an open subset of $(\cC)^{\times_{\cB}d}$. 
We define the open substack $\U_{\circ} \subset \U$ by the Cartesian square 
\begin{align*}
\xymatrix{
\U_{\circ} \inclusion \ar[d]_-{q|_{\U_{\circ}}}  \diasquare & \U \ar[d]_-{q} \\
 (\cC)^{\times_{\cB}d}_{\leq k} \inclusion & 
  (\cC)^{\times_{\cB}d}.
}
\end{align*}
Note that 
\begin{align}\label{Uintk}
\U_k=\U_{\geq k} \cap \U_{\circ}
\end{align}
and $\U_k$ is an open substack of $\U_{\geq k}$ and is a closed substack of $\U_{\circ}$.

Now by induction on $k$, suppose that 
\begin{align*}A \in \QCoh_{\cH(\chi)_{\geq k}}(\cH(\chi))
\end{align*}
and it satisfies $q_{*}p^*A=0$. 
Then we have
\begin{align*}
    (p^* A)|_{\U_{\circ}} \in \QCoh_{\U_k}(\U_{\circ}), \ 
    (q|_{\U_{\circ}})_{*}((p^* A)|_{\U_{\circ}})=0.
\end{align*}
By Lemma~\ref{lem:pq} below, it follows that $(p^* A)|_{\U_{\circ}}=0$. 
Take a quasi-compact open substack $\cH(\chi)^{\mathrm{qc}} \subset \cH(\chi)$
and below we use the notation in Lemma~\ref{lem:pq2}. 
Then 
\begin{align*}
(p^{*}A)|_{\U_{\circ}^{\mathrm{qc}}}=(p|_{\U_{\circ}^{\mathrm{qc}}})^* (A|_{\cH(\chi)^{\mathrm{qc}}})=0.
\end{align*}
By Lemma~\ref{lem:pq3}, for $d\gg 0$, we conclude that 
\begin{align*}
 A|_{\cH(\chi)^{\mathrm{qc}}} \in \QCoh_{\cH(\chi)_{\geq k+1}^{\mathrm{qc}}}(\cH(\chi)^{\mathrm{qc}}).   
\end{align*}
Since the above holds for any quasi-compact 
$\cH(\chi)^{\mathrm{qc}} \subset \cH(\chi)$, we conclude that 
$A$ is supported on $\cH(\chi)_{\geq k+1}$. By the induction on $k$, 
we have $A=0$. 
\end{proof}

We have used the following lemmas: 
\begin{lemma}\label{lem:pq}
 The map $q_k$ in the diagram (\ref{dia:Uk}) is finite. In particular, the functor 
 \begin{align*}
(q|_{\U_{\circ}})_{*} \colon \QCoh_{\U_k}(\U_{\circ}) \to \QCoh( (\cC)^{\times_{\cB}d}_{\leq k})
\end{align*}
is conservative.
\end{lemma}
\begin{proof}
  First, the map $q_k$ is proper; this follows since $q$ is proper by iterated applications of Lemma~\ref{lem:qsmooth}, and 
  $q_k$ is obtained from $q$ by restrictions to closed subschemes and base-change. Therefore 
  it is enough to show that $q_k$ is quasi-finite. 
  
    For a point 
    $(x_i)_{1\leq i\leq d} \in  (\cC_b)^d$, 
suppose that there is a subset $I\subset \{1, \ldots, d\}$ with $\sharp I=k$ such that 
$x_i \in \cC_b^{\mathrm{sg}}$ if and only if $i \in I$. 
Let a filtration (\ref{filt:F}) 
    corresponds to a fiber at $(x_1, \ldots, x_d)$ in $\mathcal{U}_k$. 
    Then we have the exact sequence on $\Coh^{\heartsuit}(\cC_b)$
\begin{align}\notag
    0\to F_0(ac)\to L(ac) \to Q \to 0
\end{align}
as in (\ref{exact:F}). 
Then by the definition of $\mathcal{U}_k$, we have that 
\begin{align}\notag
k=\min\{ \delta :\mbox{there is } 0 \to F_0(ac) \to \mathcal{L}' \to Q' \to 0, \chi(Q')=\delta  \}
\end{align}
where $\mathcal{L}'$ is a line bundle on $\cC_b$ and $Q'$ is a zero-dimensional sheaf. 
Since (\ref{exact:F}) attains the smallest in the right-hand side, 
by Lemma~\ref{lem:delta} we have 
\begin{align}\label{F0:emb}
    L^{\vee} \subset F_0^{\vee} \subset \nu_{*}\nu^* L^{\vee}
\end{align}
where $\nu \colon \widetilde{\cC}_b \to \cC_b$ is the normalization.

Since $\nu_{*}\nu^* L^{\vee}/L^{\vee}$
is a direct sum of the form $k[t]/t^{m_i}$ by Assumption~\ref{assum:Bcirc}, there are only finitely many 
$\mathcal{O}_{\cC_b}$-modules $F_0$ satisfying (\ref{F0:emb}) up to isomorphisms. 
Then a filtration (\ref{filt:F}) is uniquely reconstructed from the data $F_0 \subset L$ together with 
$x_i\in \cC_b$ for $i\neq I$ as they are non-singular points.  
Therefore $q_k$ is quasi-finite. 
\end{proof}

The following lemma will be used in Lemma~\ref{lem:pq3}: 
\begin{lemma}\label{lem:pq2}
For any quasi-compact open substack $\cH(\chi)^{\mathrm{qc}} \subset \cH(\chi)$, 
 there is $d\gg 0$ with $r|(d+\chi)$ such that 
the map $p_k$ in the diagram (\ref{dia:Uk})
is surjective on $\cH(\chi)_k^{\mathrm{qc}} \subset \cH(\chi)_{\leq k}^{\mathrm{qc}}$.  
Here $(-)^{\mathrm{qc}}$ means the pull-back with respect to $\cH(\chi)^{\mathrm{qc}} \subset \cH(\chi)$. 
\end{lemma}
\begin{proof}
For $E\in \Coh^{\heartsuit}(\cC_b)$
corresponding to a point in $\cH(\chi)_k^{\mathrm{qc}}$, there is an exact sequence 
\begin{align*}
    0\to E \to \mathcal{L} \to Q \to 0
\end{align*}
as in 
(\ref{exact:ELQ})
such that $\chi(Q)=k$. Note that $Q$ is supported on $\cC_b^{\mathrm{sg}}$, 
as otherwise $E$ lies in $\cH(\chi)_{\leq k-1}$. 

Then 
by setting $a\in \mathbb{Z}$ as in (\ref{denote:a}), 
we have 
an injection $\mathcal{L}(-a \cdot c) \hookrightarrow \mathcal{O}_{\cC_b}$
whose cokernel is supported on the smooth locus of $\cC_b$, 
as $\mathcal{L}^{\vee}(a \cdot c)$ is very ample for $d\gg 0$.
Since $\cH(\chi)^{\mathrm{qc}}$ is quasi-compact, we can take $d\gg 0$ to be independent of a point in $\cH(\chi)_k^{\mathrm{qc}}$. 
Then we have an injection $E(-a \cdot c) \hookrightarrow \mathcal{O}_{\cC_b}$ 
whose cokernel at singular points have length $k$, which implies the lemma.
\end{proof}

\begin{lemma}\label{lem:pq3}
Let $\cH(\chi)^{\mathrm{qc}} \subset \cH(\chi)$ be a quasi-compact open substack. 
Take $d \gg 0$ as in Lemma~\ref{lem:pq2}, and consider the open substack
$\mathcal{U}_{\circ}^{\mathrm{qc}} \subset \mathcal{U}^{\mathrm{qc}}$. 
Then, for the morphism
\[
p|_{\mathcal{U}_{\circ}^{\mathrm{qc}}} \colon \mathcal{U}_{\circ}^{\mathrm{qc}} \to \cH(\chi)^{\mathrm{qc}},
\]
the pullback functor
\[
(p|_{\mathcal{U}_{\circ}^{\mathrm{qc}}})^* \colon 
\QCoh_{\cH(\chi)_{\geq k}^{\mathrm{qc}}}(\cH(\chi)^{\mathrm{qc}}) 
\to \QCoh(\mathcal{U}_{\circ}^{\mathrm{qc}})
\]
is conservative on the open substack 
$\cH(\chi)_k^{\mathrm{qc}} \subset \cH(\chi)_{\leq k}^{\mathrm{qc}}$. 
That is, if an object
\[
A \in \QCoh_{\cH(\chi)_{\geq k}^{\mathrm{qc}}}(\cH(\chi)^{\mathrm{qc}})
\]
is sent to zero under this functor, then $A$ is supported on
$\cH(\chi)_{\geq k+1}^{\mathrm{qc}}$.
\end{lemma}
\begin{proof}
    We have the factorization of the map $p\colon \U \to \cH(\chi)$
    \begin{align*}
        p \colon \U \stackrel{\beta}{\to} \overline{\U} \stackrel{\alpha}{\to} \cH(\chi)
    \end{align*}
    where $\overline{\U}$ is the moduli stack which classifies exact sequences in $\Coh^{\heartsuit}(S)$
    \begin{align}\label{F0Q}
        0\to F_0 \to \mathcal{O}_{\cC_b} \to Q \to 0
    \end{align}
    where $Q$ is zero-dimensional with length $d$. As a derived scheme, it is 
    equivalent to the derived moduli stack of PT stable pairs~\cite{MR2545686, PT3} $(F_0^{\vee}, \eta)$, where 
    $F_0^{\vee}$ is a pure one dimensional sheaf on $S$ with support $\cC_b$ for $b\in \cB$ and 
    $\eta \colon \mathcal{O}_S \to F_0^{\vee}$ is generically 
    surjective whose cokernel has length $d$. 
    The map $\beta$ sends (\ref{filt:F})
    to $F_0 \subset F_d=\mathcal{O}_{\cC_b}$, and $\alpha$
    sends (\ref{F0Q}) to $F_0(ac)$. 

By taking the pull-back along $\cH(\chi)^{\mathrm{qc}} \subset \cH(\chi)$, we obtain 
\begin{align*}
       p \colon \U^{\mathrm{qc}} \stackrel{\beta}{\to} \overline{\U}^{\mathrm{qc}} \stackrel{\alpha}{\to} \cH(\chi)^{\mathrm{qc}}.
\end{align*}
The map $\alpha$ is a smooth map for $d\gg 0$. 
Indeed for $E\in \Coh^{\heartsuit}(\cC_b)$ corresponding to a point in 
$\cH(\chi)^{\mathrm{qc}}$, the fiber of $\alpha$ is an open subset of 
$H^0(E^{\vee}(a \cdot c))$ where $a$ is given by (\ref{denote:a}), 
whose dimension is constant for $d\gg 0$ as $\cH(\chi)^{\mathrm{qc}}$ is quasi-compact. 
Therefore $\overline{\U}^{\mathrm{qc}}$ is an open substack of a vector bundle over 
$\cH(\chi)^{\mathrm{qc}}$, hence $\alpha$ is smooth. 

    We have the commutative diagram 
    \begin{align}\label{com:Uqc}
        \xymatrix{
\U^{\mathrm{qc}} \ar[r]^-{\beta} \ar[d] & \overline{\U}^{\mathrm{qc}} \ar[d] \\
(\cC)^{\times_{\cB}d} \ar[r] & \mathrm{Sym}_{\cB}^d(\cC).
        }
    \end{align}
    Here $\mathrm{Sym}_{\cB}^d(\cC)$ is the relative symmetric product of $\cC$ over $\cB$. 
We take the open subset 
\begin{align}\notag
    \mathrm{Sym}_{\cB}^d(\cC)_{\leq k} \subset 
    \mathrm{Sym}_{\cB}^d(\cC)
\end{align}
to be consisting of unordered tuples $(x_1, \ldots, x_d)$
with $x_i \in \cC_b$ such that the number of $1\leq i \leq d$ with $x_i \in \cC_b^{\mathrm{sg}}$ is at most $k$. 
We denote by $\overline{\U}_{\circ}^{\mathrm{qc}}\subset \overline{\U}^{\mathrm{qc}}$
the pull-back with respect to the above open subset. 
Then the diagram (\ref{com:Uqc}) restricts to the commutative diagram 
\begin{align}\notag
       \xymatrix{
\U^{\mathrm{qc}}_{\circ} \ar[r]^-{\beta_{\circ}} \ar[d] & \overline{\U}^{\mathrm{qc}}_{\circ} \ar[d] \\
(\cC)^{\times_{\cB}d}_{\leq k} \ar[r] & \mathrm{Sym}_{\cB}^d(\cC)_{\leq k}.
        }
\end{align}
We obtain the diagram 
\begin{align*}
    \xymatrix{
\U^{\mathrm{qc}}_{\circ} \ar[r]^-{\beta_{\circ}} \dinclusion \diasquare & \overline{\U}^{\mathrm{qc}}_{\circ} \dinclusion \ar[rd]^-{\alpha_{\circ}} & \\
\U^{\mathrm{qc}} \ar[r]^-{\beta} & \overline{\U}^{\mathrm{qc}} \ar[r]^-{\alpha} & \cH(\chi)^{\mathrm{qc}}.
    }
\end{align*}

By Lemma~\ref{lem:pq2} and the property (\ref{Uintk}),
the map $\alpha_{\circ}$ is surjective on $\cH(\chi)_{k}^{\mathrm{qc}}$. 
Together with the smoothness of $\alpha_{\circ}$, 
this implies that the functor 
\begin{align}\label{funct:alpha}
    \alpha_{\circ}^* \colon \QCoh_{\cH(\chi)_{\geq k}^{\mathrm{qc}}}(\cH(\chi)^{\mathrm{qc}}) \to \QCoh(\overline{\U}_{\circ}^{\mathrm{qc}})
\end{align}
is conservative in a neighborhood of $\cH(\chi)_k^{\mathrm{qc}}$. 
Indeed as $\alpha_{\circ}$ is smooth, there is an open substack 
$\cH' \subset \cH(\chi)^{\mathrm{qc}}$
such that 
\begin{align}\label{chK}\cH' \cap \cH(\chi)^{\mathrm{qc}}_{\geq k}=\cH(\chi)_k^{\mathrm{qc}}
\end{align}
and the map 
$\alpha_{\circ}^{-1}(\cH') \to \cH'$ is smooth and surjective, in particular faithfully flat. 
Then the functor 
\begin{align}\label{funct:alpha2}
    \alpha_{\circ}^* \colon \QCoh(\cH') \to \QCoh(\alpha_{\circ}^{-1}(\cH'))
\end{align}
is conservative. Now if $A \in \QCoh_{\cH(\chi)_{\geq k}^{\mathrm{qc}}}(\cH(\chi)^{\mathrm{qc}})$
is sent to zero under the functor (\ref{funct:alpha}), 
then $\alpha_{\circ}^* A|_{\alpha_{\circ}^{-1}(\cH')}=0$, hence 
$A|_{\cH'}=0$ as (\ref{funct:alpha2}) is conservative. 
Therefore $A$ is supported on $\cH(\chi)_{\geq k+1}^{\mathrm{qc}}$ by the condition (\ref{chK}). 

It remains to show that the functor 
\begin{align}\notag
\beta_{\circ}^* \colon 
    \QCoh(\overline{\U}_{\circ}^{\mathrm{qc}}) \to \QCoh(\U_{\circ}^{\mathrm{qc}})
\end{align}
is conservative. 
Below we use the short hand notation of (nested) Hilbert schemes of points, see the proof of Proposition~\ref{prop:PE} 
\begin{align*}
    H^d =\Hilb^d(S), \ H^{d, d-1}=\Hilb^{d, d-1}(S). 
\end{align*}
We define 
\begin{align*}
    H^{1_d}:=H^{d, d-1}\times_{H^{d-1}}H^{d-1, d-2}\times_{H^{d-2}} 
    \cdots\times_{H^2} H^{2, 1}.
\end{align*}
It is the derived moduli stack which classifies flags of zero-dimensional subschemes 
\begin{align}\label{filt:H}
    Z_1 \subset Z_2 \subset \cdots \subset Z_d \subset S
\end{align}
where $Z_i$ has length $i$. 
Then we have the Cartesian square 
\begin{align*}
    \xymatrix{
    \U_{\circ}^{\mathrm{qc}} \dinclusion \ar[r]^{\beta_{\circ}} \diasquare & \overline{\U}_{\circ}^{\mathrm{qc}} \dinclusion \\
    \cB\times H^{1_d} \ar[r]^-{1\times \beta_H} & \cB\times H^d.
    }
\end{align*}
Here vertical arrows are closed embeddings, and the bottom horizontal arrow 
sends (\ref{filt:H}) to $Z_d \subset S$. 
Since the push-forward functors of quasi-coherent sheaves along with the vertical arrows are conservative, 
by the base change it is enough to show that 
\begin{align}\label{betaH}
\beta_H^* \colon \QCoh(H^d) \to \QCoh(H^{1_d})
\end{align}
is conservative. It is proved in Lemma~\ref{lem:consH}. 
\end{proof}

\begin{lemma}\label{lem:consH}
    The functor (\ref{betaH}) is conservative. 
\end{lemma}
\begin{proof}
    Let $X^d=\widetilde{\Hilb}^d(S)$ be the isospectral Hilbert scheme, see Subsection~\ref{subsec:hecke}. We show that, by the induction on $d$, that 
    there is a factorization of the projection $H^{1_d} \to H^d$ 
    \begin{align}\label{fact:Xd}
      H^{1_d} \stackrel{f^d}{\to} X^d \stackrel{\rho^d}{\to} H^d
    \end{align}
    such that $f^{d}_{*} \mathcal{O}_{H^{1_d}}=\mathcal{O}_{X^d}$. 
    The case of $d=1$ is obvious; we can take $f^1=\id$. 

    Assume that $f^{d-1}$ is constructed. We have the following Cartesian squares 
    \begin{align}\label{dia:Hd}
        \xymatrix{
        H^{1_{d}} \ar[r]^-{h^{d}} \ar[d] \diasquare & Y^d \ar[r] \ar[d] \diasquare & H^{d, d-1} \ar[d] \\
        H^{1_{d-1}} \ar[r]^-{f^{d-1}} & X^{d-1} \ar[r]^-{\rho^{d-1}} & H^{d-1}. 
        }
    \end{align}
    Here $Y^d$ satisfies that 
    \begin{align*}
Y^d=(Y^d)^{\mathrm{cl}}=(Y^d)^{\mathrm{cl, red}}.
    \end{align*}
    The first identity follows since $\rho^{d-1}$ is flat (which is the key property of 
    isospectral Hilbert scheme~\cite{Ha}), and the second identity is 
    also proved in the proof of~\cite[Lemma~3.8.5]{Ha}. 
Then as $Y^d$ is a reduced scheme, we have the induced morphism 
$g^d \colon Y^d \to X^d$, and it satisfies 
$g^d_{*}\mathcal{O}_{Y^d}=\mathcal{O}_{X^d}$, see the proof of~\cite[Lemma~3.8.5]{Ha}.
By the diagram (\ref{dia:Hd}), derived base change and induction hypothesis, 
we have $h^d_{*}\mathcal{O}_{H^{1_d}}=\mathcal{O}_{Y^d}$. 
Therefore the map 
\begin{align*}
f^d=g^d \circ h^d \colon H^{1_{d}} \to Y^d \to X^d
\end{align*}
satisfies the desired property. 

The functor 
\begin{align*}
(\rho^d)^* \colon \QCoh(H^d) \to \QCoh(X^d)
\end{align*}
is conservative as $\rho^d$ is finite flat cover of $H^d$, hence faithfully flat. 
The functor 
\begin{align*}
    (f^d)^* \colon \QCoh(X^d) \to \QCoh(H^{1_d})
\end{align*}
is conservative (indeed fully-faithful) since it admits a right adjoint 
$(f^d)_{*}$ satisfying 
$(f^d)_{*} (f^d)^{*}=\id$ because of $f^d_{*}\mathcal{O}_{H^{1_d}}=\mathcal{O}_{X^d}$. 
Therefore by the factorization (\ref{fact:Xd}), the functor (\ref{betaH}) is conservative. 
\end{proof}
The following is probably well known and follows from a standard argument.
We will give its proof in Subsection~\ref{subsec:delta} because we cannot find a reference. 
\begin{lemma}\label{lem:delta}
Let $R=k[[x, y]]/(f)$ be a planar singularity, and 
$M$ be a rank one torsion-free $R$-module. Let 
\begin{align*}
    \delta:=\min \{ \dim M/R : R\hookrightarrow M\}
\end{align*}
where the minimum is after all injective 
$R$-module homomorphisms $R\hookrightarrow M$, 
so that $M/R$ is finite dimensional. Then if an 
embedding $R \hookrightarrow M$ 
satisfies $\delta=\dim M/R$, then 
we have 
\begin{align*}
    R \hookrightarrow M \hookrightarrow \widetilde{R}=\bigoplus_{i=1}^{\ell}k[[t_i]]
\end{align*}
where $\Spec \widetilde{R} \to \Spec R$ is the normalization. 
\end{lemma}

\begin{remark}\label{rmk:nodal}
If $\cC_b$ has at worst nodal singularities, then 
$\cH(\chi)_{\leq k}$ corresponds to $E\in \Coh(\cC_b)$ such
that the number of points in $\cC_b$ at which $E$ is not a line 
bundle is at most $k$. 
\end{remark}

As a corollary of the above proposition together with Theorem~\ref{thm:cpt}, 
Lemma~\ref{lem:cptC}, we have the following: 
\begin{cor}\label{cor:Hcpt}
    Under Assumption~\ref{assum:Bcirc},
    for $d\gg 0$ the image of the functor 
    \begin{align*}
        \Coh(S^{\times d}) \to \LL_{\mathcal{N}}(\cH(\chi))_{w'}
    \end{align*}
    given by 
    \begin{align}\label{objects:Ai}
        A_1\boxtimes \cdots \boxtimes A_d \mapsto  (\otimes\mathcal{O}((\chi-\chi_0+d)/r \cdot c))_{*}\circ \mathrm{H}_{A_d}\circ 
        \cdots \circ\mathrm{H}_{A_1}(s_{!}\mathcal{O}_{\cB})
    \end{align}
    compactly generates $\IndL_{\mathcal{N}}(\cH(\chi))_{w'}$.
\end{cor}
\subsection{Equivalence of categories}
The following is the main result in this section: 
\begin{thm}\label{thm:eq}
Under Conjecture~\ref{conj:norm}
the functor $\Phi$ in (\ref{induce:Phi}) is fully-faithful 
\begin{align}\label{funct:main}
\Phi \colon \IndCoh_{\mathcal{N}}(\cH(w)^{\mathrm{ss}})_{-\chi'} \hookrightarrow
\IndL_{\mathcal{N}}(\cH(\chi))_{w'}. 
\end{align}
Moreover, under Assumption~\ref{assum:Bcirc}, the functor (\ref{funct:main}) is an equivalence 
i.e. the DL holds over $\cB$. 
\end{thm}
\begin{proof}
    We consider the adjoint pair of functors in Corollary~\ref{cor:PhiR}
    \begin{align}\notag
        \xymatrix{
        \IndCoh_{\mathcal{N}}(\cH(w)^{\mathrm{ss}})_{-\chi'} \ar@<0.5ex>[r]^-{\Phi} & \IndL_{\mathcal{N}}(\cH(\chi))_{w'} \ar@<0.5ex>[l]^-{\Phi^R}
        }
    \end{align}
    We have the distinguished triangle of functors 
    \begin{align}\label{seq:up}
        \id \to \Phi^R \Phi \to \Upsilon
    \end{align}
    where the first arrow is given by the adjunction. 
    By applying $\mathrm{W}_{d, c}^R$, we obtain 
    \begin{align}\label{seq:W}
        \mathrm{W}_{d, c}^R \to \mathrm{W}_{d, c}^R\Phi^R \Phi \to \mathrm{W}_{d, c}^R \Upsilon.
    \end{align}
    By the diagram in Corollary~\ref{cor:PhiR}, we have 
    \begin{align}\label{isom:WH}
        \mathrm{W}_{d, c}^R\Phi^R \Phi \cong \mathrm{H}_{d, c}^R \Phi \cong \mathrm{W}_{d, c}^R
    \end{align}
so that the distinguished triangle (\ref{seq:W}) is of the form 
\begin{align}\label{seq:PhiR2}
        \mathrm{W}_{d, c}^R \stackrel{\gamma}{\to} \mathrm{W}_{d, c}^R \to \mathrm{W}_{d, c}^R \Upsilon.
\end{align}

We show that $\gamma=\id$. The kernel object of $\mathrm{W}_{d, c}^R$ is
$\cE^{(d)}[d]$, where $\cE^{(d)}$ is given by 
\begin{align*}
\cE^{(d)}:=
    \bigotimes_{i=1}^d (p_i \times \id)^* \cE^{\vee} \in \Coh^{\heartsuit}(S^{\times d}\times \cH(w)^{\mathrm{ss}}).
\end{align*}
Here $p_i$ is the projection from $S^{\times d}$ onto the corresponding factor. 
Since $\cE^{\vee}$ is a $\cH(w)^{\mathrm{ss}}$-flat family of one-dimensional 
sheaves on $S$ with support in $\cC\times_{\cB}\cH(w)^{\mathrm{ss}}$, 
the sheaf $\cE^{(d)}$ is flat over $\cH(w)^{\mathrm{ss}}$ and supported over 
$\cC^{\times_{\cB}d}\times_{\cB}\cH(w)^{\mathrm{ss}}$. Then $(\gamma-\id)[-d]$ is induced by a
morphism of $\cH(w)^{\mathrm{ss}}$-flat sheaves $\cE^{(d)} \to \cE^{(d)}$, see Lemma~\ref{lem:Funcon}. 
Since $h \colon\cH(w)^{\mathrm{ss}}\to \cB$ is also flat by~\cite[Theorem~2.2.4]{BD0}, 
the sheaf $\cE^{(d)}$ is also flat over $\cB$. 

 On the other hand, the isomorphisms (\ref{isom:WH}) are induced from the isomorphism (\ref{PWHU}), which is canonical by its construction. Over 
 $\cB^{\mathrm{sm}}:=\cB \cap \mathrm{B}^{\mathrm{sm}}$,  we can identify $\cH(w)^{\mathrm{ss}}$ with the relative 
 Jacobian and $\cE^{\vee}$ with the dual of Poincar\'e line bundle. From the above description it is straightforward to check that, over $\cB^{\mathrm{sm}}$, we have $\gamma|_{S^{\times d}\times\cH(w)^{\mathrm{ss}}\times_{\cB}\cB^{\mathrm{sm}}}=\id$.

 This implies that, if $\gamma \neq \id$, then the image of $(\gamma-\id)[-d]$ is a non-zero subsheaf 
 of $\cE^{(d)}$ supported over 
 $\cC^{\times_{\cB}d}\times_{\cB}\cH(w)^{\mathrm{ss}}\times_{\cB}\cB^{\mathrm{sg}}$,
 where $\cB^{\mathrm{sg}}:=\cB\setminus \cB^{\mathrm{sm}}$. However, as $\cB$ is smooth and irreducible and $\cB^{\mathrm{sg}} \subsetneq \cB$ is a closed subset, such a subsheaf is annihilated by
 $I^N$ for $N\gg 0$, where $I\subset \cB$ is the ideal sheaf which defines $\cB^{\mathrm{sg}}$. 
 This contradicts that $\cE^{(d)}$ is flat over $\cB$, therefore $\gamma=\id$. 

As $\gamma=\id$, from the triangle (\ref{seq:PhiR2}) we have $\mathrm{W}_{d, c}^R \Upsilon=0$. 
    Then by Corollary~\ref{cor:cons1}, we have $\Upsilon=0$. Therefore from 
    (\ref{seq:up}), we have the isomorphism 
    $\id \stackrel{\cong}{\to} \Phi^R \Phi$, i.e. $\Phi$ is fully-faithful. 

The objects of the form (\ref{objects:Ai}) come from the image of the functor (\ref{funct:main})
by Corollary~\ref{cor:comC}. Therefore by Corollary~\ref{cor:Hcpt}, under Assumption~\ref{assum:Bcirc}, 
the fully-faithful functor (\ref{funct:main}) is also essentially surjective, hence an equivalence.
\end{proof}


\begin{remark}\label{rmk:check}
In Subsection~\ref{subsec:Hcpt} we give another proof that $\gamma$ is an isomorphism, 
which, however, works except $g=2$ and $G=\GL_2$. 
\end{remark}
\section{Whittaker normalization for \texorpdfstring{$G=\GL_2$}{G=GL2}}
In this section, we restrict to the case of $G=\GL_2$ and show that 
Conjecture~\ref{conj:norm} holds in this case, see Proposition~\ref{prop:norm}.
Together with Theorem~\ref{thm:eq} and Example~\ref{exam1} (1), 
we obtain the main result: 
\begin{thm}\label{thm:gl2}
The DL holds over $\rB^{\mathrm{red}} \subset \rB$ for $G=\GL_2$. 
\end{thm}

Below we assume that $G=\GL_2$. By Remark~\ref{rmk:Bcl}, it is enough to prove 
the case of
\begin{align*}
    \cB=\rB^{\mathrm{red, cl}} \subset \rB^{\mathrm{cl}}.
\end{align*}
As in the previous sections, we use the notation
\begin{align*}
    \cH=\Hig_G \times_{\rB} \cB.
\end{align*}
Recall that we have the universal families
\begin{align}\label{univ:E}
    \cE \in \Coh^{\heartsuit}(\cC\times_{\cB}\cH), \ 
    \cF \in \Coh^{\heartsuit}(C \times \cH).\end{align}
They are related by $\pi_{*}\cE=\cF$, where 
$\pi \colon \cC\times_{\cB}\cH \to C\times \cH$ is the projection. 

We take $b\in \cB$ such that the corresponding spectral curve $\cC_b \subset S$ is not irreducible; so it is of the form
\begin{align*}
    \cC_b=C_1 \cup C_2, \ C_1 \cap C_2=\sum_{i=1}^{\ell} m_i x_i
\end{align*}
for $x_i \in C_1 \cap C_2$ such that $m_1+\cdots+m_{\ell}=2g-2$, and $C_i \stackrel{\cong}{\to} C$.

\subsection{Resolution of Arinkin sheaf}\label{subsec:resol}
Let $E\in \Coh(\cC_b)$ be a rank one (not necessarily semistable) 
torsion-free sheaf. As in (\ref{PE}), 
the pull-back of the Arinkin sheaf (\ref{Arsheaf0}) yields 
the Cohen--Macaulay sheaf 
\begin{align*}
    \cP_E \in \Coh^{\heartsuit}(\cH_b).
\end{align*}
Below we construct a resolution of the above sheaf by vector bundles, which will be used to show that it lies in the limit category. 

\begin{lemma}\label{lem:LEQ}
    There is an exact sequence
    of the form 
    \begin{align}\label{exact:LEQ0}
0\to \mathcal{L} \to E \to Q \to 0
    \end{align}
    where $\mathcal{L}$ is a line bundle on $\cC_b$ and 
    $Q\cong\oplus_{i=1}^{\ell} \mathcal{O}_{Z_i}$, where 
    $Z_i$ is a closed subscheme 
    \begin{align*}
        Z_i \hookrightarrow (C_1 \cap C_2)_{x_i}=\Spec k[t]/t^{m_i}, 
    \end{align*}
    i.e. $Z_i=\Spec k[t]/t^{n_i}$ for some $0\leq n_i \leq m_i$.
\end{lemma}
\begin{proof}
From Lemma~\ref{lem:delta}, there is a line bundle $\mathcal{L}$ with 
an injection $\mathcal{L} \hookrightarrow E$ such that 
$E/\mathcal{L} \subset \nu_{*}\nu^* \mathcal{L}$, where 
$\nu \colon \widetilde{\cC}_b \to \cC_b$ is the normalization. 
Then the lemma follows from the fact that 
$\nu_{*}\nu^{*}\mathcal{L}/\mathcal{L}$ is supported at $\{x_1, \ldots, x_{\ell}\}$, where at each $x_i$ its stalk is 
isomorphic to $k[t]/t^{m_i}$.
\end{proof}

If $\chi(Q)>0$, choose an index \(j\) with \(n_j>0\); after renumbering assume \(j=1\). Then choose a surjection
$\mathcal{O}_{Z_1} \twoheadrightarrow \mathcal{O}_{x_1}$ and consider the 
exact sequence 
\begin{align}\label{EprimeEOx}
    0\to E' \to E \to \mathcal{O}_{x_1} \to 0. 
\end{align}
Here $E'$ fits into an exact sequence 
\begin{align}\label{Qprime}
   0\to \mathcal{L} \to E' \to Q' \to 0
\end{align}
such that $\chi(Q')=\chi(Q)-1$. 
The sequence (\ref{EprimeEOx}) corresponds to a map 
\begin{align}\label{point:hecke}
   \tau_0 \colon \Spec k\to  \cHecke:=\mathrm{Hecke}_G'\times_{\rB}\cB . 
\end{align}
Here $\mathrm{Hecke}_G'$ is the stack of Hecke correspondences in the diagram (\ref{dia:gregprime}).

\begin{lemma}\label{lem:Delta0}
For each $D\geq 1$, there is a pointed smooth and irreducible affine variety $(\Delta, 0)$ for $0\in \Delta$ with $\dim \Delta=D$
and an extension of the map (\ref{point:hecke}) 
\begin{align}\label{iota}
\tau \colon \Delta \to  \cHecke, \ \tau|_{0}=\tau_0
\end{align}
such that the composition 
\begin{align*}
    \Delta \to  \cHecke
    \to \cB
\end{align*}
sends the generic point of $\Delta$ to $\cB^{\mathrm{sm}} \subset \cB$, the open subset 
corresponding to smooth spectral curves.
\end{lemma}
\begin{proof}
It is enough to show that each irreducible component of $\cHecke$ containing the image of $p=\tau_0(\Spec k)$ is
dominant over $\cB$ 
under the map $\cHecke \to \cB$. 
Indeed suppose that the above statement holds. Let $W\to \cHecke$ be an atlas with a point $p'\in W$ 
which maps to $p$, and take an irreducible component $p' \in W' \subset W$. 
We take a resolution of singularities $W'' \to W'$ and a lift $p'' \in W''$ of $p'$. Then we take a smooth affine open neighborhood 
$p''\in W'''$; by the construction $W''' \to \cB$ is dominant. A desired $(\Delta, 0)$ is 
constructed by taking a generic hypersurface $0=p'' \in \Delta \subset W'''$
if $D<\dim W'''$, or if $D \geq \dim W'''$ take $\Delta=W''' \times \mathbb{A}^{D-\dim W'''}$, $0=(p'', 0)$ and construct a map $\Delta \to \cHecke$ by 
composing with the projection $\Delta \to W'''$.

We have the factorization 
\begin{align}\label{map:heckeB}\cHecke
    \stackrel{\ev_2}{\to} \cH \stackrel{h}{\to} \cB. 
\end{align}
The map $h$ is flat and surjective, see~\cite[Theorem~2.2.4]{BD0}. 
On the other hand, 
the virtual dimension of a fiber of $\ev_2$ is one by Lemma~\ref{lem:qsmooth}, and it equals 
the dimension of its classical truncation. 
Indeed from the exact sequence (\ref{exact:LEQ0}), we see that $\Hom(E, \mathcal{O}_x)$ is at most two-dimensional, hence $\Hom(E, \mathcal{O}_x)/\mathbb{G}_m$ for the scaling action of $\mathbb{G}_m$ is at most one-dimensional (and zero-dimensional if $x\in \cC_b$ is a smooth point).
It follows that the fibers of 
$\ev_2$ are one-dimensional on $h^{-1}(U)$ for an open neighborhood $b\in U \subset \cB$. Therefore the fibers of (\ref{map:heckeB}) are classical 
in a neighborhood of $b\in \cB$; this does not happen if 
there is an irreducible component of $\cHecke$ which is not 
dominant onto $\cB$ in a neighborhood of $b\in \cB$. 
Therefore we obtain the lemma. 
\end{proof}
\begin{remark}\label{rmk:Heckeflat}
From the proof of the above lemma, the stack $\cHecke$ is classical. 
Indeed $\cH$ is classical and quasi-smooth, the stack $\cHecke$ is quasi-smooth, 
and the classical truncations of the fibers of $\ev_2$ has the expected dimension one;
this implies that $\cHecke^{\mathrm{cl}}$ has the expected dimension, hence 
$\cHecke$ is classical. 
\end{remark}
The map (\ref{iota}) corresponds to an exact sequence 
in $\Coh^{\heartsuit}(\cC_{\Delta})$
where $\cC_{\Delta}:=\cC\times_{\cB}\Delta$
\begin{align}\label{exact:EQ}
    0\to E_{\Delta}'\to E_{\Delta} \to R_{\Delta}\to 0
\end{align}
such that each term is flat over $\Delta$, and restricts to the 
sequence (\ref{EprimeEOx}) at $0\in \Delta$. 

Let $\tau_i$ for $i=1, 2$ be the composition 
\begin{align*}
\tau_i \colon
    \Delta \stackrel{\tau}{\to} \cHecke \stackrel{\ev_i}{\to}\cH.
\end{align*}
We set $\cH_{\Delta}:=\Delta\times_{\cB}\cH$ and define 
\begin{align*}
    \cP_{\Delta i}:=(\tau_i \times \id)^* \cP^{\flat} \in 
    \Coh^{\heartsuit}(\cH_{\Delta}), \ i=1, 2.
\end{align*}
Here $\cP^{\flat}$ is the Cohen--Macaulay sheaf in Theorem~\ref{thm:Pflat}. 
Since $\cP^{\flat}$ is flat over both factors of $\cH$, 
the sheaves $\cP_{\Delta i}$ are 
Cohen--Macaulay sheaves on $\cH_{\Delta}$ flat over $\Delta$, 
such that $\cP_{\Delta 2}|_{0}\cong \cP_E$ and 
$\cP_{\Delta 1}|_{0}\cong \cP_{E'}$.

We also consider the covering involution of $\mathcal{C}_{\Delta} \to C\times \Delta$, 
denoted by $\sigma$. Note that $x_1 \in \cC_b=\cC_{\Delta}\times_{\Delta}0$ is fixed by $\sigma$. 
We have the following lemma, whose proof will be given in Subsection~\ref{subsec:proofxy}:
\begin{lemma}\label{lem:sigma} After shrinking $\Delta$ and replacing the map (\ref{iota}) if necessary, there is an exact sequence 
    \begin{align}\label{ext:sigma}
        0\to E_{\Delta}' \to E_{\Delta}^{\sigma} \to \sigma^* R_{\Delta} \to 0
    \end{align}
    such that $E_{\Delta}^{\sigma}$ is a $\Delta$-flat family of pure one-dimensional sheaves on the fibers of 
    $\cC_{\Delta}\to \Delta$. Moreover it is isomorphic to (\ref{exact:EQ}) at $0\in \Delta$ under the 
    identifications 
    \begin{align}\label{identify:sigma}\sigma^* R_{\Delta}|_{0} \cong \sigma^* \mathcal{O}_{x_1} \cong 
    \mathcal{O}_{x_1} \cong R_{\Delta}|_{0}.
\end{align}
\end{lemma}

An extension (\ref{ext:sigma}) defines a morphism 
\begin{align*}
    \tau_2^{\sigma} \colon \Delta \to \mathcal{H}. 
\end{align*}
We define 
\begin{align*}
    \cP_{\Delta 2}^{\sigma}:=(\tau_2^{\sigma}\times \id)^* \cP^{\flat}\in \Coh^{\heartsuit}(\cH_{\Delta}).
\end{align*}

Similarly to $\mathcal{P}_{\Delta i}$, the sheaf $\mathcal{P}_{\Delta 2}^{\sigma}$ is a Cohen--Macaulay sheaf on 
$\mathcal{H}_{\Delta}$ flat over $\Delta$ such that $\mathcal{P}_{\Delta 2}^{\sigma}|_{0} \cong \mathcal{P}_E$.
Moreover, we have the following lemma: 
\begin{lemma}\label{lem:inde}
The sheaf $\cP_{\Delta 2}^{\sigma}$ is 
independent of a choice of an extension (\ref{ext:sigma}) up to an isomorphism. 
\end{lemma}
\begin{proof}
By the construction, 
the sheaf $\mathcal{P}_{\Delta 2}^{\sigma}$ is a Cohen--Macaulay extension of a line bundle on 
$\Delta \times_{\cB} \mathcal{H}^{\mathrm{reg}}$, which is determined by the K-theory class of 
$[E_{\Delta}^{\sigma}]$. 
Since the complement of $\Delta \times_{\cB} \mathcal{H}^{\mathrm{reg}}$ in $\mathcal{H}_{\Delta}$ has 
codimension at least two by Lemma~\ref{lem:codimension} below, the lemma follows from Lemma~\ref{lem:MCM-extension}. 
\end{proof}

\begin{lemma}\label{lem:codimension}
The open immersion 
\begin{align}\label{jreg}
    j^{\mathrm{reg}}\colon \cH^{\mathrm{reg}}_{\Delta}:=\Delta\times_{\cB}\cH^{\mathrm{reg}}
    \hookrightarrow \cH_{\Delta}     
\end{align}
    is an isomorphism in codimension one. 
\end{lemma}
\begin{proof}
Since $\Delta$ is irreducible and $\Delta \to \cB$ sends its generic point to $\cB^{\mathrm{sm}}$, 
there is a closed subset $Z\subset \Delta$ with 
$\operatorname{codim}_{\Delta} Z \geq 1$ such that 
$j^{\mathrm{reg}}$ is an isomorphism over $\Delta \setminus Z$. 
Therefore $\cH_{\Delta} \setminus \cH_{\Delta}^{\mathrm{reg}}$ is contained in 
$\cH_{\Delta}\times_{\Delta} Z$. Since $\cH_{\Delta} \to \Delta$ is flat by 
Theorem~\ref{thm:higgs:basic}, the closed substack 
$\cH_{\Delta}\times_{\Delta} Z$ has codimension at least one in 
$\cH_{\Delta}$. Moreover, $\cH_{\Delta}\times_{\Delta} Z \to Z$ is flat, and 
each fiber contains the dense regular locus by Remark~\ref{rmk:dense}. Hence 
$\cH_{\Delta} \setminus \cH_{\Delta}^{\mathrm{reg}}$ has codimension at least 
one in $\cH_{\Delta}\times_{\Delta}Z$. Combining the two codimension estimates gives
$\operatorname{codim}_{\cH_{\Delta}}
(\cH_{\Delta} \setminus \cH_{\Delta}^{\mathrm{reg}})\geq 2$.
\end{proof}

The sheaf $R_{\Delta}$ in (\ref{exact:EQ}) is a $\Delta$-flat family of 
skyscraper sheaves of points on the fibers of $\cC_{\Delta} \to \Delta$. 
Therefore it determines a closed immersion
\begin{align*}
    \eta \colon \Delta \hookrightarrow \cC_{\Delta}
\end{align*}
corresponding to the support of $R_{\Delta} \in \Coh^{\heartsuit}(\cC_{\Delta})$. 
By shrinking $\Delta$ if necessary, we may assume that 
\begin{align*}
    R_{\Delta} \cong \eta_{*}\mathcal{O}_{\Delta}.
\end{align*}
We have the following commutative diagram 
\begin{align}\label{dia:delta1}
    \xymatrix{
& \Delta \ar[ld]_-{\eta} \ar[d]_-{\eta'} \ar[rd]^-{\eta''} & \\
\cC_{\Delta} \ar[r]^-{\pi} & C\times \Delta \ar[r]^-{\pi'} &
C.       }
\end{align}
Here the bottom arrows are projections. 
By taking $\times_{\cB}\cH$, it also induces the commutative diagram 
\begin{align}\label{dia:delta2}
    \xymatrix{
& \cH_{\Delta} \ar[ld]_-{\eta} \ar[d]_-{\eta'} \ar[rd]^-{\eta''} & \\
\cC_{\Delta}\times_{\Delta}\cH_{\Delta}\ar[r]^-{\pi} & C\times \cH_{\Delta} \ar[r]^-{\pi'} &
C \times \cH.       }
\end{align}
Here we have used the same symbols as in (\ref{dia:delta1}) by abuse of notation. 
We set $\cF_{\Delta}:=(\eta^{''})^{*}\cF$, which is a rank two 
vector bundle on $\cH_{\Delta}.$
We have the following proposition: 
\begin{prop}\label{prop:exact}
There is an exact sequence of vector bundles on $\cH_{\Delta}^{\mathrm{reg}}$
\begin{align}\label{exact:jF}
    0\to j^{\mathrm{reg}*} \cP_{\Delta 2}^{\sigma} \to j^{\mathrm{reg}*} \cF_{\Delta}
    \otimes j^{\mathrm{reg}*} \cP_{\Delta 1}\to j^{\mathrm{reg}*} \cP_{\Delta 2}\to 0.
\end{align}
Here $j^{\mathrm{reg}}$ is the open immersion (\ref{jreg}). 
\end{prop}
\begin{proof}
   We have the following diagram
   \begin{align*}
    \xymatrix{
& \cH_{\Delta}^{\mathrm{reg}}  \ar[d]_-{\eta}  & \\
\cC_{\Delta} &
\cC_{\Delta}\times_{\Delta}\cH_{\Delta}^{\mathrm{reg}}
\ar[r]^-{p_2} \ar[l]_-{p_1}&
\cH_{\Delta}^{\mathrm{reg}}.       }
\end{align*}
Here bottom arrows are projections. Let 
\begin{align*}\widetilde{\cE}\in \Coh^{\heartsuit}(\cC_{\Delta}\times_{\Delta}\cH_{\Delta}^{\mathrm{reg}})
\end{align*}
be the pull-back of the universal family (\ref{univ:E}) 
via $\cC_{\Delta}\times_{\Delta}\cH_{\Delta}^{\mathrm{reg}} \to \cC\times_{\cB}\cH$, which is a line bundle by the definition of $\cH_{\Delta}^{\mathrm{reg}}$. 
Then by the construction and using (\ref{exact:EQ}), we have an isomorphism 
of line bundles on $\cH_{\Delta}^{\mathrm{reg}}$
\begin{align}\notag
    j^{\mathrm{reg}*}\cP_{\Delta 2}&=\det p_{2*}(p_1^*E_{\Delta}\otimes 
   \widetilde{\cE})\otimes (\det p_{2*}p_1^*E_{\Delta})^{-1} \otimes (\det p_{2*}\widetilde{\cE})^{-1}\otimes \det p_{2*}\mathcal{O}\\
   \notag &\cong j^{\mathrm{reg}*}\cP_{\Delta 1}\otimes \det p_{2*}(p_1^* R_{\Delta}\otimes \widetilde{\cE}) \\
   \label{isom:P}&\cong j^{\mathrm{reg}*}\cP_{\Delta 1}\otimes \eta^* \widetilde{\cE}.
\end{align}
In a similar way, we also have an isomorphism 
\begin{align}\label{isom:P2}
      j^{\mathrm{reg}*}\cP_{\Delta 2}^{\sigma} \cong j^{\mathrm{reg}*}\mathcal{P}_{\Delta 1} \otimes (\sigma \circ \eta)^* \widetilde{\mathcal{E}}.
\end{align}

Below we use the same notation as in (\ref{dia:delta2}) for the 
restricted diagram 
\begin{align}\label{dia:delta3}
    \xymatrix{
& \cH_{\Delta}^{\mathrm{reg}} \ar[ld]_-{\eta} \ar[d]_-{\eta'} \ar[rd]^-{\eta''} & \\
\cC_{\Delta}\times_{\Delta}\cH_{\Delta}^{\mathrm{reg}}\ar[r]^-{\pi} & C\times \cH_{\Delta}^{\mathrm{reg}} \ar[r]^-{\pi'} &
C \times \cH^{\mathrm{reg}}.       }
\end{align}
The map $\pi$ in the diagram~(\ref{dia:delta1}) is a double cover with covering involution~$\sigma$, and we have the exact sequence in $\Coh^{\heartsuit}(\cC_{\Delta})$.

\begin{align*}
    0\to \sigma_{*}\eta_{*}\mathcal{O}_{\Delta} \to \pi^* \pi_{*}\eta_{*}\mathcal{O}_{\Delta} \to \eta_{*}\mathcal{O}_{\Delta} \to 0.
\end{align*}
By pulling it back to $\cC_{\Delta}\times_{\Delta}\cH_{\Delta}^{\mathrm{reg}}$
and applying $\otimes \widetilde{\cE}$, 
we obtain the exact sequence in $\Coh^{\heartsuit}(\cC_{\Delta}\times_{\Delta}\cH_{\Delta}^{\mathrm{reg}})$
\begin{align*}
     0\to p_1^*\sigma_{*}\eta_{*}\mathcal{O}_{\Delta} \otimes \widetilde{\cE}\to p_1^*\pi^* \pi_{*}\eta_{*}\mathcal{O}_{\Delta}\otimes \widetilde{\cE} \to p_1^*\eta_{*}\mathcal{O}_{\Delta}\otimes \widetilde{\cE} \to 0. 
\end{align*}

Then applying $\pi_{*}$ in the diagram (\ref{dia:delta3}), we obtain 
the exact sequence in $\Coh^{\heartsuit}(C\times \cH_{\Delta}^{\mathrm{reg}})$
\begin{align*}
     0\to \pi_{*}(p_1^*\sigma_{*}\eta_{*}\mathcal{O}_{\Delta} \otimes \widetilde{\cE})\to \pi_{*}(p_1^*\pi^* \pi_{*}\eta_{*}\mathcal{O}_{\Delta}\otimes \widetilde{\cE}) \to \pi_{*}(p_1^*\eta_{*}\mathcal{O}_{\Delta}\otimes \widetilde{\cE}) \to 0. 
\end{align*}
By applying the projection formula and noting that 
$(\eta')^{*}\pi_{*}\widetilde{\cE}\cong j^{\mathrm{reg}*}\cF_{\Delta}$, 
the above exact sequence is identified with  
\begin{align*}
    0\to \eta^{'}_{*}(\sigma \circ \eta)^* \widetilde{\cE} \to \eta^{'}_{*}j^{\mathrm{reg}*}\cF_{\Delta}\to \eta^{'}_{*}\eta^{*}\widetilde{\cE} \to 0.
\end{align*}
Since $\eta'$ is a closed immersion (as it is a section of the projection)
we obtain the exact sequence in $\Coh^{\heartsuit}(\cH_{\Delta}^{\mathrm{reg}})$
\begin{align*}
     0\to (\sigma \circ \eta)^* \widetilde{\cE} \to j^{\mathrm{reg}*}\cF_{\Delta}\to \eta^{*}\widetilde{\cE} \to 0.
\end{align*}
By applying $j^{\mathrm{reg}*}\cP_{\Delta 1} \otimes$ and using (\ref{isom:P}), (\ref{isom:P2}), we
obtain the desired exact sequence (\ref{exact:jF}). 
\end{proof}

As a corollary of the above proposition, we have the following: 
\begin{cor}\label{cor:CM}
    The exact sequence (\ref{exact:jF}) extends to a left exact sequence of maximal Cohen--Macaulay 
    sheaves on $\cH_{\Delta}$
    \begin{align}\label{extend:CM}
        0\to \cP_{\Delta 2}^{\sigma}\to \cF_{\Delta}\otimes \cP_{\Delta 1}
        \to \cP_{\Delta 2}.
    \end{align}
\end{cor}
\begin{proof}
    Let $j_{*}^{\mathrm{reg}\heartsuit}$ be the functor 
    \begin{align*}
    j_{*}^{\mathrm{reg}\heartsuit}:=\cH^0(j_{*}^{\mathrm{reg}}) \colon 
    \Coh^{\heartsuit}(\cH_{\Delta}^{\mathrm{reg}}) \to 
     \Coh^{\heartsuit}(\cH_{\Delta}).
    \end{align*}
    The above functor is left exact. 
    Note that the open immersion $\cH_{\Delta}^{\mathrm{reg}} \subset 
    \cH_{\Delta}$ is an isomorphism in codimension one by Lemma~\ref{lem:codimension}. Then as $\cP_{\Delta i}$ and $\cP_{\Delta2}^{\sigma}$ are 
    maximal Cohen--Macaulay sheaves and $\cF_{\Delta}$ is a vector bundle, we have 
    $j_{*}^{\mathrm{reg}\heartsuit}j^{\mathrm{reg}*} \cP_{\Delta i} \cong \cP_{\Delta i}$,  $j_{*}^{\mathrm{reg}\heartsuit}j^{\mathrm{reg}*} \cP_{\Delta 2}^{\sigma} \cong \cP_{\Delta 2}^{\sigma}$
    and $j_{*}^{\mathrm{reg}\heartsuit}j^{\mathrm{reg}*} \cF_{\Delta} \cong \cF_{\Delta}$.
    Therefore by applying $j_{*}^{\mathrm{reg}\heartsuit}$ to (\ref{exact:jF}), 
    we obtain the left exact sequence (\ref{extend:CM}). 
\end{proof}

In particular by taking the restriction of (\ref{extend:CM}) over $0\in \Delta$, we obtain 
the sequence in $\Coh^{\heartsuit}(\cH_b)$
\begin{align}\label{extend:CM2}
0\to \cP_E \to \cF_c \otimes \cP_{E'} \to \cP_E
\end{align}
which at this stage is neither left exact nor right exact. 
Here 
$c=\pi(x_1)\in C$ and $\cF_c :=\cF|_{c\times \cH_b}$. 
In what follows, we show that the sequence (\ref{extend:CM}) is right exact, which in particular implies that (\ref{extend:CM2}) is an 
exact sequence. 
We first show this in the case that $\cC_b$ has at worst nodal singularities. 
\begin{lemma}\label{lem:nodal}
If $\cC_b$ has at worst nodal singularities, the 
sequence (\ref{extend:CM}) is also right exact. In particular, the sequence (\ref{extend:CM2}) is an exact sequence. 
\end{lemma}
\begin{proof}
    The assumption on $\cC_b$ implies that 
    $m_i=1$ and 
    the sheaf $E'$ on $\cC_b$ is a line bundle at $x_1$. Let $\cC_{\Delta}'\subset \cC_{\Delta}$ be an open neighborhood of $x_1 \in \cC_b=\cC_{\Delta}\times_{\Delta}0$ such that 
    $E_t'$ is a line bundle on $\cC_{\Delta}'\times_{\Delta} t$
    for all $t\in \Delta$. 
    Then there is a morphism 
    \begin{align}\label{mor:circ}
\cC_{\Delta}' \to \cH_{\Delta}
    \end{align}
    sending $x\in \cC_{\Delta}'\times_{\Delta}t$ 
    to $E_x \in \Coh(\cC_t)$ which fits into 
    the unique non-trivial extension (up to isomorphisms)
    \begin{align*}
        0\to E_t' \to E_x \to \mathcal{O}_x \to 0. 
    \end{align*}
    
    On the other hand, since $\eta(0)=x_1 \in \cC_{\Delta}'$,
    by shrinking $\Delta$ if necessary we may assume that 
    $\eta \colon \Delta \to \cC_{\Delta}$ factors through 
    $\eta \colon \Delta \to \cC_{\Delta}'$. We define 
    $\widetilde{\Delta}$ to be the Cartesian square 
    \begin{align}\label{dia:tdelta}
        \xymatrix{
\widetilde{\Delta} \ar[d]_-{\widetilde{\pi}} \ar[rr] \ar@{}[rrd]|\square & & \cC_{\Delta}' \ar[d]_-{\pi} \\
\Delta \ar[r]^-{\eta} & \cC_{\Delta}' \ar[r]^-{\pi} & C\times \Delta.
        }
    \end{align}
The map $\pi \colon \widetilde{\Delta} \to \Delta$ is a $2\colon 1$ cover 
with involution $\sigma$ together with the image of a section $\Delta \subset \widetilde{\Delta}$ induced by $\eta$ such that 
$\widetilde{\Delta}=\Delta\cup \sigma(\Delta)$
and $\Delta \cap \sigma(\Delta)$ is a divisor in $\Delta$. By shrinking 
$\Delta$ if necessary, we may assume that $\mathcal{O}_{\Delta}(\Delta \cap \sigma(\Delta))$ is trivial, so that we have the exact sequence 
\begin{align}\label{exact:sigma}
    0\to \mathcal{O}_{\sigma(\Delta)} \to \mathcal{O}_{\widetilde{\Delta}}
    \to \mathcal{O}_{\Delta} \to 0. 
\end{align}

   By composing (\ref{mor:circ}) with the top horizontal arrow in (\ref{dia:tdelta}), 
   we obtain the map 
   \begin{align*}
       \widetilde{\tau} \colon \widetilde{\Delta} \to \cH_{\Delta}. 
   \end{align*}
   It yields the maximal Cohen--Macaulay sheaf 
   \begin{align*}
       \cP_{\widetilde{\Delta}}:=(\widetilde{\tau}\times \id)^* \cP^{\flat}\in \Coh^{\heartsuit}(\cH_{\widetilde{\Delta}})
       \end{align*}
       Here $\widetilde{\tau}\times \id$ is 
       \begin{align*}
           \widetilde{\tau}\times \id \colon \cH_{\widetilde{\Delta}}=\widetilde{\Delta}\times_{\Delta}\cH_{\Delta} \to 
           \cH_{\Delta}\times_{\Delta} \cH_{\Delta} \to \cH\times_{\cB} \cH.
       \end{align*}
       The exact sequence (\ref{exact:sigma}) then induces the exact
       sequence in $\Coh^{\heartsuit}(\cH_{\Delta})$
       \begin{align}\label{exact:Pdelta}
        0\to \cP_{\Delta 2}^{\sigma} \to \widetilde{\pi}_{*}\cP_{\widetilde{\Delta}} \to \cP_{\Delta 2} \to 0. 
       \end{align}
       Here we have used the same notation $\widetilde{\pi}$ for the 
       map 
       $\cH_{\widetilde{\Delta}} \to \cH_{\Delta}$.
       
Unraveling the construction, the above exact sequence restricted to 
$\cH_{\Delta}^{\mathrm{reg}}$ is identified with (\ref{exact:jF}). 
Since $\cH_{\Delta}^{\mathrm{reg}} \subset \cH_{\Delta}$ is an isomorphism in codimension one, and both of $\widetilde{\pi}_{*}\cP_{\widetilde{\Delta}}$ and $\cF_{\Delta}\otimes \cP_{\Delta 1}$ are maximal Cohen--Macaulay sheaves, we 
conclude that \begin{align*}
    \widetilde{\pi}_{*}\cP_{\widetilde{\Delta}}\cong 
   \cF_{\Delta}\otimes \cP_{\Delta 1}. 
\end{align*}
Therefore from the exact sequence (\ref{exact:Pdelta}), we obtain the lemma. 
       
\end{proof}

Using the above lemma for the nodal case, we prove that (\ref{extend:CM2}) is exact 
in general:
\begin{prop}\label{prop:exact2}
The sequence (\ref{extend:CM2}) is an exact sequence
\begin{align}\notag
0\to \cP_E \to \cF_c\otimes \cP_{E'} \to \cP_E \to 0.
\end{align}
\end{prop}
\begin{proof}
Let $\cB^{\mathrm{sg}} \subset \cB$ be the closed subset corresponding 
to the singular spectral curves. It is of codimension one, and its 
generic point corresponds to those with at worst nodal singularities (cf.~\cite[Corollary 1.5 and Remark 1.7]{KouvidakisPantev1995}). 
Then we can take $\Delta$ so that $\dim \Delta =2$ and 
the image of the composition 
\begin{align*}\Delta \setminus \{0\} \hookrightarrow \Delta \to \cB
\end{align*}
corresponds to the locus of spectral curves which are either smooth or have at worst nodal singularities. Then from Lemma~\ref{lem:nodal}, the sequence (\ref{extend:CM}) is 
exact on $\Delta \setminus \{0\}$. 

Assume that the sequence (\ref{extend:CM}) is not exact.
Let $\mathbb{P}_{\Delta}$ be the complex (\ref{extend:CM})
regarded as an object in $\Coh(\cH_{\Delta})$, where 
$\cP_{\Delta 2}$ is located in degree zero. 
Since (\ref{extend:CM}) is left exact, the object $\mathbb{P}_{\Delta}$ is a non-zero object
in $\Coh^{\heartsuit}(\cH_{\Delta})$ and supported 
over $0\in \Delta$. 
Let $i \colon \Delta' \subset \Delta$ be a generic smooth divisor 
such that $0\in \Delta'$.
We use the same notation $i$ for the induced map $i \colon \cH_{\Delta'} \hookrightarrow \cH_{\Delta}$. 
Then as $\mathbb{P}_{\Delta}$ is supported over $0\in \Delta'$, 
it follows that 
\begin{align*}
0\neq \cH^{-1}(i^* \mathbb{P}_{\Delta}) \in \Coh^{\heartsuit}(\cH_{\Delta'}).
\end{align*}

On the other hand, 
since each term of $\mathbb{P}_{\Delta}$ is flat over $\Delta$, 
we have $i^* \mathbb{P}_{\Delta}=\mathbb{P}_{\Delta'}$ and it 
lies in $\Coh^{\heartsuit}(\cH_{\Delta'})$ by
the left exactness of
the sequence 
   \begin{align}\notag
        0\to \cP_{\Delta' 2}^{\sigma}\to \cF_{\Delta'}\otimes \cP_{\Delta' 1}
        \to \cP_{\Delta' 2}
    \end{align}
    which is obtained as in 
(\ref{extend:CM}) by replacing $\Delta$ with $\Delta'$. 
It contradicts $\cH^{-1}(i^*\mathbb{P}_{\Delta})\neq 0$,
hence the sequence (\ref{extend:CM}) is exact. 
Then by taking the restriction of (\ref{extend:CM}) to $0\in \Delta$, 
we conclude that the sequence (\ref{extend:CM2}) is exact. 
\end{proof}

Recall the line bundle $\mathcal{L}$ on $\cC_b$ in 
Lemma~\ref{lem:LEQ}. As a corollary, we obtain the following explicit 
resolution of $\cP_E$ in terms of vector bundles: 
\begin{cor}\label{cor:resolP}
Let $c_i=\pi(x_i)\in C$, and $\mathcal{V}_E$ be the following 
vector bundle on $\cH_b$
\begin{align}\label{def:VE}
    \mathcal{V}_E:=\bigotimes_{i=1}^{\ell} \cF_{c_i}^{\otimes n_i} \otimes 
    \cP_{\mathcal{L}}. 
\end{align}
There is a left resolution of $\cP_E$ in $\Coh^{\heartsuit}(\cH_b)$
of the form for some $k_i \in \mathbb{Z}_{\geq 1}$
\begin{align}\notag
\cdots \to \mathcal{V}_E^{\oplus k_2} \to \mathcal{V}_E^{\oplus k_1} \to \cP_E \to 0. 
\end{align}
\end{cor}
\begin{proof}
We prove the corollary by the induction on $\chi(Q)\geq 0$, where $Q$ is given in Lemma~\ref{lem:LEQ}. 
The base case, $\chi(Q)=0$, is obvious since $E=\mathcal{L}$ in this case. 

Suppose that $\chi(Q)>0$. 
    By Proposition~\ref{prop:exact2}, 
    there is an exact sequence of the form 
    \begin{align*}
        \cdots\to \cF_{c_1}\otimes \cP_{E'} \to \cF_{c_1}\otimes \cP_{E'} \to 
        \cF_{c_1}\otimes \cP_{E'} \to \cP_E \to 0.
    \end{align*}
    Since $\chi(Q')=\chi(Q)-1$ in the sequence (\ref{Qprime}), 
    we can apply the induction hypothesis for $\cP_{E'}$. Then the claim 
    for $\cP_E$ holds by the above exact sequence together with 
    $\mathcal{F}_{c_1}\otimes \mathcal{V}_{E'}\cong \mathcal{V}_E$. 
\end{proof}
\subsection{Arinkin sheaf and limit category}
For $E\in \Coh^{\heartsuit}(\cC_b)$, 
let $\deg \pi_{*}E=w$ so that $\deg_{\cC_b}(E)=w':=w+2g-2$. 
In this subsection, using the resolution of the Arinkin sheaf in the previous 
subsection, we show that it lies in the limit category. 
Let $\mathcal{V}_E$ be the vector bundle given by (\ref{def:VE}). 
We prove the following proposition: 
\begin{prop}\label{prop:VL}
    Suppose that $E\in \Coh^{\heartsuit}(\cC_b)$ is semistable. Then we have 
    \begin{align}\label{VEtilde}
        \mathcal{V}_E \in \widetilde{\LL}(\cH)_{w'},    
    \end{align}
        see Remark~\ref{rmk:Ltilde} for the right-hand side. 
        Here we have omitted the notation push-forward along with the closed immersion $\cH_b \hookrightarrow \cH$.
        Moreover if $E$ is stable, then 
        \begin{align*}
             \mathcal{V}_E|_{\cH(\chi)} \in \mathrm{Im}(j_{!}) \subset 
             \LL(\cH(\chi))_{w'}. 
        \end{align*}
        Here $j_{!}$ is the fully-faithful left adjoint of $j^*$
        for the open immersion 
        $j\colon \cH(\chi)^{\mathrm{ss}} \hookrightarrow \cH(\chi)$
        \begin{align*}
            j_{!} \colon
            \LL(\cH(\chi)^{\mathrm{ss}})_{w'} \hookrightarrow
            \LL(\cH(\chi))_{w'}. 
        \end{align*}
\end{prop}

\begin{proof}
We first prove (\ref{VEtilde}) when $E$ is semistable. 
    Let $\nu$ be a map 
    \begin{align*}
    \nu \colon \bgm \to \cH, \ \mathrm{pt} \mapsto M_1 \oplus M_2
    \end{align*}
    for $M_i \in \Coh(\cC_b)$ with $\mathbb{G}_m$-weights $(\mu_1, \mu_2)$
    such that $\mu_1>\mu_2$. 
Since we have 
\begin{align*}
    \wt(\nu^* \mathbb{L}_{\cH}^{\nu>0})=(2g-2)(\mu_1-\mu_2), \
    \wt(\nu^* \mathbb{L}_{\cH}^{\nu<0})=(2g-2)(\mu_2-\mu_1), 
\end{align*}
it is enough to show that 
\begin{align}\label{wt:show}
    \wt(\nu^* \mathcal{V}_E) \subset 
    (g-1)(\mu_1-\mu_2) \cdot [-1, 1]
    +\frac{w'}{2}(\mu_1+\mu_2). 
\end{align} 

We may assume that $M_i$ is supported on $C_i$. 
Let $\mathcal{L}$ and $n_i$ be taken for $E\in \Coh^{\heartsuit}(\cC_b)$ as in Lemma~\ref{lem:LEQ}, 
and set $l_i:=\deg(\mathcal{L}|_{C_i})$. 
Then we have 
\begin{align}\label{eq:chiprime}
    l_1+l_2+\sum_{i=1}^{\ell} n_i=w'. 
\end{align}
By the definition of $\mathcal{V}_E$, a $\mathbb{G}_m$-weight $\mu$ of 
$\nu^*{\mathcal{V}_E}$ is of the form 
\begin{align}\label{wt:mu}
    \mu=l_1 \mu_1+l_2 \mu_2+\sum_{i=1}^{\ell} \left(a_i \mu_1+(n_i-a_i)\mu_2 \right)
\end{align}
for $0\leq a_i \leq n_i$. Here the first terms $l_1 \mu_1+l_2 \mu_2$ is the 
$\mathbb{G}_m$-weight of $\nu^* \cP_{\mathcal{L}}$, which is computed from 
(\ref{def:Psharp}). 
By using (\ref{eq:chiprime}), we can write (\ref{wt:mu}) as 
\begin{align*}
    \mu-\frac{w'}{2}(\mu_1+\mu_2)=\left(l_1+\sum_{i=1}^{\ell} a_i-\frac{w'}{2}\right)(\mu_1-\mu_2).
\end{align*}

We have the exact sequence 
\begin{align*}
    0 \to \Ker(\phi) \to E \stackrel{\phi}{\to} (E|_{C_2})^{\mathrm{fr}} \to 0,
\end{align*}
where $(-)^{\mathrm{fr}}$ is the torsion-free quotient. 
Using the exact sequence in Lemma~\ref{lem:LEQ}, the above exact sequence 
is identified with 
\begin{align*}
    0\to \mathcal{L}' \to E \to \mathcal{L}|_{C_2}\to 0, 
\end{align*}
where 
$\mathcal{L}'$ is a line bundle on $C_1$ with 
degree $l_1-2g+2+\sum_{i=1}^{\ell} n_i$. 
    By the semistability of $E$, we have 
    \begin{align}\label{ineq:stab}
        l_1-2g+2+\sum_{i=1}^{\ell} n_i \leq l_2. 
    \end{align}
    By substituting (\ref{eq:chiprime}), the above inequality is equivalent to 
    \begin{align*}
        l_1 \leq g-1-\sum_{i=1}^{\ell} n_i +\frac{w'}{2}. 
    \end{align*}
    Therefore 
    \begin{align*}
        \mu-\frac{w'}{2}(\mu_1+\mu_2)\leq (g-1)(\mu_1-\mu_2).
    \end{align*}
    
    The lower bound 
    \begin{align}\label{low:bound}
      (g-1)(\mu_2-\mu_1) \leq   \mu-\frac{w'}{2}(\mu_1+\mu_2)
    \end{align}
    is similarly obtained, 
    using the exact sequence 
    \begin{align}\label{exact:Lprime}
        0\to \mathcal{L}'' \to E \to \mathcal{L}|_{C_1} \to 0
    \end{align}
    where $\mathcal{L}''$ is a line bundle on $C_2$ of degree $l_2-2g+2+\sum_{i=1}^{\ell} n_i$. 
    Therefore (\ref{wt:show}) is proved. 

Suppose that $E$ is stable. The existence of $j_{!}$ and the characterization of its 
image for the open immersion
$j\colon \cH(\chi)^{\mathrm{ss}}\hookrightarrow\cH(\chi)$
is given in Theorem~\ref{thm:left} and Remark~\ref{rmk:jshrink}, Remark~\ref{rmk:variant}.  
    If $E$ is stable, the inequality (\ref{ineq:stab}) is strict, 
    hence we obtain the inclusion (\ref{wt:show}) without boundary. 
    Therefore we obtain the second statement, see Remark~\ref{rmk:jshrink}. 
\end{proof}

We obtain the following corollary, which gives a connection of the Arinkin's construction of Cohen--Macaulay extension and the limit category:
\begin{cor}\label{cor:PE}
    If $E$ is semistable, we have 
    \begin{align*}
        \cP_E \in \widetilde{\LL}(\cH)_{w'}.
    \end{align*}
    If $E$ is furthermore stable, we have 
    \begin{align*}
        \cP_E|_{\cH(\chi)} \in \Im(j_{!}) \subset \LL(\cH(\chi))_{w'}.
    \end{align*}
\end{cor}
\begin{proof}
    The corollary follows from Corollary~\ref{cor:resolP}, Proposition~\ref{prop:VL} and Lemma~\ref{lem:perfect}. 
\end{proof}

\subsection{Preservation of quasi-BPS categories}
Let 
\begin{align*}
 \LL(\cH(\chi)^{\mathrm{ss}})_w\subset \Coh(\cH(\chi)^{\mathrm{ss}})_{w}
\end{align*}
be the quasi-BPS category over $\cB$. 
Note that we have 
\begin{align*}
    \mathcal{O}_{\cH(w)^{\mathrm{ss}}} \in \LL(\cH(w)^{\mathrm{ss}})_0.
\end{align*}
Using the results in the previous subsections, 
we prove the following: 

\begin{prop}\label{prop:induceT}
The functor $\Phi$ in (\ref{induce:Phi})
satisfies that 
\begin{align*}
\Phi( \mathcal{O}_{\cH(w)^{\mathrm{ss}}}) \in j_{!}\LL(\cH(\chi_0)^{\mathrm{ss}})_{w'}.
\end{align*}
Here $j \colon \cH(\chi_0)^{\mathrm{ss}} \hookrightarrow \cH(\chi_0)$ 
is the open immersion. 
\end{prop}
\begin{proof}
It is enough to check the weight bounds near points whose spectral curve is reducible, since any Higgs bundle with integral spectral curve is 
automatically (semi)stable. 
By Lemma~\ref{lem:commute}, we may assume that $w<0$. 
The proof is long, so we divide it into 6 steps. 

\begin{sstep}\label{sstep1}
    Reduction to an \'{e}tale neighborhood of each point in the Hitchin base. 
\end{sstep}
We take $b\in \cB$ such that the spectral curve 
$\cC_b$ is reducible 
\begin{align*}\cC_b=C_1 \cup C_2, \
C_1 \neq C_2.
\end{align*}
Let $T' \hookrightarrow \cC$ be a general hypersurface such that $T' \to \cB$ is 
\'{e}tale at $b$ and intersects at $\cC_b$ with its smooth locus. 
Let $0\in T'$ be a lift of $b$ such that it lies in $C_1$ 
under the embedding $T'\hookrightarrow \cC$. 
We take an open neighborhood $0\in T \subset T'$ such that $T \to \cB$ is 
\'{e}tale, and $T\cap \cC_{b'}$ for any $b'\in \cB$ lies in the smooth locus of $\cC_{b'}$. 
We have an \'{e}tale neighborhood 
    $(T, 0) \to (\cB, b)$, and set 
\begin{align*}\cC_T:=\cC\times_{\cB}T, \ 
\cH_T(w):=\cH(w)\times_{\cB} T.
\end{align*}
By the construction, the projection $\cC_T \to T$ admits a section 
\begin{align}\label{sec:eta}
    e \colon T \hookrightarrow \cC_T.
    \end{align}
For $t \in T$ such that $\cC_t$ is reducible, 
we write $\cC_t=C_{t1} \cup C_{t2}$ where $C_{t1}$ intersects with the section (\ref{sec:eta}). 

Let
\begin{align*}
    \cP_T \in \Coh^{\heartsuit}(\cH_T(w)^{\mathrm{ss}} \times_T \cH_T(\chi_0))
\end{align*}
    be the pull-back of the Arinkin sheaf (\ref{Arsheaf0})
    by the projection to $\cH^{\mathrm{ss}}\times_{\cB}\cH$.
    The Fourier--Mukai functor with kernel $\cP_T$ gives the functor 
    \begin{align*}
        \Phi_T \colon \Coh(\cH_T(w)^{\mathrm{ss}})_{0}
        \to \Coh(\cH_T(\chi_0))_{w'}. 
    \end{align*}
    By the characterization of the image of $j_{!}$ in Remark~\ref{rmk:jshrink}, it 
    is enough to show that 
    \begin{align}\notag
    \Phi_T(\mathcal{O}_{\cH_T(w)^{\mathrm{ss}}}) \in j_{T!}\LL(\cH_T(\chi_0)^{\mathrm{ss}})_{w'}
    \end{align}
    for the open immersion $j_T \colon \cH_T(\chi_0)^{\mathrm{ss}} \hookrightarrow \cH_T(\chi_0)$. 

Let $\nu$ be a map 
\begin{align}\label{nu:F12}
    \nu \colon \bgm \to \cH_T(\chi_0), \ \mathrm{pt} \mapsto M=M_1 \oplus M_2
\end{align}
with $\mathbb{G}_m$-weight $(\mu_1, \mu_2)$, $\mu_1>\mu_2$, 
such that $M_i$ is a line bundle on $C_{ti}$ for $t\in T$
satisfying 
\begin{align}\label{sum:chi}\chi_i=\deg M_i, \ \chi_1+\chi_2=\chi_0=2-2g, \  
\chi_1 \geq \chi_2.
\end{align}
For $\chi_1>\chi_2$, it corresponds to the center of a Harder--Narasimhan stratum of 
$\cH(\chi_0)$, see Remark~\ref{rmk:jshrink}.
We need to show the following weight bound 
\begin{align}\label{wt:bound}
    \wt(\iota_{*}\nu^{\mathrm{reg}*}\Phi_T(\mathcal{O}_{\cH_T(w)^{\mathrm{ss}}})) \subset (g-1)(\mu_1-\mu_2)\cdot [-1, 1] +\frac{w'}{2}(\mu_1+\mu_2), 
\end{align}
for $\chi_1=\chi_2$ and 
\begin{align}\label{wt:bound2}
    \wt(\iota_{*}\nu^{\mathrm{reg}*}\Phi_T(\mathcal{O}_{\cH_T(w)^{\mathrm{ss}}})) \subset (g-1)(\mu_1-\mu_2)\cdot (-1, 1] +\frac{w'}{2}(\mu_1+\mu_2), 
\end{align}
when $\chi_1>\chi_2$. 

\begin{sstep}
    Perturbation of stability condition. 
\end{sstep}
Let $h$ be an ample divisor on $C$ of degree one and we denote by $h_{b}$ 
its pull-back by the projection $\cC_b \to C$. 
We perturb it to a $\mathbb{Q}$-ample divisor $h_{\varepsilon}$ on $\cC_b$ such that 
\begin{align}\label{ineq:he}
\deg h_{\varepsilon}|_{C_1}=1-\varepsilon, \
\deg h_{\varepsilon}|_{C_2}=1+\varepsilon
\end{align}
for $\varepsilon \in \mathbb{Q}$ with $0<\varepsilon \ll 1$. 
By taking the \'{e}tale neighborhood $T\to \cB$ as in the previous step, 
we can extend  
$h_{\varepsilon}$ to a $\mathbb{Q}$-ample 
    divisor $h_{T, \varepsilon}$ on $\cC_T$. 
Indeed we can take 
\begin{align*}
    h_{T, \varepsilon}=(1+\varepsilon)\cdot h_T -2\varepsilon \cdot e(T)
\end{align*}
where $h_T$ is the pull-back of $h$ under the projection $\cC_T \to C$ and 
$e$ is the section (\ref{sec:eta}).
    We consider open substacks 
    \begin{align*}
        \cH_T(w)^{\varepsilon\text{-ss}} \subset \cH_T(w)^{\mathrm{ss}} \subset \cH_T(w)
    \end{align*}
    where $\cH_T(w)^{\varepsilon\text{-ss}}$ consists of $h_{T, \varepsilon}$-semistable sheaves on the fibers of $\cC_T\to T$. 
    By a choice of $h_{\varepsilon}$, it consists of 
    $h_{T, \varepsilon}$-stable sheaves; otherwise there is $a\in \mathbb{Z}$
    satisfying $a/(1-\varepsilon)=w/2$, which cannot happen for $0<\varepsilon \ll 1$. 
Therefore the morphism to the good moduli space 
\begin{align}\label{good:Gm}
    \cH_T(w)^{\varepsilon\text{-ss}} \to \mathrm{H}_T(w)^{\varepsilon\text{-ss}}
\end{align}
is a $\mathbb{G}_m$-gerbe and $\mathrm{H}_T(w)^{\varepsilon\text{-ss}}$ is a smooth variety.
Moreover we have an equivalence 
\begin{align}\notag
    \Coh( \cH_T(w)^{\varepsilon\text{-ss}})_{0} \simeq \Coh(\mathrm{H}_T(w)^{\varepsilon\text{-ss}}). 
\end{align}
    
    The complement 
    \begin{align*}
        \cS=\cH_T(w)^{\mathrm{ss}} \setminus 
         \cH_T(w)^{\varepsilon\text{-ss}}
    \end{align*}
    is a \textit{$\Theta$-stratum} of $\cH_T(w)^{\mathrm{ss}}$, see~\cite{Halpinstab} for the notation of 
    $\Theta$-stratum. It consists of 
    $E \in \Coh^{\heartsuit}(\cC_t)$ for $t\in T$
which admits an exact sequence
    \begin{align}\label{exact:LF}
        0\to \mathcal{L}_2 \to E \to \mathcal{L}_1 \to 0
    \end{align}
   where $\mathcal{L}_i$ is a line bundle on $C_{ti}$ with 
     degree $w/2$ (so $\cS=\emptyset$ if $w$ is odd). Note that by the condition (\ref{ineq:he})
     and $w<0$, we have 
     \begin{align}\label{ineq:slope}
         \frac{w/2}{h_{T, \varepsilon} \cdot C_{t2}}>  \frac{w/2}{h_{T, \varepsilon} \cdot C_{t1}}. 
     \end{align}
     This implies that the exact sequence (\ref{exact:LF}) is a Harder--Narasimhan filtration of $E$
     with respect to $h_{T, \varepsilon}$.

    Let $g\colon \cS \to \mathcal{Z}$ be the center of $\cS$. 
It corresponds to direct sums $\mathcal{L}_1 \oplus \mathcal{L}_2$, where $\mathcal{L}_i$ is 
a line bundle on $C_{ti}$ as above.
  Moreover the stack $\cS$ is quasi-smooth, the closed immersion 
  $\cS \hookrightarrow \cH_T(w)^{\mathrm{ss}}$ is quasi-smooth 
  and the natural map 
  $g\colon \cS \to \mathcal{Z}$
    is a vector bundle with fiber 
    \begin{align}\label{ext}
        \Ext_S^1(\mathcal{L}_1, \mathcal{L}_2) =\mathbb{A}^{2g-2}.
    \end{align}
These follow from the computations of tangent complexes
\begin{align}\label{tan:comp}
    \mathbb{T}_{\cS/\cH_T(w)^{\mathrm{ss}}}|_{(0\to \mathcal{L}_2 \to E \to \mathcal{L}_1 \to 0)} &=\Hom_S(\mathcal{L}_2, \mathcal{L}_1) \\
  \notag  \mathbb{T}_{\cS/\mathcal{Z}}|_{(0\to \mathcal{L}_2 \to E \to \mathcal{L}_1 \to 0)}&=\Hom_S(\mathcal{L}_1, \mathcal{L}_2)[1]
\end{align}
and noting that 
\begin{align*}
\Hom_S(\mathcal{L}_1, \mathcal{L}_2)=k^{2g-2}[-1], \ \Hom_S(\mathcal{L}_2, \mathcal{L}_1)=k^{2g-2}[-1].
\end{align*}

    For a point $z \in \mathcal{Z}(k)$ corresponding to the 
    direct sum $\mathcal{L}_1 \oplus \mathcal{L}_2$, there is a canonical 
    map $\bgm \to \mathcal{Z}$ with image $z$, and the $\mathbb{G}_m$-weight 
    is given by 
    \begin{align}\label{wt:lambda}
    \lambda=(\lambda_1, \lambda_2), \ \lambda_1<\lambda_2
    \end{align}
    (so that (\ref{ext}) has positive $\lambda$-weight $\lambda_2-\lambda_1>0$ and $g_{*}\mathcal{O}_{\cS}$ has non-positive $\lambda$-weights). 
    Since $\mathcal{Z}$ is connected, $\lambda$ is independent of $z$. 
    We have the semiorthogonal decomposition 
    \begin{align*}
        \Coh(\cS)=\left\langle \ldots, \Coh(\cS)_{-1}, \Coh(\cS)_0, \Coh(\cS)_1, \ldots   \right\rangle
    \end{align*}
    with an equivalence $g^* \colon \Coh(\mathcal{Z})_a \stackrel{\sim}{\to} \Coh(\cS)_a$, where 
    $\Coh(\mathcal{Z})_a \subset \Coh(\mathcal{Z})$ corresponds to the 
    $\lambda$-weight $a$-part, see~\cite[Proposition 1.1.2, Lemma 1.5.6]{HalpK32}.
\begin{sstep}
    Reduction to the case of perturbed stability condition. 
\end{sstep}

We consider the following diagram 
\begin{align*}
    \xymatrix{
& \cH_T(w)^{\varepsilon\text{-ss}} \times_T \cH_T(\chi_0) \ar[ld] \inclusion^-{\gamma} &  \cH_T(w)^{\mathrm{ss}} \times_T \cH_T(\chi_0)
\ar[rd]_-{p_T} \ar[ld]^-{q_T} & \\
 \cH_T(w)^{\varepsilon\text{-ss}} \inclusion
 & \cH_T(w)^{\mathrm{ss}}  &  & \cH_T(\chi_0).
    }
\end{align*}
    We set 
    \begin{align}\label{PTeps}\cP_T^{\varepsilon}:=\gamma^* \cP_T \in \Coh^{\heartsuit}(\cH_T(w)^{\varepsilon\text{-ss}} \times_T \cH_T(\chi_0))
    \end{align}
    and 
    consider the Fourier--Mukai functor with kernel $\cP_T^{\varepsilon}$
    \begin{align*}
        \Phi_T^{\varepsilon}\colon \Coh(\cH_T(w)^{\varepsilon\text{-ss}})_{0} \to \Coh(\cH_T(\chi_0))_{w'}.
    \end{align*}
Then we have the distinguished triangle in $\Coh(\cH_T(\chi_0))_{w'}$
\begin{align}\label{dist:A}
    p_{T*}(\Gamma_{q_T^{-1}(\cS)}(\cP_T))
    \to \Phi_T(\mathcal{O}_{\cH_T(w)^{\mathrm{ss}}}) \to \Phi_T^{\varepsilon}(\mathcal{O}_{\cH_T(w)^{\varepsilon\text{-ss}}}), 
\end{align}
see (\ref{loc:coh}) for the local cohomology functor. 
For the map $\nu$ in (\ref{nu:F12}), we have 
\begin{align}\label{vanish:Gamma}
    \nu^* p_{T*}(\Gamma_{q_T^{-1}(\cS)}(\cP_T))\cong 
    \Gamma(\cH_t(w)^{\mathrm{ss}}, \Gamma_{\cS_t}(\cP_M)) \in \QCoh(\bgm). 
\end{align}

In the next step, 
we show that the above object vanishes. 
Therefore in order to show the weight bounds (\ref{wt:bound}), (\ref{wt:bound2})
it is enough to show the same conditions for $\Phi_T^{\varepsilon}(\mathcal{O}_{\cH_T(w)^{\varepsilon\text{-ss}}})$, 
\begin{align}\label{wt:eps}
   \wt( \iota_{*}\nu^{\mathrm{reg}*}\Phi_T^{\varepsilon}(\mathcal{O}_{\cH_T(w)^{\varepsilon\text{-ss}}})) \subset (g-1)(\mu_1-\mu_2)\cdot [-1, 1] +\frac{w'}{2}(\mu_1+\mu_2)
\end{align}
where the left boundary weight is not attained when $\chi_1>\chi_2$.

\begin{sstep}\label{sstep4}
    The object (\ref{vanish:Gamma}) is zero. 
\end{sstep}
We show that (\ref{vanish:Gamma}) vanishes. 
Note that the object $M=M_1\oplus M_2$ in (\ref{nu:F12}) fits into 
an exact sequence 
\begin{align*}
    0\to \widetilde{\mathcal{L}} \to M \to \mathcal{O}_{C_{t1}\cap C_{t2}} \to 0. 
\end{align*}
Here
the line bundle $\widetilde{\mathcal{L}}$ on $\cC_t$ satisfies $\deg(\widetilde{\mathcal{L}}|_{C_{ti}})=\chi_i$. 
Therefore by Corollary~\ref{cor:resolP}, it is enough to show that 
\begin{align}\label{gamma:t}
\Gamma(\cH_t(w)^{\mathrm{ss}},\Gamma_{\cS_t}(\mathcal{V}_M))=0
\end{align}
where $\mathcal{V}_M$ is a vector bundle of the form 
\begin{align*}
    \mathcal{V}_M=\bigotimes_{i=1}^{\ell} \cF_{c_i}^{\otimes m_i} \otimes \cP_{\widetilde{\mathcal{L}}} \in \Coh^{\heartsuit}(\cH_t(w)^{\mathrm{ss}}).
\end{align*}
Here $c_i \in C$ for $1\leq i\leq \ell$ and $m_i \in \mathbb{Z}_{>0}$ satisfy 
$\sum_{i=1}^{\ell} m_i=2g-2$.

The local cohomology functor is calculated by 
\begin{align}\label{Gamma:St}
    \Gamma_{\cS_t}(\mathcal{V}_M)=\operatorname*{colim}_{\cS_t \subset \cS'}
    \cHom_{\cH_t}(\mathcal{O}_{\cS'}, \mathcal{V}_M).
\end{align}
Here the colimit is taken for closed substacks $\cS' \subset \cH_t(w)^{\mathrm{ss}}$ with $\mathrm{Supp}(\cS')=\mathrm{Supp}(\cS_t)$. 
We have 
\begin{align}\label{HomOS}
    \cHom_{\cH_t}(\mathcal{O}_{\cS_t},  \mathcal{V}_M)= \mathcal{V}_M|_{\cS_t} \otimes \det \mathbb{L}_{\cS_t/\cH_t}[-d]. 
\end{align}
Here $d$ is the codimension $\cS_t$ in $\cH_t$. 
We show that it has negative $\lambda$-weights, where $\lambda=(\lambda_1, \lambda_2)$
with $\lambda_1<\lambda_2$ is 
given by (\ref{wt:lambda}). Indeed, by (\ref{tan:comp})
$\mathbb{L}_{\cS_t/\cH_t}[-1]$ is a vector bundle on $\cS_t$ of rank $2g-2$ with $\lambda$-weights $(\lambda_2-\lambda_1)>0$. 
Therefore 
$\det \mathbb{L}_{\cS_t/\cH_t}$ has $\lambda$-weight 
$(2g-2)(\lambda_1-\lambda_2)<0$. 
It follows that, noting the identity (\ref{sum:chi}), a $\lambda$-weight of (\ref{HomOS}) is 
at most 
\begin{align*}
&\chi_1 \lambda_1+\chi_2 \lambda_2 
+\sum_{i=1}^{\ell} (a_i \lambda_1+(m_i-a_i)\lambda_2)+(2g-2)(\lambda_1-\lambda_2)\\
&=\left(\chi_1+2g-2+\sum_{i=1}^{\ell} a_i \right)(\lambda_1-\lambda_2) < 0
\end{align*}
since $\chi_1 \geq\chi_0/2=1-g$ and $a_i \geq 0$, $g\geq 2$. 
Therefore (\ref{HomOS}) has negative $\lambda$-weights. 

There is a sequence of closed immersions 
\begin{align*}
    \cS_t=\cS_0 \hookrightarrow \cS_1 \hookrightarrow \cS_2 \hookrightarrow 
    \cdots \hookrightarrow 
\end{align*}
where each $\cS_n$ is a $n$-th infinitesimal thickening of $\cS_t$ in $\cH_t$, 
whose structure sheaves fit into distinguished triangles 
\begin{align*}
    \mathrm{Sym}^{n}(\mathbb{L}_{\cS_t/\cH_t}[-1]) \to \mathcal{O}_{\cS_{n+1}} \to \mathcal{O}_{\cS_n}.
\end{align*}
Then $\cHom_{\cH_t}(\mathcal{O}_{\cS_n}, \mathcal{V}_M)$ is filtered by 
\begin{align}\label{filt:negative}
\mathcal{V}_M|_{\cS_t} \otimes \det \mathbb{L}_{\cS_t/\cH_t}[-d]\otimes\mathrm{Sym}^i(\mathbb{L}_{\cS_t/\cH_t}[-1])^{\vee}    
\end{align}
for $0\leq i\leq n-1$. Since the first two factors and the last factor have negative $\lambda$-weights, 
we conclude that (\ref{filt:negative}) has negative 
$\lambda$-weights. 

Now for the good moduli space map 
$r_t \colon \cH_t(w)^{\mathrm{ss}} \to \mathrm{H}_t(w)^{\mathrm{ss}}$, we have 
the commutative diagram 
\begin{align*}
    \xymatrix{
    \cH_t(w)^{\mathrm{ss}} \ar[rr]^-{r_t} & &\mathrm{H}_t(w)^{\mathrm{ss}} \\
    \cS_t \uinclusion \ar[r]^-{g_t} & \mathcal{Z}_t \ar[r] & Z_t \uinclusion
    }
\end{align*}
Here the top and the right bottom horizontal arrows are good moduli space morphisms.  
Since (\ref{filt:negative}) has negative $\lambda$-weights, 
its push-forward $g_{t*}$ has also negative $\lambda$-weights, therefore 
its push-forward to $Z_t$ is zero. 
Therefore from the above commutative diagram, we have 
\begin{align}\label{vanish:rt}
    r_{t*}\cHom_{\cH_t}(\mathcal{O}_{\cS_n}, \mathcal{V}_M)=0. 
\end{align}

Note that we have 
\begin{align*}
\operatorname*{colim}_{\cS_t \subset \cS'}
    \cHom_{\cH_t}(\mathcal{O}_{\cS'}, \mathcal{V}_M)
    = \operatorname*{colim}_{n\geq 0}\cHom_{\cH_t}(\mathcal{O}_{\cS_n}, \mathcal{V}_M).
    \end{align*}
    Since $r_{t*}$ is continuous, by (\ref{Gamma:St}) and (\ref{vanish:rt}) we obtain the vanishing 
    \begin{align*}
        r_{t*}\Gamma_{\cS_t}(\mathcal{V}_M)=0.
    \end{align*}
Therefore (\ref{gamma:t}) holds by taking the global section over $\mathrm{H}_t(w)^{\mathrm{ss}}$.

\begin{sstep}
Proof of (\ref{wt:eps}).     
\end{sstep}
It remains to prove that (\ref{wt:eps}) holds. 
For a closed point $y\in \mathrm{H}_T(w)^{\varepsilon\text{-ss}}$ which 
lies over $t\in T$, 
let $\mathcal{O}_y$ be the corresponding object 
\begin{align*}
    \mathcal{O}_y \in 
    \Coh^{\heartsuit}(\mathrm{H}_T(w)^{\varepsilon\text{-ss}})\simeq
    \Coh^{\heartsuit}(\cH_T(w)^{\varepsilon\text{-ss}})_{0}. 
\end{align*}
Then we show that 
\begin{align}\label{wt:Oy}
    \wt(\iota_{*}\nu^{\mathrm{reg}*}\Phi_T^{\varepsilon}(\mathcal{O}_y))
    \subset (g-1)(\mu_1-\mu_2)\cdot [-1, 1] +\frac{w'}{2}(\mu_1+\mu_2)
\end{align}
and the left boundary weight is not attained when $\chi_1>\chi_2$. 

   Indeed let $E_y\in \Coh^{\heartsuit}(\cC_t)$ be the corresponding $h_{\varepsilon}$-stable sheaf. 
In the proof of Proposition~\ref{prop:VL}, the lower bound (\ref{low:bound}) 
of $\mathcal{V}_{E_y}$ 
is achieved only if there is an exact sequence as in (\ref{exact:Lprime})
\begin{align*}
    0\to \mathcal{L}_2 \to E_y \to \mathcal{L}_1 \to 0
\end{align*}
where $\mathcal{L}_i$ is a line bundle on 
$C_{ti}$ with $\deg \mathcal{L}_i=w/2<0$. 
By a choice of $h_{\varepsilon}$ satisfying (\ref{ineq:he}), it violates the $h_{\varepsilon}$-stability of $E_y$, because of the inequality (\ref{ineq:slope}).
Therefore (\ref{wt:Oy}) holds. 

Let \(\cH_T(\chi_0)' \subset \cH_T(\chi_0)\) be the open substack obtained by adjoining to \(\cH_T(\chi_0)^{\mathrm{ss}}\) the HN stratum whose associated graded is of the form \(M_1\oplus M_2\) where $M_i$ is supported on $C_{ti}$. 
Let $j_T'$ be the  
open immersion $j_T' \colon \cH_T(\chi_0)^{\mathrm{ss}} \hookrightarrow \cH_T(\chi_0)'$. 
The property (\ref{wt:Oy}) implies the following, see Remark~\ref{rmk:jshrink}
\begin{align}\label{PhiJT}\Phi_T^{\varepsilon}(\mathcal{O}_y) \in 
\mathrm{Im}((j_{T}')_{!}) \subset \LL(\cH_T(\chi_0)')_{w'}. 
\end{align}

Now the image of 
\begin{align}\label{idj_T}
    \LL(\mathrm{H}_T(w)^{\varepsilon\text{-ss}}\times \cH_T(\chi_0)^{\mathrm{ss}})_{(0, w')} 
   &\stackrel{(\id \times j_T')_{!} }{\hookrightarrow} \LL(\mathrm{H}_T(w)^{\varepsilon\text{-ss}}\times \cH_T(\chi_0)')_{(0, w')} \\
   \notag&\hookrightarrow  \Coh(\mathrm{H}_T(w)^{\varepsilon\text{-ss}}\times \cH_T(\chi_0)')
\end{align}
is a semiorthogonal summand 
\begin{align*}
    \Coh(\mathrm{H}_T(w)^{\varepsilon\text{-ss}}\times \cH_T(\chi_0)')=
    \langle \ldots,  \mathrm{Im}((\id\times j_T')_!), \ldots\rangle
\end{align*}
by Lemma~\ref{lem:sodsum}.

Let $\cP_T^{'\varepsilon}$ be the restriction of the kernel object
$\cP_T^{\varepsilon}$ of $\Phi_T^{\varepsilon}$ in (\ref{PTeps})
\begin{align*}
    \cP_T^{'\varepsilon}:=\cP_T^{\varepsilon}|_{\cH_T(w)^{\varepsilon\text{-ss}}\times_T \cH_T(\chi_0)'} \in \Coh(\cH_T(w)^{\varepsilon\text{-ss}}\times_T \cH_T(\chi_0)').
\end{align*}
It has weight zero in the first factor, so its push-forward to 
$\cH_T(w)^{\varepsilon\text{-ss}}\times \cH_T(\chi_0)'$ descends to an object 
\begin{align*}
    \cP_T^{''\varepsilon} \in \Coh(\mathrm{H}_T(w)^{\varepsilon \text{-ss}} \times \cH_T(\chi_0)').
\end{align*}
Then (\ref{PhiJT}) implies that $\cP_T^{''\varepsilon}$ lies 
in the image of (\ref{idj_T}) as in the argument of~\cite[Lemma~6.2]{PThiggs}; indeed 
let $Q$ be its semiorthogonal factor not contained in the image of $(\id \times j_T')_{!}$. 
Then the object $Q$ restricted to $\{y\} \times \cH_T(\chi_0)'$ is zero for any $y$, 
hence $Q=0$. 
In particular 
the image of $\Phi_T^{\varepsilon}(-)|_{\cH_T(\chi_0)'}$ lies in 
$\mathrm{Im}((j_T')_!)$. 
Therefore (\ref{wt:eps}) holds.

\begin{sstep}\label{step6}
    Conclusion of the proof of the proposition. 
\end{sstep}
By the arguments as above, 
for a map $\nu$ as in (\ref{nu:F12}) we have obtained the required weight bound. For a map $\nu$ with opposite ordering $\chi_2>\chi_1$ and $\mu_2>\mu_1$, 
let $h_{\varepsilon}'$ be a $\mathbb{Q}$-ample divisor on $\cC_b$ such that 
\begin{align}\notag
\deg h^{\prime}_{\varepsilon}|_{C_1}=1+\varepsilon, \
\deg h^{\prime}_{\varepsilon}|_{C_2}=1-\varepsilon.
\end{align}
We similarly define a $\mathbb{Q}$-ample divisor $h_{T}^{\varepsilon'}$
on $\cC_T$, and 
denote the corresponding semistable locus by \(\cH_T(\chi_0)^{\varepsilon'\text{-ss}}\).
The
required weight bound for the opposite ordering follows by the same argument, by replacing 
$\cH_T(\chi_0)^{\varepsilon\text{-ss}}$ with $\cH_T(\chi_0)^{\varepsilon'\text{-ss}}$. 
Therefore the proposition holds. 
\end{proof}

\subsection{Proof of Whittaker normalization for \texorpdfstring{$\GL_2$}{GL2}}
Recall the functor (\ref{induce:Phi}) induced by the Arinkin sheaf
\begin{align*}
    \Phi \colon \Coh(\cH(w)^{\mathrm{ss}})_{-\chi'}
    \to \Coh(\cH(\chi))_{w'}. 
\end{align*}
Also recall the Hitchin section $s$ and its $!$-push forward in (\ref{sshrink}). 
We show that Conjecture~\ref{conj:norm} is satisfied in this case: 
\begin{prop}\label{prop:norm}
For $G=\GL_2$, there is an isomorphism 
\begin{align*}
   \Phi(\mathcal{O}_{\cH(w)^{\mathrm{ss}}})\cong s_{!}\mathcal{O}_{\cB}. 
\end{align*} 
\end{prop}
\begin{proof}
Recall in Subsection~\ref{subsec:whit} that the Hitchin section $s$ factors through 
\begin{align*}
    s\colon \cB \stackrel{\overline{s}}{\to} \cH(\chi_0)^{\mathrm{ss}} \stackrel{j}{\hookrightarrow}
    \cH(\chi_0). 
\end{align*} 
By Proposition~\ref{prop:induceT}, 
it is enough to show that 
\begin{align*}
    \Phi(\mathcal{O}_{\cH(w)^{\mathrm{ss}}})|_{\cH(\chi_0)^{\mathrm{ss}}}\cong \overline{s}_{!}\mathcal{O}_{\cB}.
\end{align*}
By Lemma~\ref{lem:pulls}, we have 
\begin{align*}
    \overline{s}^*(\Phi(\mathcal{O}_{\cH(w)^{\mathrm{ss}}})|_{\cH(\chi_0)^{\mathrm{ss}}})\cong h_{*}\mathcal{O}_{\cH(w)^{\mathrm{ss}}}
\end{align*}
where $h \colon \cH(w)^{\mathrm{ss}}\to \cB$ is the Hitchin map. 
By the adjunction, we have the map 
\begin{align}\label{adj}
    \overline{s}_{!}\mathcal{O}_{\cB} \to \Phi(\mathcal{O}_{\cH(w)^{\mathrm{ss}}})|_{\cH(\chi_0)^{\mathrm{ss}}}.
\end{align}
In what follows, we show that (\ref{adj}) is an isomorphism. 

We use the notation in the proof of Proposition~\ref{prop:induceT}.
By applying (\ref{dist:A}) to $A=\mathcal{O}_{\cH(w)^{\mathrm{ss}}}$, 
we obtain the distinguished triangle 
\begin{align*}
    p_{T*}\Gamma_{q_T^{-1}(\cS)}(\cP_T) \to \Phi_T(\mathcal{O}_{\cH_T(w)^{\mathrm{ss}}}) \to 
    \Phi_T^{\varepsilon}
    (\mathcal{O}_{\cH_T(w)^{\varepsilon\text{-ss}}}).
\end{align*}
We take a map 
\begin{align*}
    \nu \colon \Spec k \to \cH_T(\chi_0)^{\varepsilon\text{-ss}}
\end{align*}
and let $M\in \Coh^{\heartsuit}(\cC_t)$ be the corresponding $h_{T, \varepsilon}$-semistable sheaf for $t\in T$. 
In Lemma~\ref{lem:vanish2} below, we show that 
\begin{align}\label{show:nu}
    \nu^*  p_{T*}\Gamma_{q_T^{-1}(\cS)}(\cP_T)
    =\Gamma(\cH_t(w)^{\mathrm{ss}}, \Gamma_{\cS_t}(\cP_M))=0. 
\end{align}
It follows that we have the isomorphism 
\begin{align}\label{isom:eps}
    \Phi_T(\mathcal{O}_{\cH_T(w)^{\mathrm{ss}}})|_{\cH_T(\chi_0)^{\varepsilon\text{-ss}}} \stackrel{\cong}{\to}
    \Phi_T^{\varepsilon}
    (\mathcal{O}_{\cH_T(w)^{\varepsilon\text{-ss}}})|_{\cH_T(\chi_0)^{\varepsilon\text{-ss}}}.
\end{align}

Since the
$h_{T, \varepsilon}$-semistability is equivalent to 
the $h_{T, \varepsilon}$-stability, 
the Fourier--Mukai functor 
\begin{align}\label{equiv:Phi}
    \Phi_T^{\varepsilon, \varepsilon} \colon 
    \Coh(\cH_T(w)^{\varepsilon\text{-ss}})_{0}
    \to  \Coh(\cH_T(\chi_0)^{\varepsilon\text{-ss}})_{w'}
\end{align}
given by the kernel sheaf 
\begin{align*}
    \cP_T^{\varepsilon, \varepsilon}:=\cP_T|_{\cH_T(w)^{\varepsilon\text{-ss}}\times_T \cH_T(\chi_0)^{\varepsilon\text{-ss}}}
    \in \Coh^{\heartsuit}(\cH_T(w)^{\varepsilon\text{-ss}}\times_T \cH_T(\chi_0)^{\varepsilon\text{-ss}})
\end{align*}
is an equivalence by Proposition~\ref{prop:MRV}.
The kernel object of its inverse 
functor $(\Phi_T^{\varepsilon, \varepsilon})^{-1}$ is given by 
\begin{align}\label{Pdual2}
     (\cP_T^{\varepsilon, \varepsilon})^{\vee}[4g-3] \in 
     \Coh^{\heartsuit}(\cH_T(w)^{\varepsilon\text{-ss}}\times_T \cH_T(\chi_0)^{\varepsilon\text{-ss}}).
\end{align}

The base-change of the Hitchin section $\overline{s}$ to $T$, denoted $\overline{s}_T$, further factors through
\begin{align*}
    \overline{s}_T \colon T \stackrel{\overline{s}_T^{\varepsilon}}{\to} 
    \cH_T(\chi_0)^{\varepsilon\text{-ss}} \subset\cH_T(\chi_0)^{\mathrm{ss}}
\end{align*}
where $\overline{s}_T^{\varepsilon}$ is a closed immersion 
after composing with the good moduli space map $\cH_T(\chi_0)^{\varepsilon\text{-ss}} \to \mathrm{H}_T(\chi_0)^{\varepsilon\text{-ss}}$.
Then Lemma~\ref{lem:pulls} together with the description of the kernel object (\ref{Pdual2}) of $(\Phi_T^{\varepsilon, \varepsilon})^{-1}$ imply that 
\begin{align*}
    (\Phi_T^{\varepsilon, \varepsilon})^{-1}(\overline{s}_{T!}^{\varepsilon}\mathcal{O}_T)
    \cong \mathcal{O}_{\cH_T(w)^{\varepsilon\text{-ss}}}.
\end{align*}
Here $\overline{s}_{T!}^{\varepsilon}=\overline{s}_{T*}^{\varepsilon}[3-4g]$
as in (\ref{formula:s!}). 
By the equivalence (\ref{equiv:Phi}), it follows that 
\begin{align*}
    \Phi_T^{\varepsilon, \varepsilon}(\mathcal{O}_{\cH_T(w)^{\varepsilon\text{-ss}}})
    \cong \overline{s}_{T!}^{\varepsilon}\mathcal{O}_T.
\end{align*}
Together with (\ref{isom:eps}), we obtain the isomorphism 
\begin{align*}
     \Phi_T(\mathcal{O}_{\cH_T(w)^{\mathrm{ss}}})|_{\cH_T(\chi_0)^{\varepsilon\text{-ss}}} \stackrel{\cong}{\to}\overline{s}_{T!}\mathcal{O}_T|_{\cH_T(\chi_0)^{\varepsilon\text{-ss}}}.
\end{align*}

By replacing the $h_{\varepsilon}$-stability with the $h_{\varepsilon}'$-stability in the proof of Proposition~\ref{prop:induceT} Step~\ref{step6}, 
we also obtain 
\begin{align*}
     \Phi_T(\mathcal{O}_{\cH_T(w)^{\mathrm{ss}}})|_{\cH_T(\chi_0)^{\varepsilon'\text{-ss}}} \stackrel{\cong}{\to}\overline{s}_{T!}\mathcal{O}_T|_{\cH_T(\chi_0)^{\varepsilon'\text{-ss}}}.
\end{align*}
Note that we have 
\begin{align*}
    \cH_T(\chi_0)^{\mathrm{ss}}=
    \cH_T(\chi_0)^{\varepsilon\text{-ss}} \cup
    \cH_T(\chi_0)^{\varepsilon'\text{-ss}} \cup \mathcal{Z}
\end{align*}
where $\mathcal{Z}$ is a closed substack of $\cH_T(\chi_0)^{\mathrm{ss}}$
corresponding to strictly polystable sheaves (which also appears as a center of $\cS$, see the proof of Proposition~\ref{prop:induceT}). 
Since the image of $\overline{s}_T$ is a closed substack of $\cH_T(\chi_0)^{\mathrm{ss}}$ disjoint from $\mathcal{Z}$, 
we have the distinguished triangle 
\begin{align}\label{PhiR}
R \to (\overline{s}_{T})_{!}\mathcal{O}_T  \to 
     \Phi_T(\mathcal{O}_{\cH_T(w)^{\mathrm{ss}}})
\end{align}
where $R$ is supported on $\mathcal{Z}$. 

Now we show that 
the restriction functor
\begin{align*}
    \LL(\cH_T(\chi_0)^{\mathrm{ss}})_{w'} \to \Coh(\cH_T(\chi_0)^{\mathrm{ss}}\setminus \mathcal{Z})
\end{align*}
is conservative. We have the following diagram 
\begin{align*}
    \mathcal{Z} \stackrel{q}{\leftarrow} \mathcal{S} \stackrel{p}{\hookrightarrow} \cH_T(\chi_0)^{\mathrm{ss}}
\end{align*}
where $\mathcal{S}$ is the $\Theta$-stratum with support 
the complement of $\cH_T(\chi_0)^{\varepsilon\text{-ss}}$. 
Then as in the proof of~\cite[Theorem~7.18]{PTlim}, we have the semiorthogonal decomposition
\begin{align*}
    \LL(\cH_T(\chi_0)^{\mathrm{ss}})_{w'}=\langle 
    p_{*}q^* \LL(\mathcal{Z})_{\delta}, \mathcal{W} \rangle, \ 
    \mathcal{W} \stackrel{\sim}{\to} \LL(\cH_T(\chi_0)^{\varepsilon\text{-ss}})_{
    w'}
\end{align*}
for some $\delta$. Then if $A \in  \LL(\cH_T(\chi_0)^{\mathrm{ss}})_{w'}$
satisfies $A|_{\cH_T(\chi_0)^{\mathrm{ss}}\setminus \mathcal{Z}}=0$, 
then $A|_{\cH_T(\chi_0)^{\varepsilon\text{-ss}}}=0$ as $\cH_T(\chi_0)^{\varepsilon\text{-ss}}\subset \cH_T(\chi_0)^{\mathrm{ss}}\setminus \mathcal{Z}$. Then from the above semiorthogonal decomposition, $A=p_{*}q^{*}A'$ for some $A' \in \LL(\mathcal{Z})_{\delta}$. 
Then $q^*A'$ is zero on $\mathcal{S} \setminus \mathcal{Z}$, and $\mathcal{S} \to \mathcal{Z}$ is a vector bundle, therefore $A'=0$. 


It follows that (\ref{adj}) is an isomorphism for any \'{e}tale neighborhood of 
$b\in \cB$, therefore it is an isomorphism. 
\end{proof}

\begin{lemma}\label{lem:vanish2}
    We have the vanishing (\ref{show:nu}). 
\end{lemma}
\begin{proof}
    The proof follows as in Step~\ref{sstep4} of Proposition~\ref{prop:induceT}, 
    and we use the notation in loc. cit. 
    Indeed, it is enough to prove 
    \begin{align}\notag
    \Gamma(\cH_t(w)^{\mathrm{ss}}, \Gamma_{\cS_t}(\mathcal{V}_M))=0.
    \end{align}
 Let 
 \begin{align*}
     0\to \mathcal{L} \to M \to Q \to 0
 \end{align*}
 be an exact sequence as in Lemma~\ref{lem:LEQ}, where 
 $\mathcal{L}$ is a line bundle on $\cC_t$ and $Q=\oplus_{i=1}^{\ell} Q_i$ with $\chi(Q_i)=n_i$. Let $l_i=\deg(\mathcal{L}|_{C_{ti}})$. 
 Note that $l_1+l_2+\sum_{i=1}^{\ell} n_i=0$ since $\deg \pi_{*}M=\chi_0=2-2g$. 
 Then as in the proof of Step~\ref{sstep4} of Proposition~\ref{prop:induceT}, the $\lambda$-weights of $\cHom_{\cH_t}(\mathcal{O}_{\cS_n}, \mathcal{V}_M)$
 are at most 
 \begin{align}
    \notag &\sum_{i=1}^{\ell}(a_i \lambda_1+(n_i-a_i)\lambda_2) +l_1 \lambda_1+l_2 \lambda_2+(2g-2)(\lambda_1-\lambda_2) \\
   \label{eq:lambda12}  &=\left( \sum_{i=1}^{\ell} a_i +l_1+2g-2  \right)(\lambda_1-\lambda_2).
 \end{align}
 We have the exact sequence
 \begin{align*}
     0\to \mathcal{L}' \to M \to \mathcal{L}|_{C_{t1}}\to 0
 \end{align*}
 where $\mathcal{L}'$ is a line bundle on $C_{t2}$. By the $h_{T, \varepsilon}$-stability of $M$, we have 
 \begin{align*}
 \frac{l_1}{h_{T, \varepsilon}\cdot C_{t1}}=\frac{l_1}{1-\varepsilon}
 >\frac{1}{2}\deg \pi_{*}M=1-g.
 \end{align*}
 Together with $\lambda_1<\lambda_2$, it follows that 
 (\ref{eq:lambda12}) is negative, therefore $\cHom_{\cH_t}(\mathcal{O}_{\cS_n}, \mathcal{V}_M)$
 has negative $\lambda$-weights. Then the vanishing (\ref{show:nu}) holds 
 as in the proof of Step~\ref{sstep4} of Proposition~\ref{prop:induceT}. 
\end{proof}

\section{Proofs of some lemmas}
\subsection{Limit categories for perfect complexes}

\begin{lemma}\label{lem:perfect}
In the setting of Definition~\ref{def:Lcat}, suppose that
$\cE \in \Coh(\mathfrak{M})$ is perfect. Then for a map
$\nu \colon \bgm \to \mathfrak{M}$ with image $x\in \mathfrak{M}(k)$, 
$\mathrm{wt}(\nu^*\cE)$ is contained in a bounded interval $I \subset \mathbb{R}$ 
if and only if 
$\mathrm{wt}(\iota_{*}\nu^{\mathrm{reg}*}\cE)$ is contained in 
$I+[c_1(\mathfrak{g}_x^{<0}), c_1(\mathfrak{g}_x^{>0})]$. 

In particular, $\nu^* \cE$ 
satisfies condition~(\ref{wt:nu})
if and only if $\iota_{*}\nu^{\mathrm{reg}*}\cE$ satisfies condition~(\ref{wt:cond2}).
\end{lemma}

\begin{proof}
Consider the following commutative diagram:
\begin{align*}
    \xymatrix{
\bgm \ar[r]^-{i} \ar[rd]_-{\nu} &
\mathfrak{g}_x^{\vee}[-1]/\mathbb{G}_m
\ar[r]^-{\iota} \ar[d]^{\nu^{\mathrm{reg}}} &
\bgm \\
& \mathfrak{M}. &
    }
\end{align*}
Here $i$ is the map from the classical truncation. We have
\begin{align*}
    \iota_{*}\nu^{\mathrm{reg}*}\mathcal{E}
    \cong
    i^* \iota^* \iota_{*}\nu^{\mathrm{reg}*}\mathcal{E}.
\end{align*}
The stack $\mathfrak{g}_x^{\vee}[-1]/\mathbb{G}_m$ is the derived zero locus
of the vector bundle
$\mathfrak{g}_x^{\vee}/\mathbb{G}_m \to \bgm$. Hence, by the Koszul resolution,
the object
$
    i^* \iota^* \iota_{*}\nu^{\mathrm{reg}*}\mathcal{E}
$
is obtained by iterated extensions of objects
\begin{align*}
        i^*\bigl(\nu^{\mathrm{reg}*}\mathcal{E}
        \otimes \wedge^j \mathfrak{g}_x[j]\bigr)
        \cong
        \nu^*\mathcal{E} \otimes \wedge^j \mathfrak{g}_x[j]
\end{align*}
for $0\leq j\leq \dim \mathfrak{g}_x$. Since $\mathcal{E}$ is perfect, the set
of weights of $\nu^*\mathcal{E}$ is finite. It follows from the above
isomorphisms that the weights of $\nu^*\mathcal{E}$ are contained in
$[a,b]$ if and only if the weights of
$\iota_{*}\nu^{\mathrm{reg}*}\mathcal{E}$ are contained in
\begin{align*}
        [a,b]+
        [c_1(\mathfrak{g}_x^{<0}), c_1(\mathfrak{g}_x^{>0})].
\end{align*}

The second claim follows from the computation
\begin{align*}
        \frac{1}{2}c_1(T_x^{>0})
        +\frac{1}{2}c_1(\mathfrak{g}_x)
        -c_1(\mathfrak{g}_x^{>0})
        &=
        \frac{1}{2}c_1(T_x^{>0})
        -\frac{1}{2}c_1((\mathfrak{g}_x^{\vee})^{>0})
        -\frac{1}{2}c_1(\mathfrak{g}_x^{>0})  \\
        &=
        \frac{1}{2}c_1(\nu^*\mathbb{L}^{>0}_{\mathfrak{M}}).
\end{align*}
The identity for the lower bound is similarly calculated. 
Here we used the self-duality of $\mathbb{L}_{\mathfrak{M}}$, together with the
description of $\nu^*\mathbb{L}_{\mathfrak{M}}$ in terms of
$T_x$ and $\mathfrak{g}_x$.
\end{proof}

\begin{lemma}\label{lem:perfect2}
In the setting of Definition~\ref{def:Lcat}, suppose that
$\cE \in \Coh(\mathfrak{M})$ admits an expression as a colimit in
$\QCoh(\mathfrak{M})$
\begin{align*}
    \cE = \operatorname*{colim}_{i} A_i
\end{align*}
such that each $A_i$ is perfect. For a map $\nu \colon \bgm \to \mathfrak{M}$, 
suppose that there is a bounded interval $I \subset \mathbb{R}$ such that 
each $\wt(\nu^* A_i)$ is contained in $I$. Then 
$\wt(\iota_{*}\nu^{\mathrm{reg}*}\cE)$ is contained in $I+[c_1(\mathfrak{g}_x^{<0}), c_1(\mathfrak{g}_x^{>0})]$.

In particular if each $A_i$ satisfies the condition (\ref{wt:cond2}), then $\cE$ satisfies the condition (\ref{wt:nu}). 
\end{lemma}

\begin{proof}
Since $\nu^{\mathrm{reg}*}$ and $\iota_*$ are continuous on $\QCoh$, we have
an isomorphism in $\QCoh(\bgm)$
\begin{align*}
        \iota_{*}\nu^{\mathrm{reg}*}(\cE)
        \cong
        \operatorname*{colim}_{i}
        \iota_{*}\nu^{\mathrm{reg}*}A_i.
\end{align*}
By Lemma~\ref{lem:perfect}, if each $\wt(\nu^* A_i)$
is contained in $I$, then $\wt(\iota_{*}\nu^{\mathrm{reg}*}A_i)$ is contained in
$I+[c_1(\mathfrak{g}_x^{<0}), c_1(\mathfrak{g}_x^{>0})]$.
Therefore, after taking the colimit,
the weights of $\iota_{\ast}\nu^{\mathrm{reg}\ast}(\cE)$ are also contained
in $I+[c_1(\mathfrak{g}_x^{<0}), c_1(\mathfrak{g}_x^{>0})]$. This proves the lemma.
\end{proof}
\subsection{Proof of Lemma~\ref{lem:qsmooth}}\label{subsec:qsmooth}
\begin{proof}
We 
show that $\ev_1$, $(\ev_3, \ev_2)$ are quasi-smooth and proper 
of relative virtual dimensions $1, -1$ respectively; the other statements 
are also proved similarly. 
In general for a diagram of smooth stacks 
\begin{align*}
    \mathcal{X} \stackrel{f}{\leftarrow} \mathcal{Z} \stackrel{g}{\to} \mathcal{Y}
\end{align*}
we have the following diagram 
\begin{align*}
    \xymatrix{
    & \Omega_{(f, g)}[-1] \ar[r] \ar[d] \diasquare & g^* \Omega_{\mathcal{Y}} \ar[r]  \ar[d] & \Omega_{\mathcal{Y}} \\
    \Omega_{\mathcal{X}} & \ar[l] f^* \Omega_{\mathcal{X}} \ar[r] & \Omega_{\mathcal{Z}}.  &
    }
\end{align*}
Suppose that $f$ is smooth and $g$ is smooth and proper. Then all the arrows in the 
middle squares are quasi-smooth closed immersions. Therefore, in the diagram 
\begin{align*}
     \Omega_{\mathcal{X}} \stackrel{f'}{\leftarrow} \Omega_{(f, g)}[-1] \stackrel{g'}{\to} \Omega_{\mathcal{Y}}
\end{align*}
the map $f'$ is quasi-smooth and $g'$ is proper. 

The diagram (\ref{dia:greg})
is obtained from the above construction from the left Hecke diagram of moduli stacks of 
bundles, see~\cite[Section~9.4]{PTlim}. 
Namely let $\mathrm{Hecke}_G^{\mathrm{B}}$ be the moduli stack which classifies exact 
sequences 
\begin{align*}0\to F_1 \to F_2 \to \mathcal{O}_x \to 0
\end{align*}
where $F_i$ are vector bundles on $C$ of rank $r$
and $x\in C$. We have the evaluation morphisms
\begin{align}\label{dia:Bun1}
    \cC(1)\times \Bun_G \stackrel{(\ev_3^{\rB}, \ev_2^{\rB})}{\leftarrow} \mathrm{Hecke}_G^{\mathrm{B}} \stackrel{\ev_1^{\rB}}{\to} \Bun_G.
\end{align}
Here $\cC(1)=C\times \bgm$ is the moduli stack of skyscraper sheaves $\mathcal{O}_x$ for $x\in C$. 
Then as in the proof of~\cite[Lemma~9.4]{PTlim}, the resulting diagram 
\begin{align}\label{dia:Bun2}
  \Omega_{\mathcal{C}(1)}\times   \Omega_{\Bun_G} \stackrel{(\ev_3^{\rB}, \ev_2^{\rB})'}{\leftarrow}
  \Omega_{(\ev_3^{\rB}, \ev_2^{\rB}, \ev_1^{\rB})}[-1] \stackrel{(\ev_1^{\rB})'}{\to}
  \Omega_{\Bun_G}
\end{align}
is identified with the diagram (\ref{dia:greg}). 
Therefore $\ev_1^{\sharp}$ is proper and $(\ev_3^{\sharp}, \ev_2^{\sharp})$ is quasi-smooth. By the base change, the map $(\ev_3, \ev_2)$ is quasi-smooth. Also as $S \hookrightarrow S\times k[-1]$ is a closed 
immersion, the map $\ev_1$ is proper. 

In a similar way, the diagram 
\begin{align}\notag
    \xymatrix{
\mathrm{Hecke}_G \ar[d]_{(\ev_1^{\sharp}, \ev_3^{\sharp})} \ar[r]^{\ev_2^{\sharp}} & \Hig_G \\
\Hig_G\times \mathfrak{M}(1) &
    }
\end{align}
is obtained by the above construction from the right Hecke diagram, see~\cite[Subsection~9.3]{PTlim}. Therefore $\ev_2^{\sharp}$ is proper and $(\ev_1^{\sharp}, \ev_3^{\sharp})$ is quasi-smooth. 
By the same argument as above, the map $(\ev_1, \ev_3)$ is quasi-smooth and $\ev_2$ is proper. 

Since $S\times \Hig_G\to \Hig_G$ is smooth, 
we conclude that $\ev_1$ is quasi-smooth. 
It remains to show that $(\ev_3, \ev_2)$ is proper. 
By the same argument as above, we know that 
$\ev_2$ is proper and quasi-smooth. 
The map $(\ev_3, \ev_2)$ factors through 
\begin{align*}
\mathrm{Hecke}^{\prime}_G\to
    \mathrm{C}\times_{\mathrm{B}}\Hig_G\hookrightarrow 
    S\times \Hig_G.
\end{align*}
Since $\mathrm{C}\to \mathrm{B}$ is proper, the 
morphism 
$\mathrm{C}\times_{\mathrm{B}}\Hig_G\to \Hig_G$ is also proper, therefore 
we conclude that $(\ev_3, \ev_2)$ is proper. 

The map $\ev_1^{\sharp}$ in (\ref{dia:greg}) has relative virtual dimension zero,  
since the $\ev_1^{\sharp}$-relative cotangent complex at (\ref{seq:Ei}) is $\Hom^*_S(E_1, \mathcal{O}_x)$ whose Euler characteristic is zero. 
By the base change (\ref{Hecke2}), we have 
\begin{align*}
\operatorname{vdim}\mathrm{Hecke}_G^{\prime}=\operatorname{vdim}\mathrm{Hecke}_G+1,
\end{align*}
therefore 
we conclude that the map $\ev_1$ has relative virtual dimension one. 
In a similar way, the map $\ev_2$ has relative virtual dimension one, 
hence $(\ev_3, \ev_2)$, $(\ev_3, \ev_1)$ have virtual dimension $-1$. 
\end{proof}

\subsection{Proof of Lemma~\ref{lem:HeckeL}}\label{subsec:HeckeL}
\begin{proof}
In~\cite[Section~9.4]{PTlim}, it is shown that the diagram (\ref{dia:Bun1}) induces the functor 
\begin{align*}
    (\ev_1^{\rB})_{*}^{\Omega} (\ev_3^{\rB}, \ev_2^{\rB})^{\Omega !}\colon 
    \LL(\mathfrak{M}(1))_{0} \otimes & \LL(\Hig_G(\chi))_{\delta \otimes \omega_{\Bun_G}^{1/2}}\\
    &\to \LL(\Hig_G(\chi-1))_{\delta \otimes \omega_{\Bun_G}^{1/2}}.
\end{align*}
Here we refer to loc. cit. for the notation $(-)_{*}^{\Omega}$ and $(-)^{\Omega !}$. 
Unraveling their definitions and identifications of the diagram (\ref{dia:Bun1}) with (\ref{dia:Bun2}), 
the above functor is identified with 
\begin{align}\label{evOmega}
    (\ev_1^{\rB})_{*}^{\Omega} (\ev_3^{\rB}, \ev_2^{\rB})^{\Omega !}=  (\ev_1^{\sharp})_{*}((\ev_3^{\sharp}, \ev_2^{\sharp})^* \otimes \omega_{(\ev_3^{\rB}, \ev_2^{\rB})}[d])
\end{align}
where $d$ is the relative virtual dimension of $(\ev_3^{\rB}, \ev_2^{\rB})$. Note that $\omega_{\Bun_G}^{1/2}$
exists as a line bundle on $\Bun_G$ once we fix $\omega_C^{1/2} \in \mathrm{Pic}(C)$, 
see~\cite[Section~4]{BD0}. Also note that
\begin{align*}
    \LL(\mathfrak{M}(1))_{0}=\Coh(S\times k[-1] \times \bgm)_0 \stackrel{\sim}{\leftarrow} \Coh(S\times k[-1]).
\end{align*}

We consider the functor 
\begin{align}\label{funct:cohL}
    \Coh(S\times k[-1]) \otimes \LL(\Hig_G(\chi))_w \to \LL(\Hig_G(\chi-1))_w
\end{align}
given by 
\begin{align*}
(-)\mapsto 
    (\ev_1^{\rB})^{\Omega}_{*} (\ev_3^{\rB}, \ev_2^{\rB})^{\Omega !}((-)\boxtimes \omega_{\Bun_G}^{1/2})\otimes\omega_{\Bun_G}^{-1/2}.
\end{align*}
Then from (\ref{evOmega}), the above functor is of the form 
\begin{align*}
(-)\mapsto 
    (\ev_1^{\sharp})_*((\ev_3^{\sharp}, \ev_2^{\sharp})^*(-)\otimes \mathcal{L}[d])
\end{align*}
where $\mathcal{L}$ is a line bundle on $\mathrm{Hecke}_G$ of the form 
\begin{align*}
    \mathcal{L}=\omega_{(\ev_3^{\rB}, \ev_2^{\rB})}\otimes \omega_{\Bun_G(\chi)}^{1/2}
    \otimes \omega_{\Bun_G(\chi-1)}^{-1/2}. 
\end{align*}
Here we have omitted the relevant pull-backs to $\mathrm{Hecke}_G$. 
The fibers of each line bundle are given by 
\begin{align*}
    &\omega_{(\ev_3^{\rB}, \ev_2^{\rB})}|_{(F_1, \theta_1) \subset (F_2, \theta_2)}=\det \chi_C(F_2, \mathcal{O}_x)^{\vee}, \\
    &\omega_{\Bun_G(\chi)}^{1/2}|_{(F_1, \theta_1) \subset (F_2, \theta_2)}=\det (\chi_C(F_2, F_2)[1])^{-1/2}, \\
    &\omega_{\Bun_G(\chi-1)}^{1/2}|_{(F_1, \theta_1) \subset (F_2, \theta_2)}=\det (\chi_C(F_1, F_1)[1])^{-1/2}. 
\end{align*}
Since $[F_1]=[F_2]-[\mathcal{O}_x]$ in K-theory, by substituting it to $\chi_C(F_1, F_1)$, 
a simple calculation shows that 
\begin{align*}
    &\mathcal{L}|_{(F_1, \theta_1) \subset (F_2, \theta_2)} \\
    &=\det \chi_C(F_2, \mathcal{O}_x)^{-1/2}\otimes \det \chi_C(\mathcal{O}_x, F_2)^{1/2}\otimes \det \chi_C(\mathcal{O}_x, \mathcal{O}_x)^{1/2} \\
    &\cong \det \chi_C(F_2, \mathcal{O}_x)^{-1/2}\otimes \det \chi_C(F_2, \omega_C \otimes \mathcal{O}_x)^{1/2}\otimes \det \chi_C(\mathcal{O}_x, \mathcal{O}_x)^{1/2}.
\end{align*}
From the above description, the line bundle $\mathcal{L}$ is pulled back from a $\mathbb{Q}$-line bundle $\mathcal{L}'$ 
on $\mathfrak{M}(1)\times \Hig_G(\chi)$, which is trivial on $\Omega_{U\times \bgm} \times \Hig_G(\chi)$
for any affine open subset $U\subset C$ such that $\omega_C|_{U} \cong \mathcal{O}_U$. In particular it 
satisfies that for any $\nu \colon \bgm \to \mathfrak{M}(1)\times \Hig_G(\chi)$ we have $c_1(\nu^* \mathcal{L}')=0$. 

It follows that the twist by $\mathcal{L}'$ does not change the source of (\ref{funct:cohL}), 
therefore we obtain the functor 
\begin{align*}
      (\ev_1^{\sharp})_*(\ev_3^{\sharp}, \ev_2^{\sharp})^*(-) \colon \Coh(S\times k[-1]) \otimes \LL(\Hig_G(\chi))_w \to \LL(\Hig_G(\chi-1))_{w}.
\end{align*}
Finally, the functor $\mathrm{H}$ is given as the composition of the above functor with
$i_{*} \colon \Coh(S) \to \Coh(S\times k[-1])$ for the closed immersion $i \colon S \hookrightarrow S\times k[-1]$, therefore we obtain the lemma. 
\end{proof}

\subsection{A lemma on the existence of an integral model}
The following lemma should be well-known, but we include it here because of a lack of a reference. 
\begin{lemma}\label{lem:intmodel}
    Let $\cC \subset S$ be a reduced curve in a smooth surface $S$ and take $x_1, \ldots, x_d \in \cC$.
    Then there is an integral projective curve $D\subset \mathbb{P}^2$ with $y_1, \ldots, y_d \in D$ such that 
    \begin{align}\label{isom:formC}\widehat{\mathcal{O}}_{\cC, x_i} \cong \widehat{\mathcal{O}}_{D, y_i}.
    \end{align}
    
    Moreover for a rank one torsion-free sheaf $E$ on $\cC$ there is a rank 
    one torsion-free sheaf $F$ on $D$ which are formally locally isomorphic under the isomorphisms (\ref{isom:formC}). 
\end{lemma}
\begin{proof}
We choose points
$y_1,\ldots,y_d\in \mathbb{P}^2$,
and isomorphisms 
\begin{align}\label{form:yi}
    \widehat{\mathcal{O}}_{S,x_i}
    \cong
    \widehat{\mathcal{O}}_{\mathbb{P}^2,y_i}.
\end{align}
The local defining equations of $\cC\subset S$ at $x_i$ define 
elements $f_i \in \widehat{\mathcal{O}}_{\mathbb{P}^2, y_i}$ under the 
above isomorphism. 
We take
$D\in |\mathcal{O}_{\mathbb{P}^2}(m)|$
such that its defining equation $f$ satisfies 
$f\equiv f_i$ modulo $m_i^N$ for $N\gg 0$, where $m_i \subset \widehat{\mathcal{O}}_{\mathbb{P}^2, y_i}$ are the maximal ideals;
this is possible for $m\gg 0$ by Bertini theorem. Then 
by the finite determinacy theorem for contact equivalence~\cite[Theorem I.2.23 and Corollary I.2.24]{GLS}, 
we have $(f)=(f_i)$ in 
$\widehat{\mathcal{O}}_{\mathbb{P}^2, y_i}$ after a change of formal 
coordinates. 
Therefore the first statement follows. 

For the second statement, by Remark~\ref{rmk:dense}, we may write $E\cong \cL\otimes I_Z$
for a zero-dimensional closed subscheme $Z\subset \cC$ and a line bundle $\cL$ on $\cC$. Since $\cL$ is trivialized at each $x_i$,  we can 
take $F=I_{Z'}$ for a zero-dimensional closed subscheme $Z'\subset D$
which is formally isomorphic to $Z$ at each $y_i$ under the above
isomorphisms (\ref{form:yi}). Therefore the lemma is proved. 
\end{proof}
\subsection{Proof of Proposition~\ref{prop:comad}}\label{subsec:modify}
\begin{proof}
We explain how to modify the proof of the commutativity of the diagram (\ref{com:WH4}). The only point for the modification is to prove a variant of Proposition~\ref{prop:PE}, namely we need to show that 
\begin{align}\label{HPE}
    \widetilde{\mathrm{H}}^R(\cP_E)[-1] \in \Coh^{\heartsuit}(\cH\times S)
\end{align}
and it is a maximal Cohen--Macaulay sheaf on $\cH_b\times \cC_b$. 
By (\ref{formula:ev1}) and Lemma~\ref{lem:omegaev}, we have 
\begin{align}\label{evL}
      \widetilde{\mathrm{H}}^R(\cP_E)[-1]=(\ev_2, \ev_3)_{*}\ev_1^*(\cP_E) \otimes \mathcal{L}.
\end{align}
Let $\psi$ be the isomorphism given by the shifted derived dual on $S$
\begin{align*}
    \psi \colon \cH \stackrel{\cong}{\to} \cH, \ E \mapsto E^{\vee}
\end{align*}
where $E^{\vee}=\mathbb{D}_S(E)[1]$. 
Then we have the commutative diagram 
\begin{align}\label{dia:heckedual}
    \xymatrix{
\cH \ar[d]^-{\psi}_-{\cong} & \cHecke^{\mathrm{greg}} \ar[r]^-{(\ev_2, \ev_3)} \ar[d]^-{\psi}_-{\cong} \ar[l]_-{\ev_1} & \cH \times S \ar[d]^-{\psi \times \id}_-{\cong} \\
\cH & \cHecke^{\mathrm{greg}} \ar[l]_-{\ev_2} \ar[r]^-{(\ev_1, \ev_3)} & \cH \times S.
    }
\end{align}
Here the middle vertical arrow is given by 
\begin{align*}
    (0\to E_1 \to E_2 \to \mathcal{O}_x \to 0) \mapsto (0\to E_2^{\vee} \to E_1^{\vee} \to \mathcal{O}_x \to 0)
\end{align*}
Then (\ref{HPE}) holds 
from (\ref{evL}), Proposition~\ref{prop:PE} the diagram (\ref{dia:heckedual})
and 
noting that 
$\psi_{*}\cP_E \cong \cP_{E^{\vee}\otimes \omega_{\cC_b}^{-1}}$. 

Using (\ref{HPE}), we can show that the diagram (\ref{com:WH3ad}) commutes; the arguments are 
the same as those in the proof of the commutative diagram (\ref{com:WH4})
in Subsection~\ref{subsec:proofWH}, namely we compare the kernel 
objects of both compositions, show that both are Cohen--Macaulay sheaves 
which are isomorphic in codimension one, hence an isomorphism.
Since the argument is the same, we omit details.   
\end{proof}

\subsection{Proof of Lemma~\ref{lem:delta}}\label{subsec:delta}
\begin{proof}
    Let $K$ be the total ring of fractions of $R$
    \begin{align*}
        K=\prod_{i=1}^{\ell}k((t_i)).
    \end{align*}
    Then for a given embedding $\phi \colon R \hookrightarrow M$, 
it induces the isomorphism $\phi_K \colon K\stackrel{\cong}{\to} M\otimes_R K$.
    By setting $I=\phi_K^{-1}(M)$, we have 
    $R\subset I \subset K$.
Let $I^{\mathrm{sat}}:=\widetilde{R} I \subset K$. It is a rank one 
torsion-free module of $\widetilde{R}$ which contains $\widetilde{R}$, 
therefore it is written as 
\begin{align}\label{I:sat}
    I^{\mathrm{sat}}=\prod_{i=1}^{\ell} t^{-a_i} k[[t_i]], \ a_i \geq 0.
\end{align}
It is enough to show that if $\phi \colon R\hookrightarrow M$ attains $\delta$, 
then $a_i=0$. 

By (\ref{I:sat}), there is $u\in I$ such that $\mathrm{val}_i(u)=-a_i$, 
where $\mathrm{val}_i(u)$ is the valuation of the $i$-th component of $u$. 
Then by replacing $\phi$ with $\phi' \colon R\hookrightarrow M$ with $\phi'(1)=u\phi(1)$, 
we obtain $R \subset I'=u^{-1}I \subset K$. 
We have the exact sequences 
\begin{align*}
   & 0\to I/R \to I^{\mathrm{sat}}/R \to I^{\mathrm{sat}}/I \to 0, \\
   & 0\to I'/R \to \widetilde{R}/R \to \widetilde{R}/I' \to 0. 
\end{align*}
We also have the isomorphism $u \colon \widetilde{R}/I' \stackrel{\cong}{\to}I^{\mathrm{sat}}/I$.
Therefore 
\begin{align*}
    \dim I/R &=\dim I^{\mathrm{sat}}/R-\dim \widetilde{R}/R+\dim I'/R \\
    &=\dim I^{\mathrm{sat}}/\widetilde{R}+\dim I'/R \\
    &=\sum_{i=1}^{\ell}a_i+\dim I'/R.
\end{align*}
By the definition of $\delta$, we have $\dim I/R \leq \dim I'/R$, 
therefore $a_i=0$ for all $i$. 
\end{proof}

\subsection{Categories of functors}
For cocomplete dg-categories $\cC_1, \cC_2$, let 
$\mathrm{Funct}_{\mathrm{cont}}(\cC_1, \cC_2)$ be the dg-category of 
continuous functors $\cC_1 \to \cC_2$. In the proof of Theorem~\ref{thm:eq}, 
we have used the following lemma: 

\begin{lemma}\label{lem:Funcon}
In the setting of Theorem~\ref{thm:eq}, we have 
\begin{align*}
    \mathrm{Funct}_{\mathrm{cont}}(\IndCoh_{\mathcal{N}}(\cH(w)^{\mathrm{ss}}), \QCoh(S^{\times d})) \simeq 
    \IndCoh_{\mathcal{N}}(\cH(w)^{\mathrm{ss}}\times S^{\times d}).
\end{align*}
Here the right-hand side is the subcategory 
with singular supports $\mathcal{N}\times S^{\times d}$. 
\end{lemma}
\begin{proof}
Since $\cH(w)^{\mathrm{ss}}$ is quasi-smooth and QCA of global complete intersection, the dg-category $\IndCoh_{\mathcal{N}}(\cH(w)^{\mathrm{ss}})$ is compactly generated by $\Coh_{\mathcal{N}}(\cH(w)^{\mathrm{ss}})$, see~\cite[Corollary~9.2.7]{AG}. 
Then from~\cite[Section 1.1.1, 1.1.2]{GaiFun}, and noting 
$\IndCoh(S^{\times d})=\QCoh(S^{\times d})$ as $S$ is smooth, 
we have 
\begin{align*}
      \mathrm{Funct}_{\mathrm{cont}}(\IndCoh_{\mathcal{N}}(\cH(w)^{\mathrm{ss}}), \QCoh(S^{\times d})) 
      &\simeq 
      \IndCoh_{\mathcal{N}}(\cH(w)^{\mathrm{ss}})^{\vee} \otimes 
      \IndCoh(S^{\times d}) \\
      &\simeq \Ind(\Coh_{\mathcal{N}}(\cH(w)^{\mathrm{ss}})^{\mathrm{op}})\otimes 
      \IndCoh(S^{\times d}) \\
      &\simeq \Ind(\Coh_{\mathcal{N}}(\cH(w)^{\mathrm{ss}}))\otimes 
      \IndCoh(S^{\times d}) \\
      &\simeq \IndCoh_{\mathcal{N}}(\cH(w)^{\mathrm{ss}})\otimes \IndCoh(S^{\times d}) \\
      &\simeq \IndCoh_{\mathcal{N}}(\cH(w)^{\mathrm{ss}}\times S^{\times d}).
\end{align*}
Here the 3rd equivalence is given by the derived dual 
\begin{align*}
    \mathbb{D}_{\cH(w)^{\mathrm{ss}}} \colon 
    \Coh_{\mathcal{N}}(\cH(w)^{\mathrm{ss}})^{\mathrm{op}} \stackrel{\sim}{\to} \Coh_{\mathcal{N}}(\cH(w)^{\mathrm{ss}}).
\end{align*}
\end{proof}
\subsection{Another proof of Theorem~\ref{thm:eq}}\label{subsec:Hcpt}
We give another proof that $\gamma$ is an isomorphism in the proof of Theorem~\ref{thm:eq}, which works except for the case $g=2$ and $G=\GL_2$.
The latter condition is equivalent to that 
\begin{align}\label{Bell:codim}
    \operatorname{codim}(\cB \setminus \cB^{\mathrm{ell}})=2(r-1)(g-1)-1 \geq 2.
\end{align}
\begin{proof}
Below we show that $\gamma$ is an isomorphism except $g=2$ and $G=\GL_2$.
 Since $\Phi$ is an equivalence over $\cB^{\mathrm{ell}}:=\mathrm{B}^{\mathrm{ell}} \cap \cB$, the functor $\Upsilon$ is zero over $\cB^{\mathrm{ell}}$. It follows that $\gamma$ is an isomorphism over $\cB^{\mathrm{ell}}$. However the kernel object of $\mathrm{W}_{d, c}^R$ is $(\cE^{\vee}[1])^{\boxtimes d}$, which is a Cohen--Macaulay extension of a line bundle on 
 \begin{align*}j\colon (\cC \setminus \cC^{\mathrm{sg}})^{\times_{\cB} d}\times_{\cB}\cH(w)^{\mathrm{ss}}\subset \cC^{\times_{\cB}d}\times_{\cB}\cH(w)^{\mathrm{ss}} \end{align*}
 up to shift, say $j_{*}^{\heartsuit}L$ for a line bundle $L$ and $j^{\heartsuit}_{*}=\cH^0(j_{*})$. Therefore $\gamma$ is given by an element \begin{align*}
 \gamma \in \Hom^0(j^{\heartsuit}_{*}L, j^{\heartsuit}_{*}L)&\cong \Hom^0(L, L) \\&\cong \cH^0\Gamma(\mathcal{O}_{(\cC \setminus \cC^{\mathrm{sg}})^{\times_{\cB} d}\times_{\cB}\cH(w)^{\mathrm{ss}}}) \\&\cong \cH^0\Gamma(\mathcal{O}_{\cC^{\times_{\cB} d}\times_{\cB}\cH(w)^{\mathrm{ss}}})\cong \cH^0\Gamma(\mathcal{O}_{\cB}).
 \end{align*} 
 
 Here the 3rd isomorphism follows since $\cC^{\times_{\cB}d}\times_{\cB}\cH(w)^{\mathrm{ss}}$ is normal (this follows from that it is smooth in codimension one and has at worst lci singularities; the latter is a consequence that it is both quasi-smooth and classical). The last isomorphism follows from $\cH^0(\Gamma(\mathcal{O}_{\cC}))=\cH^0(\Gamma(\mathcal{O}_{\cH(w)^{\mathrm{ss}}}))=\cH^0 \Gamma(\mathcal{O}_{\cB})$. As $\gamma$ is invertible over $\cB^{\mathrm{ell}}$, and the complement of $\cB^{\mathrm{ell}}$ in $\cB$ is of codimension bigger than or equal to two by (\ref{Bell:codim}), it follows that $\gamma$ is invertible over $\cB$. Therefore $\gamma$ is an isomorphism.    
\end{proof}

\subsection{Proof of Lemma~\ref{lem:sigma}}\label{subsec:proofxy}
\begin{proof}
    We define $\widetilde{\mathcal{H}ecke}$ to be the fiber product
    \begin{align*}
    \xymatrix{
       \widetilde{\mathcal{H}ecke} \ar[r] \ar[d] \ar@{}[dr]|{\square} & \mathcal{H}ecke \ar[d]^-{\ev_1} \\
       \mathcal{H}ecke \ar[r]_-{\ev_1} & \mathcal{H}.
    }
    \end{align*}
    It consists of diagrams in $\Coh^{\heartsuit}(\cC_b)$ for $b\in \mathcal{B}$
    \begin{align}\label{dia:xy}
    \xymatrix{
    0\ar[r] & E' \ar[r] \ar@{=}[d] & E \ar[r] & \mathcal{O}_x \ar[r] & 0 \\
    0\ar[r] & E' \ar[r] & \widetilde{E} \ar[r] & \mathcal{O}_y \ar[r] & 0
    }
    \end{align}
    where $x, y \in \cC_b$, and each horizontal row is exact.

    Let
    \begin{align*}
        \mathcal{Z} \subset \widetilde{\mathcal{H}ecke}
    \end{align*}
    be the classical closed substack consisting of \eqref{dia:xy} such that $y=\sigma(x)$. It is enough to show that each irreducible component $\mathcal{Z}' \subset \mathcal{Z}$ is dominant over $\mathcal{B}$.

Recall that $\cHecke$ is classical, see Remark~\ref{rmk:Heckeflat}. 
    By the argument in the proof of Lemma~\ref{lem:Delta0} and the duality argument in Subsection~\ref{subsec:modify}, the fibers of $\ev_1$ are also classical of dimension one, hence $\ev_1$ is flat. Therefore $\widetilde{\mathcal{H}ecke}$ is classical, and the map
    \[
        \widetilde{\mathcal{H}ecke} \to \mathcal{H}
    \]
    is flat of relative dimension two.

    Over the regular locus $\mathcal{H}^{\mathrm{reg}} \subset \mathcal{H}$, the fibers of the map $\mathcal{Z} \to \mathcal{H}$ consist of $(x,y)\in \mathcal{C}_b \times \mathcal{C}_b$ such that $y=\sigma(x)$; therefore each fiber is isomorphic to $\cC_b$. The above argument shows that $\mathcal{Z}$ is isomorphic to $\mathcal{C}\times_{\mathcal{B}} \mathcal{H}$ over $\mathcal{H}^{\mathrm{reg}}$, whose irreducible components are dominant over $\cB$. Therefore, if an irreducible component $\mathcal{Z}'$ is not dominant over $\cB$, then it lies in the complement of $\mathcal{H}^{\mathrm{reg}}$ in $\mathcal{H}$. Using Lemma~\ref{lem:codimension}, this implies that
    \begin{align*}
        \dim \mathcal{Z}' \leq \dim(\mathcal{H} \setminus \mathcal{H}^{\mathrm{reg}})+2 \leq \dim \mathcal{H},
    \end{align*}
    namely $\mathcal{Z}' \subset \widetilde{\mathcal{H}ecke}$ has codimension at least two.

    However, $\mathcal{Z}'$ is an irreducible component of the classical stack $\widetilde{\mathcal{Z}}$, defined by the (classical) Cartesian square
    \begin{align*}
    \xymatrix{
    \widetilde{\mathcal{Z}} \inclusion \ar[d] & \widetilde{\mathcal{H}ecke} \ar[d] \\
    C \inclusion^-{\Delta} & C\times C
    }
    \end{align*}
    where the right vertical arrow sends \eqref{dia:xy} to $(\pi(x),\pi(y))$. Namely, $\widetilde{\mathcal{Z}}=\mathcal{Z}_1 \cup \mathcal{Z}_2$ where each $\mathcal{Z}_i$ is a union of irreducible components, $\mathcal{Z}_1$ corresponds to \eqref{dia:xy} such that $y=x$, resp.\ 
    $\mathcal{Z}_2$ corresponds to \eqref{dia:xy} such that $y=\sigma(x)$, 
    and $\mathcal{Z}'$ is one of the irreducible components of $\mathcal{Z}_2$. The closed substack $\widetilde{\mathcal{Z}}$ has codimension at most one in $\widetilde{\cHecke}$ as it is a pull-back of the Cartier divisor $\Delta(C) \subset C\times C$, which is a contradiction.
\end{proof}

\subsection{Quasi-BPS categories as semiorthogonal summands}
We used the following lemma in the proof of Proposition~\ref{prop:induceT}.
We keep the notation in Proposition~\ref{prop:induceT}. 

\begin{lemma}\label{lem:sodsum}
 For 
any quasi-compact open substack 
$\cH_T(\chi)^{\mathrm{qc}} \subset \cH_T(\chi)$ which admits a finite Harder--Narasimhan 
stratification with an open immersion $j_T^{\mathrm{qc}} \colon \cH_T(\chi)^{\mathrm{ss}}\hookrightarrow \cH_T(\chi)^{\mathrm{qc}}$, the image
of 
\begin{align}\notag
 \LL(\cH_{T}(\chi)^{\mathrm{ss}})_{w'} \stackrel{(j_T^{\mathrm{qc}})_!}{\to} \LL(\cH_{T}(\chi)^{\mathrm{qc}})_{w'}
 \subset \Coh(\cH_{T}(\chi)^{\mathrm{qc}})
\end{align}
is a semiorthogonal summand of $\Coh(\cH_{T}(\chi)^{\mathrm{qc}})$.
\end{lemma}
\begin{proof}
There is a semiorthogonal 
decomposition of the form 
\begin{align}\label{sod:W}
    \Coh(\cH_{T}(\chi)^{\mathrm{qc}})=\langle \ldots, \mathcal{W}, \ldots \rangle, \ (j_{T}^{\mathrm{qc}})^* \colon \mathcal{W} \stackrel{\sim}{\to} \Coh(\cH_{T}(\chi)^{\mathrm{ss}})
\end{align}
by~\cite[Theorem~3.3.1]{HalpK32}, which depends on a choice of $k_{\mathcal{Z}}\in \mathbb{R}$
at each center of HN stratum $\mathcal{Z}$, and $\LL(\cH_{T}(\chi)^{\mathrm{ss}})_{w'}$
is a semiorthogonal summand of $\Coh(\cH_{T}(\chi)^{\mathrm{ss}})$, see~\cite[Theorem~1.1]{PThiggs}. 

From the construction of $(j_T^{\mathrm{qc}})_!$, the image $\mathrm{Im}((j_T^{\mathrm{qc}})_!)$
coincides with the image of the inclusion 
\begin{align*}\LL(\cH_{T}(\chi)^{\mathrm{ss}})_{w'} \subset \Coh(\cH_{T}(\chi)^{\mathrm{qc}})
\end{align*}
determined by (\ref{sod:W}) for a choice 
$k_{\mathcal{Z}}=-c_1(\nu^* \det \mathbb{L}_{\cH_{T}(\chi)}^{>0})/2$, where $\nu$ is a map $\nu \colon \bgm \to \mathcal{Z}\to 
\cH_{T}(\chi)$ as in Remark~\ref{rmk:jshrink}.
Therefore the image of $(j_T^{\mathrm{qc}})_!$ is 
a semiorthogonal summand of 
$\Coh(\cH_{T}(\chi)^{\mathrm{qc}})$. 
\end{proof}

\newcommand{\etalchar}[1]{$^{#1}$}
\providecommand{\bysame}{\leavevmode\hbox to3em{\hrulefill}\thinspace}
\providecommand{\MR}{\relax\ifhmode\unskip\space\fi MR }
\providecommand{\MRhref}[2]{%
	\href{http://www.ams.org/mathscinet-getitem?mr=#1}{#2}
}
\providecommand{\href}[2]{#2}


\begin{thebibliography}{PTVV13}
	
	\bibitem[ABC{\etalchar{+}}a]{GLC2}
	D.~Arinkin, D.~Beraldo, J.~Campbell, L.~Chen, J.~Faergeman, D.~Gaitsgory, K.~Lin, S.~Raskin, and N.~Rozenblyum, \emph{Proof of the geometric {L}anglands conjecture {II}: {K}ac-{M}oody localization and the {FLE}}, arXiv:2405.03648.
	
	\bibitem[ABC{\etalchar{+}}b]{GLC4}
	D.~Arinkin, D.~Beraldo, L.~Chen, J.~Faergeman, D.~Gaitsgory, K.~Lin, S.~Raskin, and N.~Rozenblyum, \emph{Proof of the geometric {L}anglands conjecture {IV}: ambidexterity}, arXiv:2409.08670.
	
	\bibitem[AG15]{AG}
	D.~Arinkin and D.~Gaitsgory, \emph{Singular support of coherent sheaves and the geometric {L}anglands conjecture}, Selecta Math. (N.S.) \textbf{21} (2015), no.~1, 1--199.
	
	\bibitem[AHPS]{ahlqvist2023good}
	E.~Ahlqvist, J.~Hekking, M.~Pernice, and M.~Savvas, \emph{Good {M}oduli {S}paces in {D}erived {A}lgebraic {G}eometry}, arXiv:2309.16574.
	
	\bibitem[Alp13]{MR3237451}
	J.~Alper, \emph{Good moduli spaces for {A}rtin stacks}, Ann. Inst. Fourier (Grenoble) \textbf{63} (2013), no.~6, 2349--2402.
	
	\bibitem[Ari13]{Ardual}
	D.~Arinkin, \emph{Autoduality of compactified {J}acobians for curves with plane singularities}, J. Algebraic Geom. \textbf{22} (2013), no.~2, 363--388.
	
	\bibitem[BD]{BD0}
	A.~Beilinson and V.~Drinfeld, \emph{Quantization of {H}itchin’s integrable system and {H}ecke eigensheaves}, available at http://people.math.harvard.edu/gaitsgde/grad 2009/.
	
	\bibitem[BNR89]{BeNaRa}
	A.~Beauville, M.S. Narasimhan, and S.~Ramanan, \emph{Spectral curves and the generalised theta divisor}, J.~Reine Angew.~Math.~ \textbf{398} (1989), 169--179.
	
	\bibitem[BvdB03]{Bonvan}
	A.~Bondal and M.~van~den Bergh, \emph{Generators and representability of functors in commutative and noncommutative geometry}, Mosc. Math. J. \textbf{3} (2003), no.~1, 1--36, 258.
	
	\bibitem[BZN18]{ZN}
	D.~Ben-Zvi and D.~Nadler, \emph{Betti geometric {L}anglands}, Algebraic geometry: {S}alt {L}ake {C}ity 2015, Proc. Sympos. Pure Math., vol. 97.2, Amer. Math. Soc., Providence, RI, 2018, pp.~3--41.
	
	\bibitem[CCF{\etalchar{+}}]{GLC3}
	J.~Campbell, L.~Chen, J.~Faergeman, D.~Gaitsgory, K.~Lin, S.~Raskin, and N.~Rozenblyum, \emph{Proof of the geometric {L}anglands conjecture {III}: compatibility with parabolic induction}, arXiv:2409.07051.
	
	\bibitem[CW]{CW}
	S.~Cautis and H.~Williams, \emph{Canonical bases for {C}oulomb branches of 4d $n = 2$ gauge theories}, arXiv:2306.03023.
	
	\bibitem[DG13]{MR3037900}
	V.~Drinfeld and D.~Gaitsgory, \emph{On some finiteness questions for algebraic stacks}, Geom. Funct. Anal. \textbf{23} (2013), no.~1, 149--294.
	
	\bibitem[DP12]{DoPa}
	R.~Donagi and T.~Pantev, \emph{Langlands duality for {H}itchin systems}, Invent. Math. \textbf{189} (2012), no.~3, 653--735.
	
	\bibitem[Fal93]{Faltings1993}
	G.~Faltings, \emph{Stable {G}-bundles and projective connections}, Journal of Algebraic Geometry \textbf{2} (1993), no.~3, 507--568.
	
	\bibitem[Gai16]{GaiFun}
	D.~Gaitsgory, \emph{Functors given by kernels, adjunctions and duality}, J. Algebraic Geom. \textbf{25} (2016), no.~3, 461--548.
	
	\bibitem[Gin01]{Ginzburg2001}
	V.~Ginzburg, \emph{The global nilpotent variety is {L}agrangian}, Duke Mathematical Journal \textbf{109} (2001), no.~3, 511--519.
	
	\bibitem[GLS07]{GLS}
	G-M. Greuel, C.~Lossen, and E.~Shustin, \emph{Introduction to singularities and deformations}, Springer Monographs in Mathematics, Springer, Berlin, 2007.
	
	\bibitem[GRa]{GLC1}
	D.~Gaitsgory and S.~Raskin, \emph{Proof of the geometric {L}anglands conjecture {I}: construction of the functor}, arXiv:2405.03599.
	
	\bibitem[GRb]{GLC5}
	\bysame, \emph{Proof of the geometric {L}anglands conjecture {V}: the multiplicity one theorem}, arXiv:2409.09856.
	
	\bibitem[GR17]{MR3701352}
	D.~Gaitsgory and N.~Rozenblyum, \emph{A study in derived algebraic geometry. {V}ol. {I}. {C}orrespondences and duality}, Mathematical Surveys and Monographs, vol. 221, American Mathematical Society, Providence, RI, 2017.
	
	\bibitem[Gro13]{Gdual}
	M.~Groechenig, \emph{Autoduality of the {H}itchin system and the geometric {L}anglands programme}, Dphil thesis, University of Oxford, 2013.
	
	\bibitem[Hai01]{Ha}
	M.~Haiman, \emph{Hilbert schemes, polygraphs and the {M}acdonald positivity conjecture}, J. Amer. Math. Soc. \textbf{14} (2001), no.~4, 941--1006.
	
	\bibitem[Hau22]{HauICM}
	T.~Hausel, \emph{Enhanced mirror symmetry for {L}anglands dual {H}itchin systems}, I{CM}---{I}nternational {C}ongress of {M}athematicians. {V}ol. 3. {S}ections 1--4, EMS Press, Berlin, 2022, pp.~2228--2249.
	
	\bibitem[Hit87a]{Hitchin1987}
	N.~Hitchin, \emph{The self-duality equations on a {R}iemann surface}, Proceedings of the London Mathematical Society \textbf{55} (1987), no.~1, 59--126.
	
	\bibitem[Hit87b]{Hitchin}
	\bysame, \emph{Stable bundles and integrable systems}, Duke Math. J. \textbf{54} (1987), no.~1, 91--114.
	
	\bibitem[HLa]{HalpK32}
	D.~Halpern-Leistner, \emph{Derived {$\Theta$}-stratifications and the {$D$}-equivalence conjecture}, Ann. of Math. (2), to appear, arXiv:2010.01127.
	
	\bibitem[HLb]{Halpinstab}
	\bysame, \emph{On the structure of instability in moduli theory}, arXiv:1411.0627.
	
	\bibitem[HL97]{Hu}
	D.~Huybrechts and M.~Lehn, \emph{Geometry of moduli spaces of sheaves}, Aspects in Mathematics, vol. E31, Vieweg, 1997.
	
	\bibitem[HLS20]{hls}
	D.~Halpern-Leistner and S.~V. Sam, \emph{Combinatorial constructions of derived equivalences}, J. Amer. Math. Soc. \textbf{33} (2020), no.~3, 735--773.
	
	\bibitem[HT03]{HauTha}
	T.~Hausel and M.~Thaddeus, \emph{Mirror symmetry, {L}anglands duality, and the {H}itchin system}, Invent. Math. \textbf{153} (2003), no.~1, 197--229.
	
	\bibitem[KP95]{KouvidakisPantev1995}
	A.~Kouvidakis and T.~Pantev, \emph{The automorphism group of the moduli space of semi stable bundles}, Mathematische Annalen \textbf{302} (1995), 225--268.
	
	\bibitem[KW07]{KapWit}
	A.~N. Kapustin and E.~Witten, \emph{Electric-{M}agnetic {D}uality and the {G}eometric {L}anglands {P}rogram}, Communications in {N}umber {T}heory and {P}hysics \textbf{1} (2007), no.~1, 1--236.
	
	\bibitem[Li21]{MLi}
	M.~Li, \emph{Construction of {P}oincar\'{e} sheaf on stack of rank 2 {H}iggs bundles}, Adv. Math. \textbf{376} (2021), Paper No. 107437, 55.
	
	\bibitem[Lia]{LiangSatakeTypeA}
	S.~Liang, \emph{A {C}oherent {V}ersion of {G}eometric {S}atake {E}quivalence for type {A}}, arXiv:2601.07390.
	
	\bibitem[Lie07]{MR2309155}
	M.~Lieblich, \emph{Moduli of twisted sheaves}, Duke Math. J. \textbf{138} (2007), no.~1, 23--118.
	
	\bibitem[MRV17]{MRF3}
	M.~Melo, A.~Rapagnetta, and F.~Viviani, \emph{Fine compactified {J}acobians of reduced curves}, Trans. Amer. Math. Soc. \textbf{369} (2017), no.~8, 5341--5402.
	
	\bibitem[MRV19]{MRVF2}
	\bysame, \emph{Fourier-{M}ukai and autoduality for compactified {J}acobians. {II}}, Geometry and Topology \textbf{23} (2019), 2335--2395.
	
	\bibitem[Muk81]{Mu1}
	S.~Mukai, \emph{Duality between ${D}({X})$ and ${D}(\hat{X})$ with its application to picard sheaves}, Nagoya Math.~J.~ \textbf{81} (1981), 101--116.
	
	\bibitem[PS23]{PoSa}
	M.~Porta and F.~Sala, \emph{Two dimensional categorified {H}all algebras}, J. Eur. Math. Soc. \textbf{25} (2023), no.~3, 1113--1205.
	
	\bibitem[PTa]{PThiggs}
	T.~P{\u{a}}durariu and Y.~Toda, \emph{Quasi-{BPS} categories for {H}iggs bundles}, arXiv:2408.02168.
	
	\bibitem[PTb]{PTlim}
	T.~P\u{a}durariu and Y.~Toda, \emph{The {D}olbeault geometric {L}anglands conjecture via limit categories}, arXiv:2508.19624.
	
	\bibitem[PT09]{MR2545686}
	R.~Pandharipande and R.~P. Thomas, \emph{Curve counting via stable pairs in the derived category}, Invent. Math. \textbf{178} (2009), no.~2, 407--447.
	
	\bibitem[PT10]{PT3}
	R.~Pandharipande and R.~P. Thomas, \emph{Stable pairs and {BPS} invariants}, J.~Amer.~Math.~Soc.~ \textbf{23} (2010), 267--297.
	
	\bibitem[PT25]{PThiggs2}
	T.~P{\u{a}}durariu and Y.~Toda, \emph{Topological {K}-theory of quasi-{BPS} categories for {H}iggs bundles}, J. Topol. \textbf{18} (2025), no.~4, Paper No. e70049, 72.
	
	\bibitem[PT26]{PTK3}
	T.~P\u{a}durariu and Y.~Toda, \emph{Quasi-{BPS} categories for {K}3 surfaces}, Forum of Mathematics, Pi \textbf{14} (2026), e11, 69pp.
	
	\bibitem[PTVV13]{PTVV}
	T.~Pantev, B.~To$\ddot{\textrm{e}}$n, M.~Vaquie, and G.~Vezzosi, \emph{Shifted symplectic structures}, Publ.~Math.~IHES \textbf{117} (2013), 271--328.
	
	\bibitem[{\v S}dB17]{SvdB}
	{\v S}.~{\v S}penko and M.~Van den Bergh, \emph{Non-commutative resolutions of quotient singularities for reductive groups}, Invent. Math. \textbf{210} (2017), no.~1, 3--67.
	
	\bibitem[Tod24]{T}
	Y.~Toda, \emph{Categorical {D}onaldson-{T}homas theory for local surfaces}, Lecture Notes in Mathematics \textbf{2350} (2024), 1--309.
	
	\bibitem[To{\"e}]{Toenpush}
	B.~To{\"e}n, \emph{Proper local complete intersection morphisms preserve perfect complexes}, arXiv:1210.2827.
	
	\bibitem[Viv]{Viviani}
	F.~Viviani, \emph{On the classification of fine compactified {J}acobians of nodal curves}, arXiv:2310.20317v6.
	
\end{thebibliography}

\vspace{5mm}

\textsc{\small Yukinobu Toda: Kavli Institute for the Physics and Mathematics of the Universe (WPI), University of Tokyo, 5-1-5 Kashiwanoha, Kashiwa, 277-8583, Japan.}\\
\textsc{\small 
Inamori Research Institute for Science, 620 Suiginya-cho, Shimogyo-ku, Kyoto 600-8411, Japan.} \\
\textit{\small E-mail address:} \texttt{\small yukinobu.toda@ipmu.jp}\\

\end{document}